\documentclass[11pt]{article}

\usepackage[english]{babel} 
\usepackage[utf8]{inputenc} 
\usepackage[T1]{fontenc}
\usepackage{amsmath,amssymb,fullpage}
\usepackage[all,dvips]{xy}

\usepackage{tikz}
\usepackage{tikz,fullpage}
\usetikzlibrary{positioning, arrows.meta}
\usetikzlibrary{matrix}

\usepackage{enumerate}
\usepackage{mathtools}

\usepackage{graphicx}
\usepackage{color}
\definecolor{vert}{rgb}{0.02,0.4,0.10}
\usepackage{hyperref}
\hypersetup{
	colorlinks=true,                
	breaklinks=true,                
	linkcolor= vert,                
	citecolor= blue                
}

\usepackage{graphicx}
\graphicspath{ {images/} }
\usepackage{pgf,tikz,pgfplots}
\pgfplotsset{compat=1.14}
\usepackage{mathrsfs}
\usetikzlibrary{arrows}

\usepackage{makeidx}
\usepackage{pstricks}

\providecommand{\otherindexspace}[1]{}

\usepackage{authblk}
\usepackage{amsthm}

\newtheorem{theorem}{Theorem}[section]

\newtheorem{lemma}[theorem]{Lemma}
\newtheorem{proposition}[theorem]{Proposition}
\newtheorem{remark}[theorem]{Remark}

\newtheorem{corollary}[theorem]{Corollary}

\newtheorem{conjecture}[theorem]{Conjecture}

\numberwithin{equation}{section}

\def\vp{\varepsilon}

\def\cal#1{\mathcal{#1}}

\def \R{\mathbb {R}}
\def \N{\mathbb{N}}
\def \Z{\mathbb{Z}}

\def \U{\mathbb{U}}
\def \T{\mathcal{T}}

\def\E{\mathbb{E}}
\def\P{\mathbb{P}}
\def\lb{[\![}
\def\rb{]\!]}

\usepackage{fancyhdr}

\usepackage{graphics}
\usepackage{graphicx}

\DeclareGraphicsExtensions{.eps,.bmp,.jpg,.pdf,.mps,.png,.gif}

\makeatletter
\def\titre{\@title}
\makeatother

\title{Scaling limit of first passage percolation geodesics on planar maps
}
\author{Emmanuel Kammerer \thanks{CMAP, \'Ecole polytechnique, Route de Saclay, 91120 Palaiseau, France, \textsf{emmanuel.kammerer@polytechnique.edu}
}}
\date{\today}
\begin{document}

\maketitle

\begin{abstract}
	We establish the scaling limit of the geodesics to the root for the first passage percolation distance on random planar maps. We first describe the scaling limit of the number of faces along the geodesics. This result enables {us} to compare the metric balls for the first passage percolation and the dual graph distance. It also enables {us} to {give an upper bound for} the diameter of large random maps. Then, we describe the scaling limit of the tree of first passage percolation geodesics to the root via a stochastic coalescing flow of pure jump diffusions. {Using} this stochastic flow, {we also} construct some random metric spaces which we conjecture to be the scaling limits of random planar maps with high degrees. The main tool in this work is a time-reversal of the uniform peeling exploration.
\end{abstract}

\begin{figure}[h]
	\centering
	\includegraphics[scale=0.86]{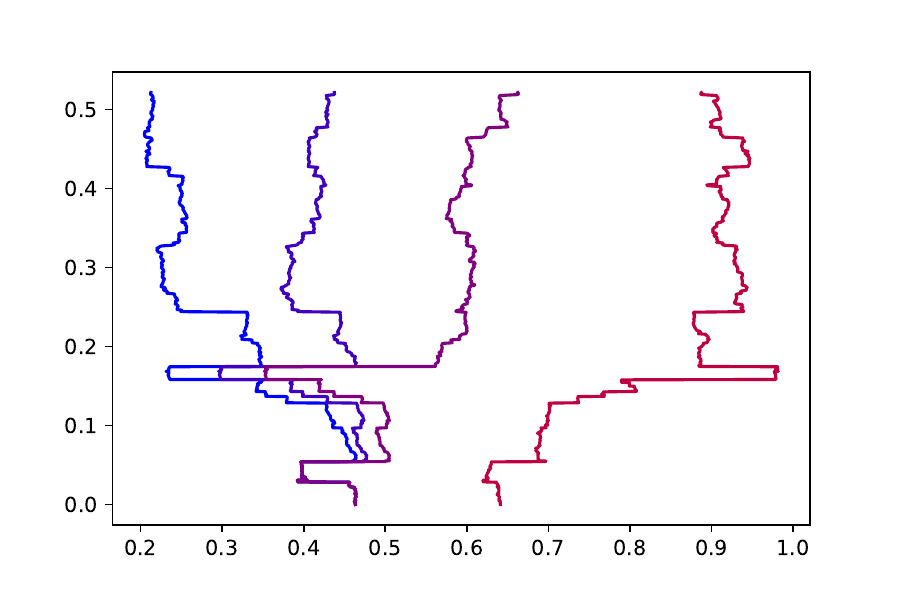}
	\caption{Simulation of four trajectories of the stochastic coalescing flow of diffusions with jumps which arises as the scaling limit of first-passage percolation geodesics.}
	\label{tronc et explo unif}
\end{figure}

\tableofcontents

\section{Introduction}
\subsection{Random Boltzmann maps}
We study the first-passage percolation distance on random Boltzmann planar maps of type $a \in (3/2,5/2]$ which correspond to the dual of the stable maps of \cite{LGM11} when $a \in (3/2, 5/2)$ and to random planar maps in the universality class of the Brownian sphere when $a=5/2$. We give applications to the (dual) graph distance and establish the scaling limit of geodesics to the root.

A planar map is a connected planar graph which is embedded in the sphere and seen up to orientation-preserving transformations. Our planar maps are equipped with a distinguished oriented edge which is called the root edge. The face $f_r$ which is on the right of the root edge is called the root face. We restrict our attention to bipartite planar maps, i.e.\@ such that each face has an even degree. We denote by $\mathcal{M}$ the set of finite bipartite planar maps and for all $\ell \ge1$ we denote by $\mathcal{M}^{(\ell)}$ the set of bipartite planar maps whose root face has degree $2\ell$. See Figure \ref{fig: carte et arbre geodesiques} for an example.

A general way to pick a map at random consists in assigning positive weights to each face according to its degree. Following \cite{MM07}, let ${\bf q} = (q_k)_{k\ge 1}$ be a non-zero sequence of non-negative real numbers. The weight of a map $\mathfrak{m} \in \mathcal{M}$ is defined by
$$
w_{\bf q}(\mathfrak{m}) = \prod_{f \in \mathrm{Faces}(\mathfrak{m}) \setminus \{f_r\}} q_{\deg (f)/2}.
$$
We then define, {for all $\ell \ge 1$,} the partition function 
$$
W^{(\ell)} = \sum_{\mathfrak{m} \in \mathcal{M}^{(\ell)}} w_{\bf q} (\mathfrak{m}).
$$
When the partition function is finite, the sequence $\bf q$ is called admissible. When ${\bf q}$ is admissible, we define the Boltzmann probability measure $\P^{(\ell)}$ on $\mathcal{M}^{(\ell)}$ by setting
$$
\forall \mathfrak{m} \in \mathcal{M}^{(\ell)}, \qquad \P^{(\ell)}(\{\mathfrak{m}\}) = 
\frac{w_{\bf q} (\mathfrak{m})}{W^{(\ell)}}.
$$
A random map of law $\P^{(\ell)}$ is called a random Boltzmann map of perimeter $2\ell$. {
The geometry of random Boltzmann maps heavily depends on the asymptotic behaviour of the partition function $W^{(\ell)}$ as $\ell \to \infty$. It is known that there exists a constant $c_{\bf q}>0$ such that
$$
W^{(\ell)} = O(c_{\bf q}^\ell \ell^{-3/2}) \qquad \text{and} \qquad W^{(\ell)} = \Omega(c_{\bf q}^\ell \ell^{-5/2})
$$
as $\ell \to \infty$. See Section 5.3 of \cite{StFlour}. As in the lecture notes \cite{StFlour}, we say that the weight sequence $\bf q$ is subcritical (or of type $a=3/2$) if there exists a constant $p_{\bf q} > 0$ such that
$$
W^{(\ell)} \mathop{\sim}\limits_{\ell \to \infty} \frac{p_{\bf q}}{2} c_{\bf q}^{\ell+1} \ell^{-3/2},
$$
while the sequence $\bf q$ is said to be critical generic (or of type $a=5/2$) if there exists $p_{\bf q} > 0$ such that
$$
W^{(\ell)} \mathop{\sim}\limits_{\ell \to \infty} \frac{p_{\bf q}}{2} c_{\bf q}^{\ell+1} \ell^{-5/2}.
$$
When $q_k = 0$ for $k$ large enough, i.e.\@ when the face degrees are bounded, and $\bf q$ is admissible, then either $\bf q$ is subcritical, or it is generic critical. To obtain intermediate exponents, one has to favor the appearance of high degrees by choosing $\bf q$ so that the weights decay ``slowly enough''.

We say that $\bf q$ is critical non-generic of type $a$ with $3/2 < a < 5/2$ if there exists $p_{\bf q} > 0$ such that
\begin{equation}\label{eq asymptotique W}
W^{(\ell)} \mathop{\sim}\limits_{\ell \to \infty} \frac{p_{\bf q}}{2} c_{\bf q}^{\ell+1} \ell^{-a}.
\end{equation}
A random Boltzmann map associated with a critical non-generic weight sequence $\bf q$ of type $a$ will also be called a stable map with parameter $a-1/2$.} Examples of such sequences $\bf q$ are exhibited in Lemma 6.1 of \cite{BC}. 

{The study of random Boltzmann maps with $\bf q$ being critical non-generic is motivated by strong connections with random maps coupled with models of statistical mechanics. More precisely, if one considers a random map coupled with a model of statistical mechanics, then, at criticality, the cluster of the root face, which is also called the gasket, is a random Boltzmann map with a critical non-generic weight sequence $\bf q$. See e.g.\@ \cite{BCM19} for Bernoulli percolation, \cite{AM22} for the Ising model and \cite{BBG12} for the O(n) loop model. Random planar maps coupled with models of statistical mechanics were first studied by physicists as models of statistical mechanics in quantum gravity. See e.g.\@ \cite{EK95} and the references therein for the O(n) loop model. Scaling limits of random maps coupled with models of statistical mechanics are conjectured to be described by Liouville quantum gravity, first introduced in physics in \cite{Pol81}, and by an independent conformal loop ensemble, defined in \cite{S09,SW12}. See e.g.\@ Conjecture 2.1 in \cite{HL24} for a precise conjecture in the case the O(n) loop model and see \cite{GHS21} for a result solving the analogous problem in the case of Bernoulli percolation.  
	
\textbf{In what follows, we assume that \eqref{eq asymptotique W} holds for some $a \in (3/2,5/2]$, i.e.\@ that $\bf q$ is either critical generic or critical non-generic.}}

If one conditions a random map of law $\P^{(\ell)}$ to have $n$ vertices and let{s} $n\to \infty$, {then the law of the ball of fixed radius $R$ centered at the root face $f_r$ converges for all $R>0$. This enables {us} to define an infinite random map called the local limit, whose law is given by} Theorem 7.1 of \cite{StFlour} and is denoted by $\P^{(\ell)}_\infty$. A random map of law $\P^{(\ell)}_\infty$ is called an infinite random Boltzmann map of perimeter $2\ell$. 

{When $\mathfrak{m}$ is a (possibly infinite) map, we denote by $\mathfrak{m}^\dagger$ the dual map associated with $\mathfrak{m}$} which is obtained by switching the roles of the faces and vertices and saying that two vertices of $\mathfrak{m}^\dagger$ are connected by an edge if the corresponding faces of $\mathfrak{m}$ are adjacent. {Different distances can then be studied, notably:
	\begin{itemize}
		\item The (primal) graph distance $d_\mathrm{gr}$ on the vertices of $\mathfrak{m}$, obtained by assigning a length $1$ to each edge;
		\item The dual graph distance $d_\mathrm{gr}^\dagger$ on the faces of $\mathfrak{m}$, which are the vertices of $\mathfrak{m}^\dagger$, obtained by assigning a length $1$ to each dual edge. This distance corresponds to the graph distance on the dual map $\mathfrak{m}^\dagger$;
		\item The first-passage percolation (fpp) distance $d^\dagger_{\mathrm{fpp}}$ on the faces of $\mathfrak{m}$, obtained by assigning i.i.d.\@ exponential random edge lengths of parameter $1$ to the dual edges.
\end{itemize}}
{The primal graph distance is now well understood in the generic case $a=5/2$ since the works of Le Gall \cite{LG13} and Miermont \cite{Mie13}. The scaling limit of random maps of law $\P^{(\ell)}$ in the case $a=5/2$ as $\ell \to \infty$ is given in Theorem 8 of \cite{BM17}: if one divides the distance $d_\mathrm{gr}$ by $\sqrt{\ell}$, then the random metric space converges in distribution towards the so-called free Brownian disk which was initially introduced in \cite{Bett15}. It is conjectured that in the generic case $a=5/2$ the dual distances $d^\dagger_\mathrm{gr}$ and $d^\dagger_\mathrm{fpp}$ behave in the same way as the primal distance $d_\mathrm{gr}$ since the face degrees are not too large. In the case of triangulations, the question is solved and the distances $d_\mathrm{gr}$, $d^\dagger_\mathrm{gr}$ and $d^\dagger_\mathrm{fpp}$ are actually very close to each other up to a multiplicative constant as shown in \cite{CLG19}. In the critical non-generic case $a \in (3/2,5/2)$ the scaling limit of random maps of law $\P^{(1)}$ conditioned to have $n$ vertices equipped with de distance $d_\mathrm{gr}$ as $n \to \infty$ is determined in \cite{CMR25} and the scaling limit of random maps of law $\P^{(\ell)}$ should be obtained in the same way. Yet, when $a \in (3/2,5/2)$, the behaviour of dual distances $d^\dagger_\mathrm{gr}$ and $d^\dagger_\mathrm{fpp}$ is drastically different from the primal distance $d_\mathrm{gr}$ insofar as the appearance of large faces leads to high degree vertices playing the role of hubs in the dual map. The geometry of large random maps of law $\P^{(\ell)}$ and infinite random maps of law $\P^{(\ell)}_\infty$ for the dual distances  $d^\dagger_\mathrm{gr}$ and $d^\dagger_\mathrm{fpp}$ is only partially understood. }

The geometry of infinite random maps {equipped with the dual distances} under $\P^{(\ell)}_\infty$ was first studied in the case of the uniform infinite planar quadrangulation (which is a particular case of the case $a=5/2$) in \cite{CLG17} and \cite{CC18}, then in \cite{BC} in the cases $a\in (3/2,5/2)\setminus\{2\}$ and in \cite{BCM} in the critical case $a=2$.  {The geometry of the map satisfies a phase transition at $a=2$ in the sense that the growth of the balls centered at the root is polynomial in the dilute phase $a\in (2,5/2)$, exponential in the dense phase $a \in (3/2,2)$ and intermediate in the case $a=2$.} The scaling limit of the distances $d^\dagger_\mathrm{gr}$ to the root under $\P^{(\ell)}$ or $\P^{(\ell)}_\infty$ as $\ell \to \infty$ was established using the framework of growth-fragmentations in \cite{BCK} in the case of triangulations (corresponding informally to the case $a=5/2$), then in \cite{BBCK} in the case $a \in (2,5/2)$ and in \cite{Kam23CLE} in the case $a=2$. {More precisely, it is known that in a map of law $\P^{(\ell)}$, the distances $d^\dagger_\mathrm{gr}$ from $f_r$ to the other faces, once divided by $\ell^{a-2}$ in the case $a\in (2,5/2]$ and by $\log \ell$ in the case $a=2$ converge in distribution. In the case $a \in (2,5/2]$, it is conjectured that random maps of law $\P^{(\ell)}$ equipped with $d^\dagger_\mathrm{gr}$ or $d^\dagger_\mathrm{fpp}$ satisfy a scaling limit towards a random compact metric space. In the case $a=2$ it is conjectured in \cite{Kam23CLE} that they satisfy a scaling limit towards a non-compact random metric space related to the conformal loop ensemble of parameter $\kappa=4$ on an independent critical Liouville quantum gravity. When $a \in (3/2,2)$ random planar maps of law $\P^{(\ell)}$ equipped with $d^\dagger_\mathrm{gr}$ are not expected to satisfy a scaling limit when $\ell \to \infty$.}

{The study of the fpp distance $d^\dagger_\mathrm{fpp}$ proves to be simpler than the study of the dual graph distance $d^\dagger_\mathrm{gr}$ and gives the same scaling limit in the cases $a \in [2,5/2]$ (with the disappearance of the $\log \ell$ factor in the case $a=2$). Actually, the study of the fpp distance can be seen as a first step which turns out to be useful for the study of the dual graph distance, see e.g.\@ \cite{BC, BCM, Kam23, Kam23CLE}. In this work, we thus focus on the fpp distance $d^\dagger_{\mathrm{fpp}}$ on the dual map $\mathfrak{m}^\dagger$ under $\P^{(\ell)}$ and $\P^{(\ell)}_\infty$ and we go further than the distances to the root in the study of the geometry of the map.} For {all} faces $f,f'$ of $\mathfrak{m}$, the shortest path for the fpp distance from $f$ to $f'$ is unique since the exponential distribution has no atom{s}. It does not necessarily correspond to a shortest path for the dual graph distance (see e.g.\@ Figure \ref{fig: carte et arbre geodesiques}). The main focus of this work is the study of the shortest paths to the root face $f_r$ and their scaling limits. These shortest paths actually form a tree which is denoted by $\T(\mathfrak{m})$. This study also has applications to the dual graph distance $d^\dagger_{\mathrm{gr}}$, which corresponds to the graph distance on $\mathfrak{m}^\dagger$.

\begin{figure}[h]
	\centering
	\includegraphics[scale=1.7]{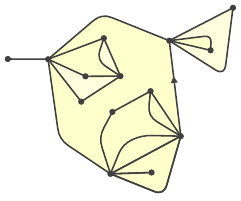}
	\includegraphics[scale=1.7]{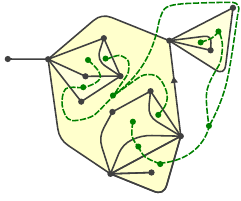}
	\caption{Left: a planar map $\mathfrak{m} \in \mathcal{M}^{(4)}$. The root edge is marked with an arrow and the root face corresponds to the outer face. Right: illustration for the tree $\T(\mathfrak{m})$ in dashed green lines of fpp geodesics to the root face $f_r$. Note that $\T(\mathfrak{m})$ is not a deterministic function of $\mathfrak{m}$ but depends on the random exponential edge lengths.}
	\label{fig: carte et arbre geodesiques}
\end{figure}

An important tool in the study of distances on random planar maps is the peeling exploration. See Subsection \ref{sous-section explo uniforme} for more details, as well as the lecture notes \cite{StFlour} or \cite{BuddPeeling} for a detailed presentation of this technique. It consists in a step by step exploration of the map $\mathfrak{m}$ starting from the root face, where at each step we discover what is behind an edge on the boundary of the unexplored region. This edge is called the peeled edge. The explored region at step $n\ge0$ is denoted by $\overline{\mathfrak{e}}_n$. We denote by $(P(n))_{n\ge 0}$ the perimeter process which is defined as half of the number of edges on the boundary of the explored region $\overline{\mathfrak{e}}_n$ at time $n\ge 0$. Under $\P_\infty^{(\ell)}$, the perimeter process is a Markov chain which is a random walk on $\Z$ of step distribution $\nu$ conditioned to stay positive, whose scaling limit is described by a stable Lévy process conditioned to stay positive $\Upsilon_a^\uparrow$. See Subsection \ref{sous-section loi de l'explo} for the definition of $\nu$ and Subsection \ref{sous-section preliminaires marche conditionnee} for the definition of $\Upsilon_a^\uparrow$. The choice of the peeled edge {is crucial in order} to study different properties of the map. In this work, the peeled edge is chosen uniformly at random on the boundary of $\overline{\mathfrak{e}}_n$ for all $n\ge 0$ so that the explored region $\overline{\mathfrak{e}}_n$ follows the growth of the fpp balls of center $f_r$.

\subsection{Main results}

Our results can be divided into two parts. We first focus on the scaling limit of the number of faces along the geodesics, hence relating the fpp distance with the dual graph distance on the geodesics. This scaling limit enables {us} to compare the growth of fpp and dual graph balls and also to obtain an upper bound for the diameter of large random planar maps for the dual graph distance. Then, we describe the scaling limit of the fpp geodesics to the root in large random planar maps.

\paragraph{Number of faces.} We first give the scaling limit of the number of faces in the fpp geodesics to the root and give applications to the dual graph distance. We choose to state these results for infinite Boltzmann planar maps of law $\P_\infty^{(1)}$. When $a \in (2,5/2]$, as in Definition 11.1 in \cite{StFlour}, we let ${e}_{\bf q} = \sum_{k\ge 0} (2k+1){q_{k+1} c_{\bf q}^k}$ be the so-called mean exposure. Recall the constant $p_{\bf q}$ from \eqref{eq asymptotique W}.
\begin{theorem}\label{th nombre de faces}
	For all $n \ge 0$, let $\Gamma_n$ be the shortest path for $d^\dagger_{\mathrm{fpp}}$ from $f_r$ to some face of $\overline{\mathfrak{e}}_n$ adjacent to the boundary which is chosen arbitrarily. Then we have the following convergences:
	\begin{itemize}
		\item If $a \in (2,5/2]$, then for the topology of uniform convergence on compact subsets,
		$$
		\text{Under } \P^{(1)}_\infty, \qquad
		\left(
		n^{-\frac{a-2}{a-1}}\# \mathrm{Faces}(\Gamma_{\lfloor nt \rfloor}) 
		\right)_{t\ge 0}
		\mathop{\longrightarrow}\limits_{n \to \infty}^{(\mathrm{d})}
		\left(
		\frac{e_{\bf q}}{2}
		\int_0^t \frac{ds}{\Upsilon_a^\uparrow(p_{\bf q} s)}
		\right)_{t\ge 0};
		$$
		\item If $a =2$, then 
		$$
		\text{Under } \P^{(1)}_\infty, \qquad
		(\log n)^{-2} \# \mathrm{Faces}(\Gamma_{ n }) 
		\mathop{\longrightarrow}\limits_{n \to \infty}^{(\mathrm{L}^1)}
		\frac{1}{\pi^2};
		$$
		\item If $a \in (3/2,2)$, then 
		$$
		\text{Under } \P^{(1)}_\infty, \qquad
		(\log n)^{-1} \# \mathrm{Faces}(\Gamma_{ n }) 
		\mathop{\longrightarrow}\limits_{n \to \infty}^{(\mathrm{L}^1)}
		\frac{2}{\pi \tan((2-a) \pi)}.
		$$
	\end{itemize}
\end{theorem}
{One can compare the above result with the study of the growth of balls for the dual graph distance and see that the exact same scaling limits appear in Theorem 4.2 of \cite{BC} in the dilute phase $a \in (2,5/2)$ and in Proposition 4 of \cite{BCM} in the case $a=2$ up to a multiplicative constant.}

{The above observation is not a coincidence:} Theorem \ref{th nombre de faces} has some applications to the dual graph distance $d^\dagger_{\mathrm{gr}}$. Namely, in the cases $a \ge 2$, {we are able to} compare the distances $d^\dagger_{\mathrm{fpp}}$ and $d^\dagger_{\mathrm{gr}}$, improving the results of \cite{CM20} in the case of the uniform exploration. When $\mathfrak{m}$ is an infinite (one-ended) planar map, for all $r \ge 0$, we denote by $\overline{\mathrm{Ball}}^{\mathrm{fpp}}_r(\mathfrak{m})$ the hull of the ball with center $f_r$ of radius $r$ for the fpp distance obtained from the ball of radius $r$ by filling in the holes which contain finite portions of $\mathfrak{m}$. Similarly, for all $r \ge 0$, we denote by $\overline{\mathrm{Ball}}^\dagger_r$ the hull of the ball with center $f_r$ of radius $r$ for the dual graph distance obtained by filling in the holes which contain finite portions of $\mathfrak{m}$.

\begin{corollary}\label{cor inclusion des boules dilue}
	Suppose ${\bf q}$ of type $a \in (2,5/2]$. Under $\P_\infty^{(1)}$, for all $\vp>0$, almost surely, for all $r$ large enough,
	$$
	\overline{\mathrm{Ball}}^{\mathrm{fpp}}_r(\mathfrak{m}) \subset \overline{\mathrm{Ball}}^\dagger_{\lfloor (1+\vp)e_{\bf q}r \rfloor}(\mathfrak{m}).
	$$
\end{corollary}
{The above result improves Theorem 1.1 of \cite{CM20} in the case of the uniform exploration which enables {us} to obtain the same inclusion as above but with a non-explicit constant in place of $(1+\vp)e_{\bf q}$.}
An analogous result holds in the case $a=2$.
\begin{corollary}\label{cor inclusion des boules cas Cauchy}
	Suppose $\bf q$  of type $a=2$. Under $\P_\infty^{(1)}$, for all $\vp>0$, with probability $1-o(1)$ as $r \to \infty$, 
	$$
	\overline{\mathrm{Ball}}^{\mathrm{fpp}}_{r} (\mathfrak{m}) 
	\subset \overline{\mathrm{Ball}}^{\mathrm{\dagger}}_{\lfloor (1+\vp) \pi^2p^2_{\bf q} r^2\rfloor}(\mathfrak{m})  .
	$$
\end{corollary}
A last consequence of the counting of faces along geodesics to the root, which is obtained using the martingale techniques introduced in \cite{Kam23}, is an upper bound for the diameter $\mathrm{Diam}^\dagger_{\mathrm{gr}}(\mathfrak{m})$ of (finite) random maps $\mathfrak{m}$ in the case $a \in (3/2,2)$ when they are equipped with the dual graph distance. {This completes the results of Part II of \cite{BCR24} and Theorem 1.2 of \cite{Kam23} which show that in the respective cases $a\in (2,5/2]$ and $a=2$ the diameter $\mathrm{Diam}^\dagger_{\mathrm{gr}}(\mathfrak{m})$ under the law $\P^{(\ell)}$ is of order respectively $\ell^{2-a}$ and $(\log \ell)^2$ as $\ell \to \infty$.}
\begin{theorem}\label{th diametre cas dense}
	Suppose ${\bf q}$ is of type $a \in (3/2,2)$. Then, there exists a constant $C_{\bf q}>0$ such that under $\P^{(\ell)}$ as $\ell \to \infty$, we have with probability $1-o(1)$
	$$
	\mathrm{Diam}_{\mathrm{gr}}^\dagger
	(\mathfrak{m})
	\le C_{\bf q}  \log \ell.
	$$
\end{theorem}
The main tool for obtaining the above results is a backward exploration of the map, obtained by reversing time in the uniform peeling exploration, which {makes possible} to keep track of the geodesic $\Gamma_n$.

\paragraph{Scaling limit of the tree of geodesics.} If $\mathfrak{m}$ is a planar map, recall that $\T(\mathfrak{m})$ is the random tree of fpp geodesics from the root. Its vertices are the faces of $\mathfrak{m}$ and its edges are the dual edges which appear in the shortest paths for $d^\dagger_\mathrm{fpp}$ from each face to the root $f_r$. See e.g.\@ Figure \ref{fig: carte et arbre geodesiques}. We equip each edge of $\T(\mathfrak{m})$ with the random fpp length of the associated dual edge of $\mathfrak{m}$. We denote by $d_{\T(\mathfrak{m})}$ the corresponding distance on $\T(\mathfrak{m})$. Let $(f_i)_{i\ge 1}$ be the family of the faces of $\mathfrak{m}$ in non-increasing order of degrees.
\begin{theorem}\label{Th limite d echelle de l arbre}
	Assume $\bf q$ is of type $a \in (3/2,5/2)$. Then,
	$$
	\text{Under }\P^{(\ell)}, \qquad \qquad
	\left( \ell^{2-a}d_{\T(\mathfrak{m})} 
	(f_i,f_j)
	\right)_{i,j\ge 1}
	\mathop{\longrightarrow}\limits_{\ell \to \infty}^{(\mathrm{d})}
	\left(\frac{1}{2c_a p_{\bf q}}d_{\T_a}(w_i,w_j)\right)_{i,j\ge 1},
	$$
	for the product topology, where $c_a\coloneqq \pi/\Gamma(a) $ and  $\T_a =(\{w_i; \ i \ge 1\},  d_{\T_a})$ is a random countable metric space.
\end{theorem}

Let us give some insight into this scaling limit. Using the time reversal of the uniform peeling exploration, we describe the geodesics to the root using a family of coalescing processes. We prove that these processes converge towards a stochastic flow of coalescing diffusions with jumps. The metric space $\T_a =(\{w_i; \ i \ge 1\},  d_{\T_a})$ is then defined using this flow of diffusions and we show that it is indeed the scaling limit of $\T(\mathfrak{m})$. The exact definition of the limit is in Section \ref{sous-section definition de la limite}. The stochastic flow is driven by the jumps of the self-similar growth-fragmentation corresponding to the scaling limit of the slicing by heights as shown in \cite{BCK} in the generic case $a=5/2$, \cite{BBCK} in the dilute case $a\in (2,5/2)$ and \cite{Kam23CLE} in the critical case $a =2$ and by independent uniform random variables which arise from the uniform peeling exploration. The set $\{w_i; \ i \ge 1\}$ corresponds to the set of positive jumps of the growth-fragmentation, plus an additional point corresponding to the root face. One can then focus on the trajectories of the flow which ``start from'' the positive jumps and their coalescences give the genealogical structure of $\T_a$. Finally, the distances on $\T_a$ are measured using a Lamperti transform.

\subsection{Perpectives}
\paragraph{Generic case.} {Even though Theorem \ref{Th limite d echelle de l arbre} does not hold in the critical generic case $a=5/2$,} the scaling limit of the coalescing processes also holds in the generic case $a=5/2$. This should give a new description of the geodesics to the root in the Brownian disk as a flow of diffusions with jumps. However, due to the absence of high degree vertices, the scaling limit of the tree of fpp geodesics cannot be stated in the same way as in Theorem \ref{Th limite d echelle de l arbre} in this case so that another approach is needed in the case $a=5/2$. Actually, a related approach is carried out in Section 3 of \cite{MS21} for the axiomatic characterization of the Brownian sphere introducing a ``Lévy net''. {Let us also mention the fact that another construction of the Brownian disk is given in \cite{LG19}, which is different from that of \cite{Bett15,BM17}. One could try to identify the coalescing flow in one of these constructions. Furthermore, let us stress that the tree $\T(\mathfrak{m})$ of geodesics to the root does not correspond to the so-called cactus of the metric space $(\mathfrak{m}^\dagger,d^\dagger_\mathrm{fpp})$ whose analogue for $d_\mathrm{gr}$ is studied in \cite{CLGM13} since the cactus is defined using the geodesic between two points and the point of this geodesic which is closest to the root. The cactus does not describe the coalescence of geodesics to the root. In particular, we believe that the conjectural scaling limit of $\T(\mathfrak{m})$ under $\P^{(\ell)}$ in the case $a=5/2$ is different from the Brownian cactus of \cite{CLGM13}.}

\paragraph{Link with self-similar Markov trees.} Besides, one could define the stochastic flow on top of the self-similar Markov trees from \cite{BCR24} in the case $a \in (2, 5/2]$ but we do not in order to cover the whole range $(3/2,5/2]$ {since the self-similar Markov trees of \cite{BCR24} are required to be compact. Let us also insist on the fact that self-similar Markov trees only describe the so-called exploration tree of the map and do not give the scaling limit of geodesics to the root. Nevertheless, we believe that defining the stochastic flow on top of the self-similar Markov tree would be very relevant in order to define the scaling limit of the tree of geodesics to the root as a random compact measured metric space when $a \in (2,5/2]$.}
{
\paragraph{Critical case.} In the critical case $a=2$, in view of the results in \cite{Kam23CLE}, we expect the tree $\T_2$ to describe the genealogy of the loops in the uniform exploration of the conformal loop ensemble of parameter $\kappa=4$ introduced in \cite{WW13} where the boundary lengths of the loops are measured by an independent critical Liouville quantum gravity and where the distance on $\T_2$ is induced by the Lamperti transform of the quantum distance to the boundary which is denoted by $d_\mathrm{Lamperti}$ in \cite{Kam23CLE}. Indeed, the Lamperti transform of the quantum distance corresponds to the scaling limit of the distances to the root in $3/2$-stable maps as shown in \cite{Kam23CLE}. The quantum distance from the loops to the boundary first appeared in \cite{AHPS21}.
	
\paragraph{Geodesics for the dual graph distance.} The study of the fpp distance can be seen as a first step towards the understanding of the dual graph distance $d^\dagger_\mathrm{gr}$ so that the next step would be to perform the same study of geodesics to the root for the distance $d^\dagger_\mathrm{gr}$. This is a challenging task inasmuch as the study of $d^\dagger_\mathrm{gr}$ is usually far more technical than the study of $d^\dagger_\mathrm{fpp}$. We expect a scaling limit in the cases $a\in [2,5/2]$.
	
\paragraph{Scaling limit of random maps with high degrees?} Last but not least, using the flow of coalescing diffusions with jumps by adding some ``shortcuts'' we introduce in Section \ref{section conjectures} a random countable metric space which we conjecture to be the scaling limit of random Boltzmann maps for $a \in (3/2,5/2)$ when the maps are equipped with the fpp distance and for $a \in [2,5/2)$ when we equip the maps with the dual graph distance.}

{
\subsection{Techniques} 
As mentioned before, the main technique which is introduced and extensively used in the paper is a time reversal of the uniform peeling exploration. More precisely, conditionally on the perimeter process up to time time $n\ge 0$, we progressively construct a map which has the same law as the explored region $\overline{\mathfrak{e}}_n$ by gluing the discovered faces and inserting the holes which are filled in starting from the boundary $\partial \overline{\mathfrak{e}}_n$. At each step, the position where we glue the face or insert the hole is chosen uniformly at random. 

Using the time reversal of the uniform peeling exploration, we show that conditionally on the perimeter process, the number of faces along a fpp geodesic to the root can be written as a sum of independent Bernoulli random variables. The asymptotics of the number of faces along the geodesics is then obtained using precise estimates for the perimeter process which were first introduced in \cite{BCM}. Next, the comparisons between the fpp and the dual graph distances are performed using concentration inequalities which enable {us} to extend the control over the number of faces along one geodesic to all the geodesics from the root to the boundary of the ball. The upper bound {for} the diameter for $d^\dagger_\mathrm{gr}$ in the dense case $a \in (3/2,2)$ is obtained by combining the identification of the number of faces along a fpp geodesic with the techniques introduced in \cite{Kam23}. More precisely, in order to control the diameter, we first use a first moment method, then a re-rooting trick designed in \cite{Kam23} in order to exchange the roles of the root and of a uniform random edge and finally we use some martingales which were introduced in \cite{Kam23} for the same purpose in the case $a=2$. The use of fpp geodesics {is a way} to circumvent the direct study of the dual graph distance via the peeling by layers algorithm performed in Section 5.2 of \cite{BC}.

In order to obtain the scaling limit of the tree of geodesics, we code the time reversal of the uniform exploration using a family of coalescing processes on the torus $\R/\Z$ (which we define in a periodic way on $\R$). These processes are not Markovian since time runs in the opposite direction for the perimeter process $P$. Still, their limit is given by a flow of coalescing pure jump diffusions. This flow is described as a flow of SDEs driven by a Poisson point process. Due to the lack of Lipschitzianity of the coefficients, the existence and uniqueness of strong solutions is not obtained using the classical theorems such as Theorem 9.1 of Chapter IV of \cite{IW89} but it is obtained using the work \cite{LP12} which requires less regularity for the coefficients but asks in return {for} a monotonicity condition. We do not rely on the theory of coalescing flows of \cite{LJR04}. So as to show that the convergence of the coalescing processes associated with the time reversal of the exploration towards the coalescing flow of diffusions with jumps, since time is reversed for the perimeter process, we cannot directly apply classical limit theorems for diffusions with jumps from Chapter IX Section 4 of \cite{JS87}. Instead, building on the uniform random variables involved in the time-reversal of the uniform peeling exploration, we construct a coupling between one trajectory and an approximation where we removed all the jumps of size smaller than $\vp$ for some small $\vp>0$ such that the two trajectories are close to each other uniformly in $\ell\ge 1$. We establish the convergence of this approximation and then let $\vp\to 0$, applying the results of \cite{JS87} in the continuum. Once the convergence towards the coalescing flow is established, the scaling limit of the tree of fpp geodesics to the root boils down to the convergence of the ``times of birth'' of each face and of the coalescence times of the associated trajectories. The convergence of the coalescence times is obtained thanks to the fact that trajectories of the coalescing flow of diffusions with jumps have ``many large jumps'' in the case $a \in (3/2,2)$, so that the coalescence time is a time of a large jump, and ``many small jumps'' in the case $a \in [2,5/2)$, in a way that shortly after coalescing the trajectory has small jumps.
}

\subsection{Plan of the paper}
We define the time reversal of the uniform peeling exploration in Section \ref{section renversement de l'exploration}. Then, Section \ref{section nombre de faces}, independent of the next sections, is devoted to the scaling limit of the number of faces along the geodesics and its applications, i.e.\@ to the proof of Theorem \ref{th nombre de faces}, Corollaries \ref{cor inclusion des boules dilue} and \ref{cor inclusion des boules cas Cauchy} {as well as} Theorem \ref{th diametre cas dense}. Next, in Section \ref{section flot}, we define a coalescing flow of pure jump diffusions and prove that it arises as the scaling limit of a discrete coalescing flow which encodes the fpp geodesics to the root. Based on the coalescing flow, in Section \ref{section arbre geodesiques}, we define a random countable metric space and we prove that it is the scaling limit of the tree of fpp geodesics in Theorem \ref{Th limite d echelle de l arbre}. Finally, in Section \ref{section conjectures} we define a random countable metric space which we believe to be the scaling limit of large random planar maps with high degrees.

\section{The time reversal of the uniform peeling exploration}\label{section renversement de l'exploration}

In this section, we present a time-reversal of the uniform exploration of a random Boltzmann map conditionally on the perimeter process. This reversed uniform exploration is the central tool of this work. But first, we recall the definition of the uniform peeling exploration.

\subsection{The uniform peeling exploration}\label{sous-section explo uniforme}

We say that a map $\mathfrak{e}$ is a map with one hole if $\mathfrak{e}$ has a distinguished face with simple boundary which is different from $f_r$. Moreover, we say that $\mathfrak{e}$ is a submap of a map $\mathfrak{m} $ if one can obtain $\mathfrak{m}$ by gluing a map in the hole of $\mathfrak{e}$. The boundary $\partial \mathfrak{e}$ of a map with one hole is the set of edges on the boundary of its hole. {If $\mathfrak{e}$ is a submap of $\mathfrak{m}$ and $e \in \partial \mathfrak{e}$, then $e$ is adjacent to two faces of $\mathfrak{e}$ and one of them is the hole. The face which is ``behind'' the edge $e$ in $\mathfrak{m}$ is then defined as the face of $\mathfrak{m}$ which is glued to $e$ when we glue a map in the hole of $\mathfrak{e}$ in order to obtain $\mathfrak{m}$  (note that this face may be in $\mathfrak{e}$ or not).} A peeling algorithm $\mathcal{A}$ is a function which takes a map with one hole $\mathfrak{e}$ and returns an edge of $\partial \mathfrak{e}$. The algorithm $\mathcal{A}$ may be random.

Let $\mathfrak{m}$ be a map which is possibly infinite. A peeling exploration of $\mathfrak{m}$ following the peeling algorithm $\mathcal{A}$ is an increasing sequence of submaps $\overline{\mathfrak{e}}_0 \subset \overline{\mathfrak{e}}_1 \subset \cdots \subset \overline{\mathfrak{e}}_n \subset \cdots \subset \mathfrak{m}$ such that $\overline{\mathfrak{e}}_0$ is the map formed by gluing the root face $f_r$ of $\mathfrak{m}$ to a hole of the same degree, and such that for all $n\ge 0$, the submap $\overline{\mathfrak{e}}_{n+1}$ is obtained from $\overline{\mathfrak{e}}_n$ in the following way:
\begin{itemize}
	\item If the face $f$ behind the edge $\mathcal{A}(\overline{\mathfrak{e}}_n)$ in $\mathfrak{m}$ is not in $\overline{\mathfrak{e}}_n$, then we ``discover'' this face by gluing $f$ to $\overline{\mathfrak{e}}_n$ on the edge $\mathcal{A}(\overline{\mathfrak{e}}_n)$. See the two peeling steps on the right of the first row in Figure \ref{fig: explo unif}.
	\item If the face behind $\mathcal{A}(\overline{\mathfrak{e}}_n)$ in $\mathfrak{m}$ is already in  $\overline{\mathfrak{e}}_n$, then it is necessarily adjacent to the hole. Then we glue the edge $\mathcal{A}(\overline{\mathfrak{e}}_n)$ to the edge in $\partial \overline{\mathfrak{e}}_n$ following the adjacency in $\mathfrak{m}$. The hole then splits into two holes. When $\mathfrak{m}$ is infinite, we choose to fill in the hole which contains only a finite number of faces in $\mathfrak{m}$. When $\mathfrak{m}$ is finite, we choose to fill in the hole which has the smallest perimeter (and in case of equality we fill in one of the two holes chosen arbitrarily in a deterministic way). See the left of the first row and the second row in Figure \ref{fig: explo unif} for such peeling steps.
\end{itemize}
We denote by $(P(n))_{n\ge 0}$ the perimeter process which is defined as half of the number of edges in $\partial \overline{\mathfrak{e}}_n$. When $\mathfrak{m}$ is finite, the exploration stops when $P$ reaches zero.

\begin{figure}[h]
	\centering
	\includegraphics[scale=1.1]{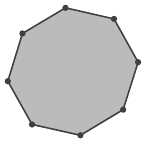}
	\includegraphics[scale=1.4]{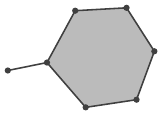}
	\includegraphics[scale=1.3]{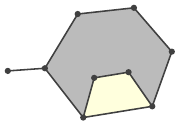}
	\includegraphics[scale=1]{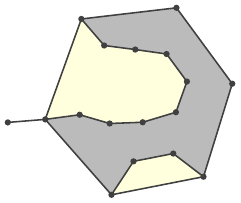}\\
	\includegraphics[scale=1]{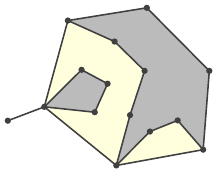}
	\includegraphics[scale=1]{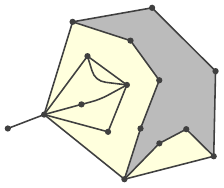}
	\includegraphics[scale=1.1]{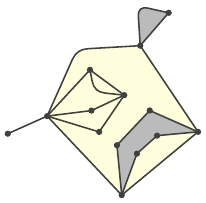}
	\includegraphics[scale=1.1]{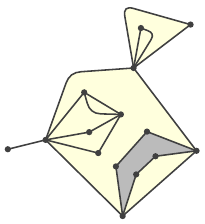}
	\caption{The first five steps of the exploration of the map $\mathfrak{m}$ from Figure \ref{fig: carte et arbre geodesiques}. From left to right on the first row: $\overline{\mathfrak{e}}_0$, $\overline{\mathfrak{e}}_1$, $\overline{\mathfrak{e}}_2$, $\overline{\mathfrak{e}}_3$. From left to right on the second row: the submap obtained after gluing two edges on $\partial \overline{\mathfrak{e}}_3$, $\overline{\mathfrak{e}}_4$ after filling in the hole with minimal perimeter, the submap obtained after gluing two edges on $\partial \overline{\mathfrak{e}}_4$, $\overline{\mathfrak{e}}_5$ after filling in the hole with minimal perimeter.}
	\label{fig: explo unif}
\end{figure}

Now, let us describe a peeling exploration which is very useful in the study of the fpp distance, which is the uniform peeling exploration. See Section 13.1 of \cite{StFlour} for more details. For all $t\ge 0$, let $\mathrm{Ball}^{\mathrm{fpp}}_t (\mathfrak{m})$ be the closed ball centered at $f_r$ of radius $t\ge0$ for the fpp distance $d^\dagger_{\mathrm{fpp}}$. It is defined as the submap whose faces are at fpp distance at most $t$ to the root face and containing only the dual edges whose points are at fpp distance at most $t$ from the root. 

Note that $\overline{\mathfrak{e}}_0 = \mathrm{Ball}^{\mathrm{fpp}}_0$. Next, for some step $n\ge 0$ of the exploration, assume that the submap $\overline{\mathfrak{e}}_n $ can be obtained from $\mathrm{Ball}^{\mathrm{fpp}}_{T_n} (\mathfrak{m})$ for some $T_n \ge 0$ by filling in all the holes except one. By the absence of memory of the exponential distribution, we know that the lengths of the dual edges on the boundary of $\overline{\mathfrak{e}}_n $ are i.i.d. exponential random variables of parameter $1$. We choose the next edge to peel to be the edge of $\partial \overline{\mathfrak{e}}_n$ whose dual edge has the smallest length. By well-known properties of the exponential distribution, this edge is actually chosen uniformly at random and the fpp length of the corresponding dual edge is an exponential random variable of expectation $1/\# \partial \overline{\mathfrak{e}}_n$ conditionally on $\overline{\mathfrak{e}}_n $ and on the fpp lengths of the dual edges of $\overline{\mathfrak{e}}_n $. This edge length can then be written $\mathcal{E}_n/(2P(n))$, w{h}ere $\mathcal{E}_n$ is an exponential r.v.\@ of parameter $1$. Let $T_{n+1} = T_n +\mathcal{E}_n /(2P(n))$. Then $\overline{\mathfrak{e}}_{n+1} $ is obtained from the ball $\mathrm{Ball}^{\mathrm{fpp}}_{T_{n+1}} (\mathfrak{m})$ by filling in all the holes, except one.

{Briefly}, using the uniform peeling algorithm which chooses at each step an edge on the boundary uniformly at random, we have for all $n \ge 0$, 
$$
\overline{\mathfrak{e}}_n = \overline{\mathrm{Ball}}^{\mathrm{fpp}}_{T_n}(\mathfrak{m})
\qquad
\text{with}
\qquad
T_n = \sum_{i=0}^{n-1} \frac{\mathcal{E}_i}{2P(i)},
$$
where the $\mathcal{E}_i$'s for $i\ge 0$ are i.i.d.\@ exponential r.v.s of parameter $1$ and where for all $t$, the submap $\overline{\mathrm{Ball}}^{\mathrm{fpp}}_{T_n}(\mathfrak{m})$ is obtained by filling in all the holes of $\mathfrak{m}$ except one, which corresponds to the infinite hole when $\mathfrak{m}$ is infinite.

\subsection{The law of the peeling exploration of Boltzmann planar maps}\label{sous-section loi de l'explo}
Under $\P^{(\ell)}$ (respectively $\P^{(\ell)}_\infty$) the peeling exploration has several nice features, which do not depend on the peeling algorithm. It is a Markov chain and its transitions can be described using the perimeter process.

First, the peeling exploration satisfies a spatial Markov property, stating that at each step $n\ge 0$, conditionally on $\overline{\mathfrak{e}}_n$, the map which fills the hole in $\mathfrak{m}$ is a random map of law $\P^{(p)}$ (resp.\@ $\P_\infty^{(p)}$), where $p= P(n)$. Moreover, at each time $n$ such that $P(n+1)< P(n)$, the map in the hole which is filled in after gluing the two edges is a random map of law $\P^{(p)}$ with $p = -\Delta P(n)-1$. See Propositions 4.6 and 6.3 of \cite{StFlour} for a detailed statement.

Moreover, the perimeter process itself is a Markov chain, whose transition probabilities are better written using the distribution $\nu$ defined by
\begin{equation}\label{eq def nu}
\forall k \in \Z,
\qquad
\qquad
\nu(k) = 
\left\{
\begin{matrix}
	q_{k+1} c_{\bf q}^k &\text{ if } k \ge 0 \\
	2 W^{(-1-k)} c_{\bf q}^k &\text{ if } k < 0,
\end{matrix}
\right.
\end{equation}
where we recall that $(q_k)_{k\ge 1}$ is the weight sequence, $W^{(\ell)}$ is the partition function of maps of perimeter $2\ell$ for all $\ell \ge 1$, and $c_{\bf q}$ appears in \eqref{eq asymptotique W}. 
It is known that $\nu $ is a probability measure (see e.g.\@ Lemma 5.2 of \cite{StFlour}). Under $\P^{(\ell)}_\infty$, the perimeter process $(P(n))_{n\ge 0}$ is the Doob $h$-transform of the $\nu$-random walk starting at $\ell$, using the harmonic function $h^\uparrow$ on $\Z_{\ge 1}$ defined by
\begin{equation}\label{eq h fleche}
\forall \ell \ge 1, \qquad 
h^\uparrow (\ell) = 2 \ell 2^{-2\ell} \binom{2\ell}{\ell}.
\end{equation}
See e.g.\@ Remark 7.5 of \cite{StFlour}. This Doob $h$-transform can be interpreted as a conditioning of the $\nu$-random {walk} to always stay positive. {Moreover}, under $\P^{(\ell)}$, the perimeter process $(P(n))_{n\ge 0}$ is a Markov chain starting at $\ell$, absorbed at zero, with transition probabilities for all $n\ge 0$, $p \ge 1, k \in \Z$,
\begin{equation}\label{eq probas de transition perimetre fini}
\P^{(\ell)}(P(n+1)= p+k \vert P(n) = p)
= \nu(k) c_{\bf q}^{-k}\frac{ W^{(p+k)}}{W^{(p)}} \left({\bf 1}_{p+k> (p-1)/2} +
\frac{1}{2} {\bf 1}_{p+k= (p-1)/2}
\right).
\end{equation}
See the second display of page 53 of \cite{BBCK} coming from Proposition 6.3  of \cite{BBCK}.

The tail behavior of $\nu$ is well understood for $\bf q$ of type $a$ thanks to Propositions 5.9 and 5.10 of \cite{StFlour}. We have, when $a \in (3/2,5/2)$, as $k\to \infty$,
\begin{equation}\label{queue nu}
\nu(-k) \sim p_{\bf q} k^{-a} \qquad \text{and} \qquad  \nu([k,\infty)) \sim p_{\bf q} \cos (a \pi )\frac{1}{a-1} k^{1-a},
\end{equation}
where the constant $p_{\bf q}$ comes from \eqref{eq asymptotique W}. When $a=5/2$, only the first asymptotic holds and the second one becomes $\nu([k,\infty))=o(k^{-3/2})$.

\subsection{Reversing the uniform exploration}\label{sous-section explo renversee}

Let $\mathfrak{m} $ be a planar map. Let $P$ be the perimeter process associated with the uniform peeling exploration of $\mathfrak{m}$. Let us define the decreasing sequence of submaps $\mathfrak{u}^n_0\supset \mathfrak{u}^n_1 \supset \cdots \supset \mathfrak{u}^n_n$ as follows: let $\mathfrak{u}^n_n$ be the map made of only one distinguished face of degree $2P(n)$ glued to one hole of the same degree. For all $k \in \lb 0,n-1 \rb $, we obtain $\mathfrak{u}^n_k$ from $\mathfrak{u}^n_{k+1}$ by distinguishing two cases:
\begin{itemize}
	\item If $\Delta P(k)  <0$ then $\mathfrak{u}^n_k$ is obtained by inserting $-2 \Delta P(k) $ consecutive edges in the boundary of the distinguished face at a position which is chosen uniformly at random. The two extremal inserted edges are glued and we fill in the hole delimited by the $-2 \Delta P(k) -2 $ remaining edges with a Boltzmann map of law $\P^{(p)}$, where $p\coloneqq -\Delta P(k)-1$. See the first row of Figure \ref{fig: explo unif renverse} and the rightmost figure in the second row for such exploration steps.
	\item If $\Delta P(k)  \ge 0$, then $\mathfrak{u}^n_k$ is obtained by gluing a new face $f$ of degree $2\Delta P(k)+2$ to $2\Delta P(k)+1$ consecutive edges where the position of the first of these consecutive edges is chosen uniformly at random on the boundary of the distinguished face. See the exploration steps $\mathfrak{u}^5_2$ and $\mathfrak{u}^5_1$ in the second row of Figure \ref{fig: explo unif renverse} for this case.
\end{itemize}

\begin{figure}[h]
	\centering
	\includegraphics[width=0.20\textwidth]{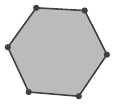}
	\hspace{0.8cm}
	\includegraphics[width=0.22\textwidth]{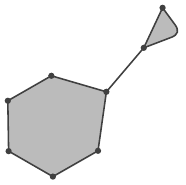}
	\hspace{0.6cm}
	\includegraphics[width=0.22\textwidth]{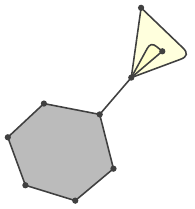}
	\hspace{0.7cm}
	\includegraphics[width=0.15\textwidth]{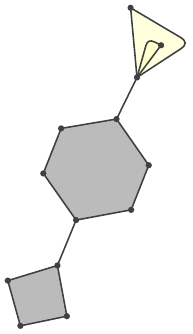}\\
	\includegraphics[width=0.14\textwidth]{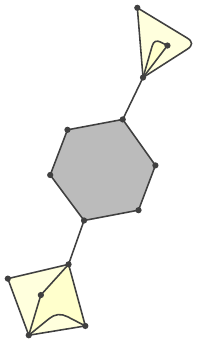}
	\includegraphics[width=0.25\textwidth]{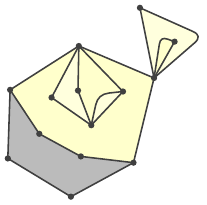}
	\includegraphics[width=0.25\textwidth]{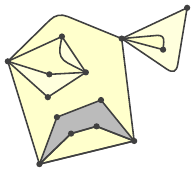}
	\includegraphics[width=0.25\textwidth]{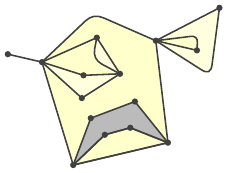}
	\caption{The five steps of the reversed uniform exploration $\mathfrak{u}^5_5 \subset \mathfrak{u}^5_4 \subset \mathfrak{u}^5_3 \subset \mathfrak{u}^5_2 \subset \mathfrak{u}^5_1 \subset \mathfrak{u}^5_0$ of the map $\mathfrak{m}$ starting from time $n=5$. They correspond to the time-reversal of the exploration depicted in Figure \ref{fig: explo unif}. From left to right on the first row: $\mathfrak{u}^5_5$, the map obtained by inserting $4$ edges at a position chosen uniformly at random and gluing the two extremal edges, hence creating a hole of perimeter $2$, the map $\mathfrak{u}^5_4$ obtained by filling in the hole, the map obtained by inserting $6$ edges at a position chosen uniformly at random and gluing the two extremal edges, hence creating a hole of perimeter $4$. From left to right on the second row: the submap $\mathfrak{u}^5_3$ obtained after filling in the hole, $\mathfrak{u}^5_2$, $\mathfrak{u}^5_1$ and $\mathfrak{u}^5_0$.}
	\label{fig: explo unif renverse}
\end{figure}

Conditionally on the perimeter process, let us relate this backward exploration to the uniform peeling exploration $(\overline{\mathfrak{e}}_k)_{k\ge 0}$. For all $n\ge 0$, for all $k \in \lb 0, n \rb$, let $\overline{\mathfrak{e}}_n \setminus \overline{\mathfrak{e}}_k$ be the submap of $\overline{\mathfrak{e}}_n$ obtained by identifying the explored region $\overline{\mathfrak{e}}_k$ at time $k$ to one (non-simple) distinguished face of degree $2P(k)$.

\begin{proposition}\label{lien forward backward}
	For all $n\ge 0$, under either $\P^{(\ell)}_\infty$ or $\P^{(\ell)}( \ \cdot \ \vert P(n)>0)$, conditionally on $(P(k))_{0\le k \le n}$, the sequence $(\mathfrak{u}^n_k)_{0 \le k \le n}$ has the same distribution as $(\overline{\mathfrak{e}}_n \setminus \overline{\mathfrak{e}}_k)_{0\le k \le n}$. In particular, conditionally on $(P(k))_{0\le k \le n}$, the map $\mathfrak{u}^n_0$ has the same law as $\overline{\mathfrak{e}}_n$.
\end{proposition}

\begin{proof}
	Let us prove the result by induction on $n$. The result is obvious for $n=0$ since $\overline{\mathfrak{e}}_0$ is the map consisting of the root face $f_r$ and one hole. Assume that the statement holds for $n$. By definition of the uniform peeling exploration, $\overline{\mathfrak{e}}_{n+1}$ is obtained after peeling an edge of $\partial \overline{\mathfrak{e}}_n$ uniformly at random. Let us check that under $\P^{(\ell)}_\infty$ (or $\P^{(\ell)}$ under the event ``$P(n)>0$''), conditionally on $(P(k))_{0\le k \le n}$ and on $(\overline{\mathfrak{e}}_k)_{0\le k \le n}$, the map $\mathfrak{u}^{n+1}_n$ has the same law as $\overline{\mathfrak{e}}_{n+1} \setminus \overline{\mathfrak{e}}_n$ obtained by identifying all the explored region $\overline{\mathfrak{e}}_n$ to one non-simple face. To do this we reason differently according to the sign of $\Delta P(n)$.
	\begin{itemize}
		\item If $\Delta P(n)<0$, then the uniform exploration glues in the hole the peeled edge to another edge of $\partial \overline{\mathfrak{e}}_n$ hence creating a hole of degree $-2\Delta P(n) -2$ and then fills in this hole with a Boltzmann map of law $\P^{(p)}$, where $p\coloneqq -\Delta P(n) -1$.
		\item  If $\Delta P(n) \ge 0$, then the uniform exploration glues in the hole a new face $f$ of degree $2 \Delta P(n) +2$ to the peeled edge.
	\end{itemize}
	In both cases, it is clear that the map $\overline{\mathfrak{e}}_{n+1} \setminus \overline{\mathfrak{e}}_n$ has the same law as $\mathfrak{u}^{n+1}_n$. To conclude, it suffices to notice that the $\mathfrak{u}^{n+1}_k$ for $k \in \lb 0, n \rb$ are obtained by gluing the boundary of the hole of $\mathfrak{u}^n_k$ in the distinguished face of $\mathfrak{u}^{n+1}_n$ in a uniformly chosen way. Similarly, since the peeled edge is chosen uniformly at random, $\overline{\mathfrak{e}}_{n+1} \setminus \overline{\mathfrak{e}}_k$ is obtained by gluing the boundary of the hole of $ \overline{\mathfrak{e}}_n \setminus \overline{\mathfrak{e}}_k$ in the distinguished face of $\overline{\mathfrak{e}}_{n+1} \setminus \overline{\mathfrak{e}}_n$ in a uniformly chosen way.
\end{proof}

{For all $n\ge 0$, let $e_n$ be an edge of $\mathfrak{u}^n_n$ which is chosen independently from $(\mathfrak{u}^n_k)_{0\le k\le n}$ conditionally on $(P(k))_{0\le k \le n}$. }
For all $n\ge 0$, let us denote by $\Gamma_n$ the fpp geodesic from $f_r$ to the already explored face next to {the edge $e_n$ in the explored region $\mathfrak{u}^n_0 = \overline{\mathfrak{e}}_n$.} The geodesic $\Gamma_n$ only contains faces corresponding to the positive jumps of $(P(i))_{0 \le i \le n}$. Moreover, conditionally on $P$, independently for each $i \in \lb 0, n-1\rb$ such that $\Delta P(i)\ge 0$, the face $f_i$ discovered at time $i$ belongs to $\Gamma_n$ with probability
\begin{equation}\label{proba coalescence}
\theta_i\coloneqq \P(f_i \in \Gamma_n \vert P)= \frac{2 \Delta P(i)+1}{2P(i+1)} {\bf 1}_{\Delta P(i) \ge 0}.
\end{equation}
Indeed, this can be checked by induction using Proposition \ref{lien forward backward} and the fact that two adjacent faces are in the same fpp geodesic to the root if and only if they share an edge which has been peeled by the uniform exploration (and in that case there is a shared edge corresponding to the discovery of one of the two faces).
\section{Number of faces in the geodesics}\label{section nombre de faces}

In this section, we prove Theorem \ref{th nombre de faces} which gives the scaling limit of the number of faces along the fpp geodesics to the root for infinite Boltzmann planar maps of law $\P^{(1)}_\infty$. We start by stating some useful results on random walks conditioned to be positive. Then we enter the proof of the theorem. Finally, we prove the applications to the dual graph distance: the comparison between the dual graph distance and the fpp distance and the upper bound on the diameter.
\subsection{Preliminaries on random walks conditioned to stay positive}\label{sous-section preliminaires marche conditionnee}
We first state some useful scaling limit and local limit theorems for the perimeter process $(P(n))_{n\ge 0}$ under $\P^{(1)}_\infty$ which can be viewed as the $\nu$-random walk starting at $1$ and conditioned to {always stay} positive and whose law is described in Subsection \ref{sous-section loi de l'explo}. We will often add a subscript $\infty$ to $P$ in order to underline the fact that we work under $\P^{(1)}_\infty$. Thanks to \eqref{queue nu}, the scaling limit of a $\nu$-random walk starting at zero $(S_n)_{n\ge 0}$ is given by the $(a-1)$-stable Lévy process $\Upsilon_a$  starting at zero with Lévy measure $\cos(a \pi) (dx/x^a) {\bf 1}_{x>0} + (dx/\vert x \vert^a) {\bf 1}_{x<0}$. It is known that $(S_{\lfloor nt \rfloor}/n^{1/(a-1)})_{t\ge 0}$ converges in distribution towards $(\Upsilon_a(p_{\bf q} t))_{t\ge 0}$. See e.g.\@ Proposition 10.1 of \cite{StFlour}. The Lévy process starting at $x>0$ conditioned to {always stay} positive $\mathcal{S}^\uparrow_x$ is defined as the Doob $h$-transform of $x+\Upsilon_a$ using the harmonic function $y \mapsto {\bf 1}_{y\ge 0} \sqrt{y}$. By letting $x\downarrow 0$, one obtains the Lévy process $\Upsilon_a$ starting at zero and conditioned to {always stay} positive which is denoted by $\Upsilon^\uparrow_a$. The scaling limit of the perimeter process is stated in Proposition 10.3 in \cite{StFlour}:
\begin{equation}\label{cv perimetre}
	\text{Under } \P^{(1)}_\infty, 
	\qquad 
	\qquad
	\left( \frac{P_\infty(nt)}{n^{1/(a-1)}} \right)_{t\ge 0}
	\mathop{\longrightarrow}\limits_{n \to \infty}^{(\mathrm{d})}
	\left(
	\Upsilon^\uparrow_a(p_{\bf q} t)
	\right)_{t\ge 0}
\end{equation}
for the Skorokhod $J_1$ topology. We then state a local limit theorem which is proven in the case $a=2$ in \cite{BCM}. The same proof actually works for all $a \in (3/2,5/2]$. Let $f_a^\uparrow$ be the density of $\Upsilon_a^\uparrow(1)$.
\begin{lemma}\label{lemme limite locale}(Lemma 3 of \cite{BCM})
For all $n,x,y \in \N$, let
$$
\eta_n(x,y) = 
\sqrt{\frac{n^{1/(a-1)}}{y}} \left(
(np_{\bf q})^{1/(a-1)} \frac{2\sqrt{y}}{h^\uparrow(y) \sqrt{\pi}} \P^{(x)}_\infty(P_\infty(n)=y)
- f^\uparrow_a\left(\frac{y}{(p_{\bf q} n)^{1/(a-1)}} \right) 
\right).
$$
Then
\begin{itemize}
\item Uniformly in $y\ge 1$, as $n \to \infty$ we have $\eta_n(1,y)\to0$.
\item For every $\vp >0$ and every sequence $\delta_n \to 0$, uniformly in $x \in [1,n^{1/(a-1)}\delta_n]$ and uniformly in $y \in[\vp n^{1/(a-1)} , \vp^{-1} n^{1/(a-1)} ]$, as $n\to \infty$ we have $\eta_n(x,y) \to 0$.
\end{itemize}
\end{lemma}
The above lemma has the following direct corollary:
\begin{corollary}\label{corollaire decorrelation}
For every $\vp >0$ and every sequence $\delta_n \to 0$, we have
$$
n^{1/(a-1)}
\left\vert
\P^{(x)}_\infty(P_\infty(n)=y) - \P^{(x')}_\infty(P_\infty(n)=y)
\right\vert
\mathop{\longrightarrow}\limits_{n\to \infty} 0
$$
uniformly in $x,x' \in [1,n^{1/(a-1)}\delta_n]$ and uniformly in $y \in[\vp n^{1/(a-1)} , \vp^{-1} n^{1/(a-1)} ]$.
\end{corollary}

Furthermore, the following lemma is proven in the case $a=2$ but the same proof works for all $a \in (3/2, 5/2]$.
\begin{lemma}\label{lemme uniforme integrabilite}(Lemma 4 of \cite{BCM})
For all $0<r<3/2$, we have
$$
\sup_{n\ge 1} \E^{(1)}_\infty\left(
\left(\frac{n^{1/(a-1)}}{P_\infty(n)}\right)^r
\right)
<\infty.
$$
\end{lemma}

\begin{remark}\label{remarque cv 1 sur P}
By Lemma \ref{lemme uniforme integrabilite} and by (\ref{cv perimetre}), if ${\bf q}$ is of type $a <5/2$, then $(n/P_\infty(n)^{a-1})_{n\ge 0}$ is uniformly integrable and converges in distribution. Hence
$$
\E^{(1)}_\infty \left(
\frac{n}{P_\infty(n)^{a-1}}
\right)
\mathop{\longrightarrow}\limits_{n\to \infty}
\E\left(
\frac{1}{\Upsilon^\uparrow_a(p_{\bf q})^{a-1}}
\right)
=\E\left(\frac{1}{p_{\bf q}\Upsilon^\uparrow_a(1)^{a-1}} \right)
=\frac{2\Gamma(a)}{p_{\bf q} (a-1) \pi^{3/2} \Gamma(a-3/2)}
.
$$
where the last equality comes from the end of the proof of Proposition 2.7 in \cite{Kam23}.
\end{remark}
From the previous lemma, as shown in \cite{BCM}, one can deduce that the perimeter process decorrelates at different scales, in the following sense:
\begin{lemma}\label{Lemme decorrelation}(Lemma 5 of \cite{BCM})
Let $A_n$ be a sequence tending to $\infty$. For all $\vp>0$, for all $i \ge 1$, let $I(i,\vp) =  [\vp i^{1/(a-1)}, \vp^{-1} i^{1/(a-1)}]$ and let
$$
X_i^{(\vp)} \coloneqq \frac{1}{P_\infty(i)^{a-1}} {\bf 1}_{P_\infty(i) \in I(i, \vp)} - \E^{(1)}_\infty \left(\frac{1}{P_\infty(i)^{a-1}} {\bf 1}_{P_\infty(i) \in I(i, \vp)}
\right).
$$
If ${\bf q}$ is of type $a<5/2$, then 
$$
\lim_{\vp \to 0} \limsup_{n\to \infty} \sup_{1\le i \le n} \sup_{A_n i \le j \le n}
\left\vert
\E^{(1)}_\infty
\left(
i X_i^{(\vp)}
j X_j^{(\vp)}
\right)
\right\vert
=0.
$$
\end{lemma}
\begin{proof}
The proof follows the same lines as in the case $a=2$ treated in \cite{BCM}, but let us write it for completeness. We write
\begin{align*}
\Big\vert
\E^{(1)}_\infty
\left(
i X_i^{(\vp)}
j X_j^{(\vp)}
\right)
\Big\vert 
\le 
&\left\vert
\E^{(1)}_\infty
\left(
i X_i^{(\vp)}
j X_j^{(\vp)}
{\bf 1}_{P_\infty(i)\not\in I(i, \vp)}
\right)
\right\vert 
+\left\vert
\E^{(1)}_\infty
\left(
i X_i^{(\vp)}
j X_j^{(\vp)}
{\bf 1}_{P_\infty(j)\not\in I(j, \vp)}
\right)
\right\vert
\\
&+
\left\vert
\E^{(1)}_\infty
\left(
i X_i^{(\vp)}
j X_j^{(\vp)}
{\bf 1}_{P_\infty(i)\in I(i, \vp)} {\bf 1}_{P_\infty(j) \in I(j, \vp)}
\right)
\right\vert.
\end{align*}
Let $K_i^{(\vp)}= \E^{(1)}_\infty (\frac{1}{P_\infty(i)^{a-1}} {\bf 1}_{P_\infty(i) \in I(i, \vp)})$. Note that on the event $\{P_\infty(i) \not \in I(i, \vp)\}$, we have $X_i^{(\vp)}= -K_i^{(\vp)}$, so the first term can be written {as}
$$
\left\vert
\E^{(1)}_\infty
\left(
i X_i^{(\vp)}
j X_j^{(\vp)}
{\bf 1}_{P_\infty(i)\not\in I(i, \vp)}
\right)
\right\vert
=
(i K_i^{(\vp)})
\left\vert
\E^{(1)}_\infty
\left(
j X_j^{(\vp)}
{\bf 1}_{P_\infty(i)\not\in I(i, \vp)}
\right)
\right\vert,
$$
which converges to zero as $\vp \to 0$ uniformly in $i,j\ge 1$ since $(i K_i^{(\vp)})_{i\ge 1,\vp >0}$ is {bounded from above} by some constant $C>0$ by Remark \ref{remarque cv 1 sur P}, since $(j X_j^{(\vp)})_{j\ge 1,\vp >0}$ is uniformly integrable by Lemma \ref{lemme uniforme integrabilite} using that $a-1<3/2$ and since $(P_\infty(i)/i^{1/(a-1)})_{i\ge 1}$ is tight by \eqref{cv perimetre}. The second term is handled in the same way. Thus, it only remains to show that
\begin{equation}\label{eq lim limsup sup sup X epsilon}
\lim_{\vp \to 0} \limsup_{n\to \infty} \sup_{1 \le i \le n} \sup_{A_n i \le j \le n}
\left\vert
\E^{(1)}_\infty
\left(
i X_i^{(\vp)}
j X_j^{(\vp)}
{\bf 1}_{P_\infty(j)/j^{1/(a-1)},P_\infty(i)/i^{1/(a-1)} \in [\vp,1/\vp]}
\right)
\right\vert
=0.
\end{equation}
Notice that for $n$ large enough, for all $A_n i \le j \le n$ we have $j-i \ge j/2$, so that $(2(j-i))/A_{j-i} \ge i$. Therefore, by Corollary \ref{corollaire decorrelation} with $\vp, n, \delta_n$ replaced by $\vp/2^{1/(a-1)},j-i,(2/A_{j-i})^{1/(a-1)}/\vp$, we know that for $n$ large enough, for all $A_n i \le j \le n$,
\begin{align*}
&\!\!\!\!\sup_{\substack{1\le x \le i^{1/(a-1)}/\vp\\ \vp j^{1/(a-1)} \le y \le j^{1/(a-1)}/\vp}} \!
\left\vert
\P_\infty^{(x)}(P_\infty(j-i)=y)
-\P_\infty^{(1)}(P_\infty(j-i)=y)
\right\vert
\\
&\le \!\!\!\!\!
\sup_{\substack{1\le x \le (j-i)^{1/(a-1)}(2/A_{j-i})^{1/(a-1)}/\vp 
\\ (\vp/2^{1/(a-1)}) (j-i)^{1/(a-1)} \le y \le (2^{1/(a-1)}/\vp) (j-i)^{1/(a-1)}}}
\!\!
\left\vert
\P_\infty^{(x)}(P_\infty(j-i)=y)
-\P_\infty^{(1)}(P_\infty(j-i)=y)
\right\vert \\
&\le \vp^{a}/j^{1/(a-1)}
\end{align*}
and so by the Markov property
\begin{align*}
{\bf 1}_{P_\infty(i) \le i^{{1}/{(a-1)}}/\vp}
&\left\vert
\E^{(1)}_\infty\left(
\left.
j X_j^{(\vp)}
{\bf 1}_{P_\infty(j) \in I(j,\vp)}
\right\vert P_\infty(i)
\right)-
\E^{(1)}_\infty\left(
\left.
j X_j^{(\vp)}
{\bf 1}_{P_\infty(j) \in I(j,\vp)]}
\right\vert P_\infty(i)=1
\right)
\right\vert \\
&\le 
{\bf 1}_{P_\infty(i) \le i^{{1}/{(a-1)}}/\vp}
\sum_{j^{1/(a-1)} \vp \le y \le j^{1/(a-1)}/\vp}
\left\vert
\frac{j}{y^{a-1}}- jK_j^{(\vp)}
\right\vert
\frac{\vp^{a}}{j^{1/(a-1)}}
\\
&\le
{\bf 1}_{P_\infty(i) \le i^{{1}/{(a-1)}}/\vp}
\left(\frac{1}{\vp^{a-1}} + C \right)
\vp^{a-1}
.
\end{align*}
Moreover, for $n$ large enough, for all $A_n i \le j \le n$,
$$
\left\vert
\E^{(1)}_\infty\left(
\left.
j X_j^{(\vp)}
{\bf 1}_{P_\infty(j) \in [j^{1/(a-1)}\vp,j^{1/(a-1)}/\vp]}
\right\vert P_\infty(i)=1
\right)
\right\vert
\le \E_\infty^{(1)} \left(
\left\vert
j X^{(\vp)}_{j-i}
\right\vert
\right)
\le  4 C.
$$
{Furthermore}, using the fact that $\E^{(1)}_\infty  X_i^{(\vp)}=0$, we have
\begin{align*}
\left|
\E^{(1)}_\infty
\left(
iX^{(\vp)}_i 
{\bf 1}_{P_\infty(i)\in I(i, \vp)}
\right)
\right|
&=
\left\vert
\E^{(1)}_\infty
\left(
iX^{(\vp)}_i 
{\bf 1}_{P_\infty(i)\not\in I(i, \vp)}
\right)
\right\vert \\
&=i K_i^{(\vp)}
\P^{(1)}_\infty
\left(
P_\infty(i) \not\in
I(i, \vp)
\right) \\
&\le C \P^{(1)}_\infty
\left(
\Upsilon^\uparrow_a(p_{\bf q}) \not\in
[\vp , 1/\vp]
\right),
\end{align*}
for some constant $C>0$. Therefore,
\begin{align*}
\Big\vert
\E_\infty^{(1)}
&\left(
iX_i^{(\vp)} j X_j^{(\vp)} 
{\bf 1}_{P_\infty(i)/i^{1/(a-1)}, P_\infty(j)/j^{1/(a-1)} \in [\vp,1/\vp]}
\right)
\Big\vert \\
&\le C \P^{(1)}_\infty
\left(
\Upsilon^\uparrow_a(p_{\bf q}) \not\in
[\vp , 1/\vp]
\right)
(4C + (1/\vp^{a-1}+C)\vp^{a-1}),
\end{align*}
which converges towards zero as $\vp \to 0$.
\end{proof}

\subsection{Number of faces in the geodesics}

We are now well equipped to deal with the proof of Theorem \ref{th nombre de faces}. For sake of readability, we treat separately the dilute case $a \in (2,5/2]$, the critical case $a=2$ and the dense case $a \in (3/2,2)$.

\subsubsection{Dilute case}
\begin{proof}[Proof of Theorem \ref{th nombre de faces} in the dilute case]
For all $i\ge 0$, let $f_i$ be the face of degree $2\Delta P_\infty(i)+2$ which is discovered at time $i$ of the exploration when $\Delta P_\infty(i)\ge 0$. By (\ref{proba coalescence}), we know that conditionally on $P_\infty$, the $X_i \coloneqq {\bf 1}_{f_i \in \Gamma_{\lfloor nt \rfloor}}$'s for $i \in \lb 0, \lfloor nt \rfloor -1 \rfloor \rb $ are independent Bernoulli random variables of parameters $\theta_i$.

Let us start with the case $a \in (2,5/2]$. Notice that it suffices to check that
	\begin{equation}\label{eq cv theta dilue}
		\left(n^{-\frac{a-2}{a-1}}\sum_{i=0}^{\lfloor nt \rfloor} \theta_i \right)_{t\ge 0}
		\mathop{\longrightarrow}\limits_{n \to \infty}^{(\mathrm{d})}
		\left(\frac{{e}_{\bf q}}{2}
		\int_0^t \frac{ds}{\Upsilon_a^\uparrow(p_{\bf q} s)}\right)_{t\ge 0}
	\end{equation}
	for the topology of uniform convergence on compact sets, which is equivalent to the convergence in terms of finite dimensional distributions since the limiting process is a continuous increasing process.
	
	Indeed, if we suppose (\ref{eq cv theta dilue}), then we deduce that for all $t \ge 0$,
	$$
	\E^{(1)}_\infty \left( \left. \left( n^{-\frac{a-2}{a-1}} \sum_{i=0}^{\lfloor nt \rfloor} (X_i-\theta_i) \right)^2 \ \right\vert  P_\infty \right)
	= n^{-2\frac{a-2}{a-1}}\sum_{i=0}^{\lfloor nt \rfloor} (\theta_i-\theta_i^2)
	\le n^{-2\frac{a-2}{a-1}}\sum_{i=0}^{\lfloor nt \rfloor } \theta_i
	\mathop{\longrightarrow}\limits_{n\to \infty}^{(\P)} 0.
	$$
	In order to prove (\ref{eq cv theta dilue}), it suffices to show that
	\begin{equation}\label{cv theta conditionne dilue}
	\left(n^{-\frac{a-2}{a-1}}\sum_{i=0}^{\lfloor nt \rfloor} \E^{(1)}_\infty( \theta_i | P_\infty(i))\right)_{t \ge 0}
		\mathop{\longrightarrow}\limits_{n \to \infty}^{(\mathrm{d})}
		\left(\frac{{e}_{\bf q}}{2}
		\int_0^t \frac{ds}{\Upsilon_a^\uparrow(p_{\bf q} s)}\right)_{t\ge 0}.
	\end{equation}
	Indeed, by a straightforward computation, for all $t\ge 0$, 
	\begin{align*}
	\E^{(1)}_\infty
	\left(
	\left(
	\sum_{i=0}^{\lfloor nt \rfloor}
	\left(
	\theta_i- \E^{(1)}_\infty( \theta_i | P_\infty(i))
	\right)
	\right)^2
	\right)
	&=
	\sum_{i=0}^{\lfloor nt \rfloor}
	\E^{(1)}_\infty
	\left(
	\left(
	\theta_i- \E^{(1)}_\infty( \theta_i | P_\infty(i))
	\right)^2
	\right)
	\le 
	\sum_{i=0}^{\lfloor nt \rfloor}
	\E^{(1)}_\infty
	\left(
	\theta_i^2
	\right),
	\end{align*}
	and we have $\E^{(1)}_\infty
	\left(
	\theta_i^2
	\right) = O(i^{-1/(a-1)})$ due to the fact that, thanks to \eqref{queue nu} and since $h^\uparrow(p)\sim 2 \sqrt{p/\pi}$ as $p \to \infty$, 
	\begin{equation}\label{majoration theta carre}
	\E^{(1)}_\infty (\theta_i^2| P_\infty(i)=p) = \sum_{k\ge0} \left(\frac{2k+1}{2(p+k)}\right)^2 \frac{h^\uparrow(p+k)}{h^\uparrow(p)}\nu(k) = O \left(\frac{1}{\sqrt{p}} \sum_{k\ge 0} \frac{k^{2-a}}{(p+k)^{3/2}}  \right) 
	= O\left(\frac{1}{p^{a-1}}\right)
	\end{equation}
	and to the fact that $a-1>1$ and $\E^{(1)}_\infty(1/P_\infty(i)) = O(i^{-1/(a-1)})$ by Equation (18) of \cite{BC}.
	As a result, 
	\begin{equation}\label{eq theta proche de esp conditionnelle}
	n^{-2\frac{a-2}{a-1}}\E^{(1)}_\infty
	\left(
	\left(
	\sum_{i=0}^{\lfloor nt \rfloor}
	\left(
	\theta_i- \E^{(1)}_\infty( \theta_i | P_\infty(i))
	\right)
	\right)^2
	\right)
	=O\left(
	n^{-\frac{a-2}{a-1}}
	\right)
	\mathop{\longrightarrow}\limits_{n \to \infty} 0,
	\end{equation}
	so that (\ref{cv theta conditionne dilue}) implies (\ref{eq cv theta dilue}).
	
	Now, since $h^\uparrow(p)\sim 2 \sqrt{p/\pi}$ as $p \to \infty$ and since ${p}{h^\uparrow(p+k)}/(h^\uparrow(p)({p+k}))$ is {bounded from above} by a constant, we have
	\begin{equation}\label{eq esp conditionnelle proche de un sur P}
	p\E^{(1)}_\infty(\theta_i \vert P_\infty(i)=p ) = \sum_{k\ge0} \frac{2k+1}{2} \nu(k) \frac{p}{p+k} \frac{h^\uparrow(p+k)}{h^\uparrow(p)}
	\mathop{\longrightarrow}\limits_{p \to \infty}
	\frac{e_{\bf q}}{2}.
	\end{equation}
	Therefore, by \eqref{cv perimetre}, to show (\ref{cv theta conditionne dilue}) it suffices to see that
	$$
	n^{-\frac{a-2}{a-1}} \sum_{i=0}^{\lfloor nt \rfloor} \frac{1}{P_\infty(i)}
	\mathop{\longrightarrow}\limits_{n\to \infty}^{(\mathrm{d})}
	\int_0^t \frac{1}{\Upsilon_a^\uparrow(p_{\bf q} s) }ds,
	$$
	which is a consequence of the end of the proof of Proposition 4.1 in \cite{BC}.
\end{proof}
\subsubsection{Critical case}
\begin{proof}[Proof of Theorem \ref{th nombre de faces} in the case $a=2$]
By the same reasoning as in the dilute case, it is enough to prove that
\begin{equation}\label{eq cv theta critique}
\frac{1}{(\log n)^2}\sum_{i=1}^n \theta_i
\mathop{\longrightarrow}\limits_{n\to \infty}^{(\mathrm{L}^1)}
\frac{1}{\pi^2}.
\end{equation}
Let us first show that
\begin{equation}\label{eq cv theta critique condition}
\frac{1}{(\log n)^2}\sum_{i=1}^n \E^{(1)}_\infty(\theta_i \vert P_\infty(i))
\mathop{\longrightarrow}\limits_{n\to \infty}^{(\mathrm{L}^1)}
\frac{1}{\pi^2}.
\end{equation}
For all $i \ge 1$, we have
\begin{equation}\label{eq expression theta condition}
\E^{(1)}_\infty(\theta_i \vert P_\infty(i)) = \sum_{k=0}^\infty \frac{2k+1}{2(P_\infty(i)+k)} \frac{h^\uparrow(P_\infty(i)+k)}{h^\uparrow(P_\infty(i))} \nu(k).
\end{equation}
So, by (\ref{queue nu}) and the fact that $h^\uparrow(p) \sim 2\sqrt{p/\pi}$, as $p\to \infty$, the expectation $\E^{(1)}_\infty(\theta_i \vert P_\infty(i)=p)$ is {bounded from above} by
$$
O\left(\frac{1}{p}\right)
+O\left(\sum_{k=1}^{\infty} \frac{1}{k\sqrt{p(p+k)}}\right) 
=O\left(\frac{1}{p}\right)
+O\left(\frac{1}{p}\int_{1/p}^{\infty} \frac{1}{x\sqrt{1+x}} dx\right)
=O\left(\frac{\log p}{p}\right).
$$
Moreover, by the end of the proof of Proposition 5 in \cite{BCM}, we know that the sequence $((i\log P_\infty(i))/(P_\infty(i)\log i))_{i\ge 2}$ is uniformly integrable, so that $((i/\log i)\E^{(1)}_\infty(\theta_i \vert P_\infty(i)))_{i\ge 1}$ is also uniformly integrable.
{Furthermore}, using (\ref{queue nu}), (\ref{cv perimetre}) and the fact that $h^\uparrow(p) \sim 2\sqrt{p/\pi}$ as $p \to \infty$, it is easy to see that $\E^{(1)}_\infty(\theta_i \vert P_\infty(i))$ is equivalent in probability to
$$
\sum_{k=1}^\infty \frac{p_{\bf q}}{k \sqrt{P_\infty(i) (P_\infty(i)+k)}} 
 \mathop{\sim}\limits_{i \to \infty}
\frac{1}{P_\infty(i)} \int_{1/P_\infty(i)}^\infty \frac{p_{\bf q}}{u\sqrt{1+u}} du 
 \mathop{\sim}\limits_{i \to \infty}
\frac{p_{\bf q} \log P_\infty(i)}{P_\infty(i)}
\mathop{\sim}\limits_{i \to \infty}
\frac{p_{\bf q} \log i}{P_\infty(i)}
.
$$

Thus, by \eqref{cv perimetre},
$$
\frac{i \log P_\infty(i)}{P_\infty(i) \log i}-
\frac{i}{\log i}\E^{(1)}_\infty(\theta_i \vert P_\infty(i))
\mathop{\longrightarrow}\limits_{i \to \infty}^{(\mathrm{L}^1)}
0
.
$$
As a consequence, in order to show \eqref{eq cv theta critique condition} it suffices to prove that
$$
\frac{1}{(\log n)^2}\sum_{i=1}^n \frac{p_{\bf q} \log i}{P_\infty(i)}
\mathop{\longrightarrow}\limits_{n\to \infty}^{(\mathrm{L}^1)}
\frac{1}{\pi^2}.
$$
By Remark \ref{remarque cv 1 sur P} (or by Remark 3 of \cite{BCM} for this precise case), we know that
\begin{equation}\label{eq cv un sur P critique}
\E_\infty^{(1)} \frac{i}{P_\infty(i)}
\mathop{\longrightarrow}\limits_{i\to \infty}
\E \frac{1}{\Upsilon^\uparrow_2(p_{\bf q})}
= \E \frac{1}{p_{\bf q} \Upsilon^\uparrow_2(1)}
=\frac{2}{\pi^2 p_{\bf q}}.
\end{equation}
Hence, since $\sum_{i=1}^n (\log i)/i \sim (\log n )^2/2$ as $n \to \infty$, to show \eqref{eq cv theta critique condition} it remains to check that
\begin{equation}\label{eq cv log p sur p}
\frac{1}{(\log n)^2}\sum_{i=1}^n (\log i)\left( \frac{1}{P_\infty(i)}
-\E^{(1)}_\infty \left(\frac{1}{P_\infty(i)}\right)
\right)
\mathop{\longrightarrow}\limits_{n\to \infty}^{(\mathrm{L}^1)}
0.
\end{equation}
We prove \eqref{eq cv log p sur p} using the same techniques as in the proof of Proposition 3 in \cite{BCM}. For all $i\ge 1$ and $\vp>0$, as in \cite{BCM}, we set
\begin{eqnarray*}
X_i^{(\vp)} &\coloneqq &
\frac{1}{P_\infty(i)}{\bf 1}_{\vp \le P_\infty(i)/i \le 1/\vp}- \E^{(1)}_\infty \left( \frac{1}{P_\infty(i)}{\bf 1}_{\vp \le P_\infty(i)/i \le 1/\vp} \right)
\\
Y_i^{(\vp)} &\coloneqq &
\frac{1}{P_\infty(i)}{\bf 1}_{P_\infty(i)/i \not\in [\vp, 1/\vp]}- \E^{(1)}_\infty \left(\frac{1}{P_\infty(i)}{\bf 1}_{P_\infty(i)/i \not\in [\vp, 1/\vp]}\right).
\end{eqnarray*}
Then, to prove \eqref{eq cv log p sur p} it is sufficient to show that
\begin{equation}\label{eq cv X Y epsilon}
\lim_{\vp \to 0}\limsup_{n\to \infty}\E^{(1)}_\infty\left(
\left(
\frac{1}{(\log n)^2}\sum_{i=1}^n (\log i)X^{(\vp)}_i
\right)^2
\right)
=\lim_{\vp \to 0} \limsup_{n\to \infty}
\E^{(1)}_\infty
\left\vert
\frac{1}{(\log n)^2}\sum_{i=1}^n (\log i)Y^{(\vp)}_i
\right\vert
=0
\end{equation}
The second term is zero since
$$
\E^{(1)}_\infty \left\vert i Y^{(\vp)}_i \right\vert 
\le  2 \E^{(1)}_\infty \left(
\frac{i}{P_\infty(i)} {\bf 1}_{P_\infty(i)/i \not\in [\vp, 1/\vp]}
\right)
$$
and since $(i/P_\infty(i))_{i\ge 1}$ is uniformly integrable by Lemma \ref{lemme uniforme integrabilite} and is tight in $\R_+^*$ by \eqref{cv perimetre}. 
For the first term of \eqref{eq cv X Y epsilon}, let $A_n$ be an increasing sequence tending to infinity such that $\log A_n = o(\log n)$. We compute
\begin{align}
\! \!\!\!\E^{(1)}_\infty\left(
\left(
\frac{1}{(\log n)^2}\sum_{i=1}^n (\log i)X^{(\vp)}_i
\right)^2
\right)
\le
&\frac{1}{(\log n )^4} \sum_{i=1}^n (\log i )^2 \E^{(1)}_\infty\left(\left(X^{(\vp)}_i\right)^2\right) \label{premiere ligne X epsilon}\\
&+ \frac{2}{(\log n)^4} \sum_{i=1}^n \sum_{j=i+1}^{A_ni} 
(\log i)(\log j)
\left\vert \E^{(1)}_\infty \left(X^{(\vp)}_i X^{(\vp)}_j \right) \right\vert \label{deuxieme ligne X epsilon} \\
&+\frac{2}{(\log n)^4} \sum_{i=1}^n \sum_{j=A_n i+1}^{n} 
(\log i)(\log j)
\left\vert \E^{(1)}_\infty \left(X^{(\vp)}_i X^{(\vp)}_j \right) \right\vert. 	  \label{troisieme ligne X epsilon}
\end{align}
The quantity on the right of the first line \eqref{premiere ligne X epsilon} is a $O(1/(\log n)^4)$ as $n\to \infty$ since $|X_i^{(\vp)}|\le 2/(\vp i)$ by definition. Using that $\sum_{j=i+1}^{A_n i }( \log j)/j \le C(\log (A_ni)^2 - \log (i)^2) =C( (\log A_n)^2 + 2 (\log A_n)( \log i))$ for some constant $C>0$, the second line \eqref{deuxieme ligne X epsilon} is {bounded from above} by
$$O\!\left(\frac{1}{(\log n)^4}\sum_{i=1}^n \sum_{j=i+1}^{A_ni} (\log i)(\log j)\frac{1}{ij} \right) 
=O\!\left(
\frac{\log A_n}{(\log n)^4} \sum_{i=1}^n \frac{2(\log i)^2 + (\log A_n)(\log i)}{i}
\right)
=
 O\!\left(\frac{\log A_n}{\log n}\right).$$
To deal with the third line \eqref{troisieme ligne X epsilon}, it suffices to combine Lemma \ref{Lemme decorrelation} together with the fact that $\sum_{i=1}^n (\log i)/i = O((\log n)^2)$. This proves the convergence \eqref{eq cv X Y epsilon}, hence \eqref{eq cv log p sur p} and thus the convergence \eqref{eq cv theta critique condition}.

To conclude the case $a=2$ it remains to prove \eqref{eq cv theta critique} which is implied by the following convergence:
\begin{equation}\label{cv theta moins condition}
\frac{1}{(\log n)^2}\sum_{i=1}^n \left( \theta_i -\E^{(1)}_\infty(\theta_i \vert P_\infty(i))\right)
\mathop{\longrightarrow}\limits_{n\to \infty}^{(\mathrm{L}^2)}
0.
\end{equation}
To prove \eqref{cv theta moins condition}, by a straightforward computation, one can see that
\begin{align}
\E^{(1)}_\infty \left(
\left(\sum_{i=1}^n \left(
\theta_i-\E^{(1)}_\infty(\theta_i \vert P_\infty(i))
\right)
\right)^2
\right) &=
\sum_{i=1}^n
\E^{(1)}_\infty \left(
\left(
\theta_i-\E^{(1)}_\infty(\theta_i \vert P_\infty(i))
\right)^2
\right) \notag \\
&=
\sum_{i=1}^n
\left(
\E^{(1)}_\infty \left(
\theta_i^2 \right)
-
\E^{(1)}_\infty\left(
\E^{(1)}_\infty(\theta_i \vert P_\infty(i))^2
\right)
\right)  \notag \\
&\le 
\sum_{i=1}^n
\E^{(1)}_\infty \left(
\theta_i^2 \right) \label{carre martingale theta}
\\
&= O(\log n), \notag
\end{align}
where the last line comes from the fact that, by (\ref{queue nu}) and the fact that $h^\uparrow(p) \sim 2\sqrt{p/\pi}$, as $p\to \infty$,
$$
\E^{(1)}_\infty (\theta_i^2| P_\infty(i)=p)
= \sum_{k\ge 0} \left(\frac{2k+1}{2(p+k)}\right)^2 \frac{h^\uparrow(p+k)}{h^\uparrow(p)} \nu(k)
= O \left( \frac{1}{\sqrt{p}} \sum_{k=1}^\infty \frac{1}{(p+k)^{3/2}} \right)
=O\left(\frac{1}{p} \right),
$$
and from the fact that $\E^{(1)}_\infty (1/P_\infty(i)) = O(1/i)$ as $i\to \infty$ by \eqref{eq cv un sur P critique}.
\end{proof}
\subsubsection{Dense case}

\begin{proof}[Proof of Theorem \ref{th nombre de faces} in the dense case]
For the dense case $a \in (3/2,2)$, the proof is analogous to the case $a=2$. By the same reasoning as in the dilute case, it is enough to prove that
\begin{equation}\label{eq cv theta critique dense}
	\frac{1}{\log n}\sum_{i=1}^n \theta_i
	\mathop{\longrightarrow}\limits_{n\to \infty}^{(\mathrm{L}^1)}
	\frac{2}{\pi \tan((2-a) \pi)}.
\end{equation}
Let us first show that
\begin{equation}\label{eq cv theta dense condition}
	\frac{1}{\log n}\sum_{i=1}^n \E^{(1)}_\infty(\theta_i \vert P_\infty(i))
	\mathop{\longrightarrow}\limits_{n\to \infty}^{(\mathrm{L}^1)}
	\frac{2}{\pi \tan((2-a) \pi)}.
\end{equation}
By splitting the sum in \eqref{eq expression theta condition} before and after $\log i$ and using the fact that $h^\uparrow(p) \sim 2\sqrt{p/\pi}$ as $p \to \infty$, we get that there exists a constant $C>0$ such that for all $i \ge0$ for all $p\ge 1$,
$$
\left\vert
\E^{(1)}_\infty(\theta_i \vert P_\infty(i)=p) -
 \sum_{k=\lfloor \log i \rfloor}^\infty
\frac{2k+1}{2(p+k)} \frac{h^\uparrow(p+k)}{h^\uparrow(p)} \nu(k)
\right\vert
\le  C \frac{(\log i )^2}{p}.
$$
As a result, by Lemma \ref{lemme uniforme integrabilite}, as $i \to \infty$, 
\begin{equation}\label{eq reste esperance conditionelle theta}
\E^{(1)}_\infty
\left\vert
\E^{(1)}_\infty(\theta_i \vert P_\infty(i)) -
\sum_{k=\lfloor \log i \rfloor}^\infty
\frac{2k+1}{2(P_\infty(i)+k)} \frac{h^\uparrow(P_\infty(i)+k)}{h^\uparrow(P_\infty(i))} \nu(k)
\right\vert
\le  C \E^{(1)}_\infty \left( \frac{(\log i )^2}{P_\infty(i)}
\right)
= O\left(
\frac{(\log i )^2}{i^{1/(a-1)}}
\right).
\end{equation}
Moreover, by (\ref{queue nu}), since $h^\uparrow(p) \sim 2\sqrt{p/\pi}$ as $p\to \infty$, since $P_\infty(i)$ goes to $\infty$ as $i \to \infty$ in probability and using a Riemann sum, we have 
\begin{equation}\label{equivalent esperance conditionnelle theta dense}
\sum_{k=\lfloor \log i \rfloor}^\infty
\frac{2k+1}{2(P_\infty(i)+k)} \frac{h^\uparrow(P_\infty(i)+k)}{h^\uparrow(P_\infty(i))} \nu(k)
\mathop{\sim}\limits_{i \to \infty}^{(\P)}
\frac{1}{P_\infty(i)^{a-1}} \int_{0}^\infty \frac{u}{\sqrt{1+u}} \frac{p_{\bf q} \cos(a \pi)}{ u^a} du.
\end{equation}
By (\ref{queue nu}) and since $h^\uparrow(p) \sim 2\sqrt{p/\pi}$ as $p\to \infty$, we can also {bound from above} for all $i,p\ge 1$,
\begin{equation}\label{eq majoration esperance conditionnelle theta dense}
\sum_{k=0}^\infty
\frac{2k+1}{2(P_\infty(i)+k)} \frac{h^\uparrow(P_\infty(i)+k)}{h^\uparrow(P_\infty(i))} \nu(k)
\le 
\frac{C_{\bf q}}{P_\infty(i)^{a-1}} \int_{0}^\infty \frac{u}{\sqrt{1+u}} \frac{p_{\bf q} \cos(a \pi)}{ u^a} du,
\end{equation}
for some constant $C_{\bf q}>0$, so that by applying Lemma \ref{lemme uniforme integrabilite} and \eqref{eq reste esperance conditionelle theta} we obtain that $(i \E^{(1)}_\infty(\theta_i \vert P_\infty(i)))_{i \ge 1}$ is uniformly integrable. 
Thus, by \eqref{equivalent esperance conditionnelle theta dense} and \eqref{cv perimetre}, we deduce that
\begin{equation}\label{eq equivalent L1 esperance conditionnelle theta dense}
i \left( \E^{(1)}_\infty(\theta_i \vert P_\infty(i)) -  \frac{1}{P_\infty(i)^{a-1}} p_{\bf q} \cos (a\pi)  \int_0^\infty \frac{1}{u^{a-1} \sqrt{1+ u}} du \right)
\mathop{\longrightarrow}\limits_{i\to \infty}^{(\mathrm{L}^1)} 0.
\end{equation}
Moreover, by the change of variable $t=1/(1+u)$, the relation between the Beta and the Gamma functions and the fact that $\Gamma(1/2)= \sqrt{\pi}$, 
$$
\int_0^\infty \frac{1}{u^{a-1} \sqrt{1+u}} du  = \frac{\Gamma(2-a) \Gamma(a-3/2)}{\sqrt{\pi}}.
$$
By Euler's reflection formula,
$$
\frac{2\Gamma(a)}{(a-1) \pi^{3/2} \Gamma(a-3/2)} 
\frac{\Gamma(2-a) \Gamma(a-3/2)}{\sqrt{\pi}} 
=\frac{-2}{\pi \sin(\pi a)}.
$$
So by Remark \ref{remarque cv 1 sur P} and \eqref{eq equivalent L1 esperance conditionnelle theta dense}, we obtain that
\begin{equation}\label{eq cv esp theta}
i\E^{(1)}_\infty(\theta_i)
\mathop{\sim}\limits_{i \to \infty}
i \E^{(1)}_\infty \left( \frac{1}{P_\infty(i)^{a-1}} p_{\bf q} \cos(a \pi) \int_0^\infty \frac{1}{u^{a-1} \sqrt{1+u}} du\right)
\mathop{\longrightarrow}_{i \to \infty}
\frac{2}{\pi \tan((2-a) \pi)}.
\end{equation}
Hence, by \eqref{eq equivalent L1 esperance conditionnelle theta dense}, to show \eqref{eq cv theta dense condition} it remains to check that
\begin{equation}\label{eq cv 1 sur P a moins moyenne}
	\frac{1}{\log n}\sum_{i=1}^n \left( \frac{1}{P_\infty(i)^{a-1}}
	-\E^{(1)}_\infty \left(\frac{1}{P_\infty(i)^{a-1}}\right)
	\right)
	\mathop{\longrightarrow}\limits_{n\to \infty}^{(\mathrm{L}^1)}
	0,
\end{equation}
which follows from truncations which are similar to the case $a=2$. To prove \eqref{eq cv 1 sur P a moins moyenne}, for all $i\ge 1$ and $\vp>0$, similarly to \cite{BCM}, we set
\begin{eqnarray*}
	X_i^{(\vp)} &\coloneqq &
	\frac{1}{P_\infty(i)^{a-1}}{\bf 1}_{\vp \le P_\infty(i)^{a-1}/i \le 1/\vp}- \E^{(1)}_\infty \left( \frac{1}{P_\infty(i)^{a-1}}{\bf 1}_{\vp \le P_\infty(i)^{a-1}/i \le 1/\vp} \right)
	\\
	Y_i^{(\vp)} &\coloneqq &
	\frac{1}{P_\infty(i)^{a-1}}{\bf 1}_{P_\infty(i)^{a-1}/i \not\in [\vp, 1/\vp]}- \E^{(1)}_\infty \left(\frac{1}{P_\infty(i)^{a-1}}{\bf 1}_{P_\infty(i)^{a-1}/i \not\in [\vp, 1/\vp]}\right).
\end{eqnarray*}
Then, to prove \eqref{eq cv 1 sur P a moins moyenne} it is sufficient to show that
\begin{equation}\label{eq cv X Y epsilon dense}
	\lim_{\vp \to 0}\limsup_{n\to \infty}\E^{(1)}_\infty\left(
	\left(
	\frac{1}{\log n}\sum_{i=1}^n X^{(\vp)}_i
	\right)^2
	\right)
	=\lim_{\vp \to 0} \limsup_{n\to \infty}
	\E^{(1)}_\infty
	\left\vert
	\frac{1}{\log n}\sum_{i=1}^n Y^{(\vp)}_i
	\right\vert
	=0
\end{equation}
The second term of \eqref{eq cv X Y epsilon dense} is zero since
$$
	\E^{(1)}_\infty \left\vert i Y^{(\vp)}_i \right\vert 
	\le  2 \E^{(1)}_\infty \left(
	\frac{i}{P_\infty(i)^{a-1}} {\bf 1}_{P_\infty(i)^{a-1}/i \not\in [\vp, 1/\vp]}
	 \right)
$$
and since $(i/P_\infty(i)^{a-1})_{i\ge 1}$ is uniformly integrable by Lemma \ref{lemme uniforme integrabilite} and is tight by \eqref{cv perimetre}. 
For the first term of \eqref{eq cv X Y epsilon dense}, let $A_n$ be an increasing sequence tending to infinity such that $\log A_n = o(\log n)$. One bounds from above
\begin{align}
	\E^{(1)}_\infty\left(
	\left(
	\frac{1}{\log n}\sum_{i=1}^n X^{(\vp)}_i
	\right)^2
	\right)
	\le
	&\frac{1}{(\log n )^2} \sum_{i=1}^n  \E^{(1)}_\infty\left(\left(X^{(\vp)}_i\right)^2\right) \label{premiere ligne X epsilon dense}\\
	&+ \frac{2}{(\log n)^2} \sum_{i=1}^n \sum_{j=i+1}^{A_ni} 
	\left\vert \E^{(1)}_\infty \left(X^{(\vp)}_i X^{(\vp)}_j \right) \right\vert \label{deuxieme ligne X epsilon dense} \\
	&+\frac{2}{(\log n)^2} \sum_{i=1}^n \sum_{j=A_n i+1}^{n} 
	\left\vert \E^{(1)}_\infty \left(X^{(\vp)}_i X^{(\vp)}_j \right) \right\vert. \label{troisieme ligne X epsilon dense}
\end{align}
Since for all $i\ge 1$ we have $\vert X_i^{(\vp)} \vert\le 2/(i\vp)$, the first term \eqref{premiere ligne X epsilon dense} is a $O(1/(\log n)^2)$, the second line \eqref{deuxieme ligne X epsilon dense} is a $O((\log A_n)/(\log n))$ and the third one is controlled using Lemma \ref{Lemme decorrelation}. This proves \eqref{eq cv X Y epsilon dense}, hence \eqref{eq cv 1 sur P a moins moyenne} and thus \eqref{eq cv theta dense condition}. To end the proof of Theorem \ref{th nombre de faces}, it is enough to check that
\begin{equation}\label{cv theta moins condition dense}
	\frac{1}{\log n}\sum_{i=1}^n \left( \theta_i -\E^{(1)}_\infty(\theta_i \vert P_\infty(i))\right)
	\mathop{\longrightarrow}\limits_{n\to \infty}^{(\mathrm{L}^2)}
	0,
\end{equation}
which implies \eqref{eq cv theta critique dense}. By \eqref{carre martingale theta} and by \eqref{eq cv esp theta}, we know that
$$
\E^{(1)}_\infty \left(
\left(\sum_{i=1}^n \left(
\theta_i-\E^{(1)}_\infty(\theta_i \vert P_\infty(i))
\right)
\right)^2
\right)  \le \sum_{i=1}^n \E^{(1)}_\infty (\theta_i^2) \le  \sum_{i=1}^n \E^{(1)}_\infty \theta_i = O(\log n),
$$
hence \eqref{cv theta moins condition dense}, thus ending the proof.
\end{proof}

\subsection{Application: comparison of fpp and graph distance in the dilute case}
We apply our results on the fpp geodesics to show Corollary \ref{cor inclusion des boules dilue} which states that in the dilute case, fpp balls are included in balls for the graph distance with a suitable radius.

\begin{proof}[Proof of Corollary \ref{cor inclusion des boules dilue}]
Let $\vp \in (0,1)$. We reason conditionally on the perimeter process $P_\infty$. {Recall from Subsection \ref{sous-section explo uniforme} that $(\overline{\mathfrak{e}}_n)_{n\ge 0}$ is the uniform exploration of the map.} Let $e$ be an edge on $\partial\overline{\mathfrak{e}}_n$, then, by \eqref{proba coalescence}, the number of faces in the fpp geodesic from $e$ to the root face $f_r$ in $\mathfrak{m}_\infty$ has the same law as 
$
\sum_{i=0}^{n-1} X_i
$ 
where the $X_i$'s are i.i.d.\@ Bernoulli random variables of parameters $\theta_i$. Thus, by a union bound on $e$,
\begin{align*}
\P^{(1)}_\infty
&\left(
\left.
\overline{\mathfrak{e}}_n \not\subseteq \overline{\mathrm{Ball}}^\dagger_{\lfloor (1+\vp)
\sum_{i=0}^{n-1} \theta_i	
\rfloor}(\mathfrak{m}_\infty)
\right\vert
P_\infty
\right)
\\
&\le 
2P_\infty(n)
\P^{(1)}_\infty
\left(
\left.
\sum_{i=0}^{n-1} X_i
\ge 
(1+\vp)
\sum_{i=0}^{n-1} \theta_i
\right\vert
P_\infty
\right) \\
&\le 4 P_\infty(n) \exp\left(
-c
\vp^2
\sum_{i=0}^{n-1} \theta_i
\right) \qquad \text{(Chernoff)}
\\
&\le 4P_\infty(n) \exp\left(-c\vp^2 \sum_{i=0}^{n-1} \frac{1}{2P_\infty(i)+1} \right)  \\
&\le
4 n^{\frac{1+\vp}{a-1}}
\exp \left(
-\frac{c \vp^2 }{3}
n^{\frac{a-2-\vp}{a-1}}
\right)
\end{align*}
a.s.\@ for all $n$ large enough, where $c$ is some absolute constant from Chernoff's inequality and where the last inequality comes from the fact that for all $\vp>0$, a.s.\@ for $n$ large enough we have 
\begin{equation}\label{eq encadrement p s perimetre}
	n^{\frac{1-\vp}{a-1}} \le P_\infty(n) \le n^{\frac{1+\vp}{a-1}}
\end{equation} 
by Lemma 10.9 of \cite{StFlour}. 

By summing over $n$, by Borel-Cantelli's lemma, we deduce that almost surely, for all $\vp>0$, for all $n$ large enough, we have 
$$
\overline{\mathfrak{e}}_n \subseteq \overline{\mathrm{Ball}}^\dagger_{\lfloor (1+\vp)
\sum_{i=0}^{n-1} \theta_i	
\rfloor}(\mathfrak{m}_\infty).
$$
Next, let us prove that almost surely,
\begin{equation}\label{eq equivalents ps somme des theta}
\sum_{i=0}^{n-1} \theta_i \mathop{\sim}\limits_{n \to \infty}
\frac{e_{\bf q}}{2}
\sum_{i=0}^{n-1} \frac{1}{P_\infty(i)}
\mathop{\sim}\limits_{n \to \infty}
e_{\bf q}
\sum_{i=0}^{n-1} \frac{\mathcal{E}_i}{2P_\infty(i)}.
\end{equation}
Let $\vp>0$. For the first equivalent of \eqref{eq equivalents ps somme des theta}, one can see that we have uniformly in $i \ge 0$, for all $p \ge 1$,
\begin{align*}
\E^{(1)}_\infty \left( \left.
\exp
\left(
\theta_i - \E^{(1)}_\infty (\theta_i \vert P_\infty(i))
\right)
\right\vert
P_\infty(i) =p
\right)
&\le 1 + 
C\E^{(1)}_\infty \left( \left.
\left(
\theta_i - \E^{(1)}_\infty (\theta_i \vert P_\infty(i))
\right)^2
\right\vert
P_\infty(i) 
=p
 \right)\\
 &\le 1+ C'\frac{1}{p^{a-1}}
\end{align*}
for some constants $C,C'>0$, where the first inequality comes from the fact that $\theta_i \in [0,1]$ and the second one stems from \eqref{majoration theta carre}. As a result, 
$$
\E^{(1)}_\infty\left(
\exp \left(
\sum_{i=0}^{n-1} \left(
\theta_i - \E^{(1)}_\infty(
\theta_i \vert P_\infty(i))
-C' \frac{1}{P_\infty(i)^{a-1}}
\right)
\right)
\right) \le 1
$$
Thus, for all $\delta>0$, we have
$$
\P^{(1)}_\infty
\left(
\sum_{i=0}^{n-1} \left(
\theta_i - \E^{(1)}_\infty(
\theta_i \vert P_\infty(i))
-C' \frac{1}{P_\infty(i)^{a-1}}
\right)
\ge
\delta n^{\frac{a-2}{a-1}-\vp}
\right)
\le  \exp\left( -\delta n^{\frac{a-2}{a-1}-\vp} \right).
$$
So, taking $\vp<(a-2)/(a-1)$, by Borel-Cantelli's lemma, we have a.s. as $n\to \infty$,
$$
\sum_{i=0}^{n-1} \left(
\theta_i - \E^{(1)}_\infty(
\theta_i \vert P_\infty(i))
-C' \frac{1}{P_\infty(i)^{a-1}}
\right)
\le
o\left(
n^{\frac{a-2}{a-1}-\vp} 
\right).
$$
Then, applying \eqref{eq encadrement p s perimetre}, we deduce that a.s.\@ $\sum_{i=0}^n 1/P_\infty(i)^{a-1}  = O (\sum_{i=1}^n 1/i^{1-\vp})=O(n^{\vp} )$, so
$$
\sum_{i=0}^{n-1} \left(
\theta_i - \E^{(1)}_\infty(
\theta_i \vert P_\infty(i))
\right)
\le
o\left(
n^{\frac{a-2}{a-1}-\vp} 
\right).
$$
By the same reasoning, we also get almost surely as $n \to \infty$
$$
-\sum_{i=0}^{n-1} \left(
\theta_i - \E^{(1)}_\infty(
\theta_i \vert P_\infty(i))
\right)
\le
o\left(
n^{\frac{a-2}{a-1}-\vp} 
\right).
$$
{Moreover}, by \eqref{eq encadrement p s perimetre}, we have a.s.\@ $\sum_{i=0}^\infty 1/P_\infty(i) = \infty$ so that by \eqref{eq esp conditionnelle proche de un sur P}, we deduce that
$$
\sum_{i=0}^{n-1} \E^{(1)}_\infty(\theta_i \vert P_\infty(i)) \mathop{\sim}\limits_{n \to \infty}^{(\mathrm{a.s.})} \sum_{i=0}^{n-1} \frac{e_{\bf q}}{2} \frac{1}{P_\infty(i)},
$$
hence the first asymptotic in \eqref{eq equivalents ps somme des theta} using that $n^{\frac{a-2}{a-1}-\vp} = o(\sum_{i=0}^{n-1} 1/P_\infty(i))$ a.s.\@ by applying \eqref{eq encadrement p s perimetre} once more. For the second asymptotic in \eqref{eq equivalents ps somme des theta}, we apply Bernstein's concentration inequality for exponential random variables (see e.g.\@ Corollary 2.10 of \cite{Ta94}
) and get that for $\vp>0$ small enough, a.s.\@ for all $\delta>0$, for $n$ large enough,
\begin{align*}
\P^{(1)}_\infty \left( \left. \left| \sum_{i=0}^{n-1} \frac{\cal{E}_i-1}{P_\infty(i)} \right| > \delta n^{\frac{a-2}{a-1}-\vp}  \right| P_\infty \right) 
&\le
2\exp \left(-C''\min \left(\frac{\delta^2 n^{2\frac{a-2}{a-1}-2\vp} }{\sum_{i=0}^{n-1} {1}/{P_\infty(i)^2}}, \frac{\delta n^{\frac{a-2}{a-1}-\vp} }{\max_{0\le i \le n-1} {1}/{P_\infty(i)}}\right)\right) 
\\
&\le 2\exp \left(-
C''\delta n^{\frac{a-2}{a-1}-\vp} 
\right),
\end{align*}
where $C''>0$ is an absolute constant and where the second inequality comes from \eqref{eq encadrement p s perimetre} and from the fact that $2/(a-1)\ge 4/3>1$ since $a\le 5/2$ so that a.s.\@ $\sum_{i=0}^\infty 1/P_\infty(i)^2 < \infty$, hence the second equivalent of \eqref{eq equivalents ps somme des theta} by Borel-Cantelli's lemma and by \eqref{eq encadrement p s perimetre}.
The almost sure asymptotic \eqref{eq equivalents ps somme des theta} ends the proof since $\overline{\mathfrak{e}}_n$ is the hull of the fpp ball of radius $\sum_{i=0}^{n-1} {\mathcal{E}_i}/{(2P_\infty(i))}$.
\end{proof}

We next turn to the proof of Corollary \ref{cor inclusion des boules cas Cauchy}
\begin{proof}[Proof of Corollary \ref{cor inclusion des boules cas Cauchy}]
Recall that $\overline{\mathfrak{e}}_n$ is the hull of the fpp ball of radius $T_n=\sum_{i=0}^{n-1} {\mathcal{E}_i}/{(2P_\infty(i))}$. The result follows from the same lines as in the beginning of the proof of Corollary \ref{cor inclusion des boules dilue}, using \eqref{eq cv theta critique} and using the fact that 
$T_n/ \log n$ converges in probability towards $1/(p_{\bf q} \pi^2)$ by Proposition 3 of \cite{BCM}.
\end{proof}
\subsection{Second application: the diameter of dense maps is logarithmic for the dual graph distance}
We prove in this subsection Theorem \ref{th diametre cas dense} which states that in the case $a \in (3/2,2)$, the diameter of random maps of law $\P^{(\ell)}$ grows at rate at most $\log \ell$. It is not a consequence of Theorem \ref{th nombre de faces}, but it relies on its proof, more precisely on the inequality \eqref{eq majoration esperance conditionnelle theta dense}. 
The proof relies on the martingales introduced in \cite{Kam23}. Before entering the proof of the theorem, we recall the notion of maps with a target face. A planar map with a target face $\mathfrak{m}_\square$ is a planar map with a distinguished face $\square$ which is not the root face $f_r$. For all $\ell, p \ge 1$, we denote by $\mathcal{M}^{(\ell)}_p$ the set of planar maps of perimeter $2\ell$ with a target face of degree $2p$. The weight $w_{\bf q}(\mathfrak{m}_\square)$ of a map $\mathfrak{m}_\square \in \mathcal{M}^{(\ell)}_p$ is defined by
$$w_{\bf q}(\mathfrak{m}_\square)= \prod_{f \in \mathrm{Faces}(\mathfrak{m}_\square) \setminus \{f_r, \square\}} q_{\deg (f)/2}.$$
The partition function of such maps is given by Equation (3.9) of \cite{StFlour} for all $\ell, p \ge 1$:
$$
W^{(\ell)}_p \coloneqq \sum_{\mathfrak{m}_\square \in \mathcal{M}^{(\ell)}_p} w_{\bf q}(\mathfrak{m}_\square)
=
\frac{1}{2} h^\downarrow_p(\ell) c_{\bf q}^{\ell+p},
$$
where 
\begin{equation}\label{eq h p}
h^\downarrow_p(\ell) = h^\downarrow(\ell) h^\downarrow(p) \frac{\ell}{\ell +p} \qquad \text{with } \qquad 
h^\downarrow(\ell)  = \frac{1}{2^{2\ell}} \binom{2\ell}{\ell}.
\end{equation}
We also set by convention $h^\downarrow_p(\ell) = {\bf 1}_{\ell=-p}$ for all $\ell \le 0$. We denote by $\P^{(\ell)}_p$ the associated Boltzmann probability measure. The peeling exploration of a map with a target face is defined in the same way as for maps without target face except that at each gluing event, we choose to fill in the hole which does not contain the target face. When we discover the target face, the exploration ends and by convention, we say that the perimeter is absorbed at $-p$. The perimeter process under $\P^{(\ell)}_p$, which is denoted by $(P_p(n))_{n\ge 0}$ is then a Doob $h^\downarrow_p$-transform of the $\nu$-random walk starting at $\ell$ which is absorbed when it reaches $-p$ by Proposition 4.7 of \cite{StFlour}. Its law can also be interpreted as a $\nu$-random walk starting at $\ell$ conditioned to stay positive until it jumps and dies at $-p$ by Proposition 5.3 of \cite{StFlour}.

Next, we introduce two supermartingales.
\begin{lemma}\label{lemme martingale 1}
	Suppose ${\bf q}$ is of type $a \in (3/2,2)$. Let $\ell \ge 1$. Under $\P^{(1)}_\ell$, conditionally on $P_\ell$, let $(X_i)_{i\ge 0}$ be a family of independent Bernoulli random variables of parameters $(2\Delta P_\ell(i)+1)/(2(P_\ell(i+1))){\bf 1}_{\Delta P_\ell(i) \ge 0 \text{ and } P_\ell(i) \ge 1}$. Then, there exists a constant $C_{\bf q}'$ such that for all $\lambda \in (0,1)$, the process $(M_n^{(\ell,\lambda)})_{n\ge 0}$ defined by
	$$
	\forall n \ge 0, \qquad
	M_n^{(\ell,\lambda)} =
	\exp\left(\lambda \left(
	\sum_{i=0}^{n-1} X_i
	-C'_{\bf q} \sum_{i=0}^{n-1} \frac{1}{P_\ell(i)^{a-1}}{\bf 1}_{P_\ell(i) \ge 1}
	\right)
	\right)
	$$
	is a supermartingale with respect to the filtration $(\mathcal{G}_n)_{n\ge 0}$ defined by
	$$
	\forall n \ge 0, \qquad
	\mathcal{G}_n =
	\sigma\left(
	(P_\ell(i))_{0\le i \le n}, (X_i)_{0\le i \le n-1}
	\right).
	$$
\end{lemma}
\begin{proof}
	The process is clearly adapted. Moreover, by \eqref{eq h p} and \eqref{eq h fleche}, for all $p\ge 1$, we have $h^\downarrow_\ell(p) = h^\downarrow(\ell) h^\uparrow(p)/(2(\ell+p))$, so that for all $n\ge 0$, 
	\begin{align*}
		\E^{(1)}_\ell\left(
		\left.
		\exp(\lambda X_n)
		\right\vert
		\mathcal{G}_{n}
		\right)
		&=
		\E^{(1)}_\ell \left(
		\left.
		(e^\lambda-1) \frac{2 \Delta P_\ell(n)+1}{2P_\ell(n+1)} {\bf 1}_{\Delta P_\ell(n)\ge 0 \text{ and } P_\ell(n) \ge 1}
		+1
		\right\vert
		\mathcal{G}_{n}
		\right) \\
		&= 1 + (e^\lambda -1)
		\sum_{k= 0}^\infty
		\frac{2k+1}{2P_\ell(n)+2k} \frac{h^\downarrow_\ell(P_\ell(n)+k)}{h^\downarrow_\ell(P_\ell(n))} \nu(k) {\bf 1}_{P_\ell(n) \ge 1}\\
		& \le 1 + (e^\lambda -1)
		\sum_{k= 0}^\infty
		\frac{2k+1}{2P_\ell(n)+2k} \frac{h^\uparrow(P_\ell(n)+k)}{h^\uparrow(P_\ell(n))} \nu(k) {\bf 1}_{P_\ell(n) \ge 1}\\
		& \le 1 + C'_{\bf q} \frac{\lambda}{P_\ell(n)^{a-1}} {\bf 1}_{P_\ell(n) \ge 1} \qquad \text{by \eqref{eq majoration esperance conditionnelle theta dense}} \\
		& \le \exp \left( C'_{\bf q} \frac{\lambda}{P_\ell(n)^{a-1}}{\bf 1}_{P_\ell(n) \ge 1}\right),
	\end{align*}
where $C'_{\bf q}>0$ is a constant which depends on $\bf q$, hence the desired result.
\end{proof}
We then introduce a second supermartingale, which is closely related to the martingale (6.1) in \cite{Kam23} and corresponds the the supermartingale of Lemma 6.2 in \cite{Kam23} in the case $a=2$.
\begin{lemma}\label{lemme martingale 2}
	Suppose ${\bf q}$ is of type $a \in (3/2,5/2]$. Let $\lambda>0$ such that $\lambda<\min_{k,\ell \ge 1} k^{a-1} (\nu((-\infty,-k])-\nu(-k-\ell))$ (such a $\lambda$ exists thanks to \eqref{queue nu}). Then the process  $(\widetilde{M}_n^{(\ell,\lambda)})_{n\ge 0}$ defined by
	$$
	\forall n\ge 0, \qquad
	\widetilde{M}^{(\ell,\lambda)}_n =
	\frac{1}{h^\downarrow_\ell(P_\ell(n))}
	\exp\left(
	\lambda \sum_{i=0}^{n-1} \frac{1}{P_\ell(i)^{a-1}} {\bf 1}_{P_\ell(i) \ge 1}
	\right)
	$$
	is a supermartingale with respect to the natural filtration $(\mathcal{F}_n)_{n\ge 0}$ associated with the process $P_\ell$. 
\end{lemma}
\begin{proof}
The proof is the same as the proof of Lemma 6.2 in \cite{Kam23} and boils down to the fact that for all $n\ge 0$, by definition of $\lambda$ and since $\exp(x) \le 1/(1-x)$ for all $x \in [0,1)$, on the event $\{P_\ell(n) \ge 1\}$, we have
$$
\exp\left( 
\frac{\lambda}{P_\ell(n)^{a-1}}
\right)
\le \frac{1}{1-\lambda/P_\ell(n)^{a-1}}
\le \left( \nu([-P_\ell(n)+1,\infty))+ \nu(-P_\ell(n) -\ell) \right)^{-1}.
$$
It then suffices to use that $P_\ell$ is a Doob $h_\ell^\downarrow$-transform of the $\nu$-random walk.
\end{proof}

\begin{proof}[Proof of Theorem \ref{th diametre cas dense}]
It suffices to prove that there exists a constant $C_{\bf q}>0$ large enough such that if $E$ is a uniform random edge of the map, then
\begin{equation}\label{eq distance arete uniforme}
\P^{(\ell)}\left(
d^\dagger_\mathrm{gr} (f_r, E) \ge C_{\bf q} \log \ell 
\right)
\le \ell^{-2},
\end{equation}
where $d^\dagger_\mathrm{gr} (f_r, E)$ is defined as the smallest dual graph distance from $f_r$ to a face which is adjacent to $E$. Indeed,
\begin{align*}
	\P^{(\ell)} &\left(
	\exists e \in \mathrm{Edges}(\mathfrak{m}), \ d^\dagger_\mathrm{gr} (f_r, e) \ge C_{\bf q} \log \ell 
	\right)\\
	&\le
	\P^{(\ell)}\left(
	\exists e \in \mathrm{Edges}(\mathfrak{m}), \ d^\dagger_\mathrm{gr} (f_r, e) \ge C_{\bf q} \log \ell \text{ and } 
	\# \mathrm{Edges}(\mathfrak{m}) \le \ell^a
	\right)
	+\P^{(\ell)} \left(
	\# \mathrm{Edges}(\mathfrak{m}) \ge \ell^a
	\right) \\
	&\le
	\E^{(\ell)}
	\left(
	\# \{ e \in \mathrm{Edges}(\mathfrak{m}), \ 
	d^\dagger_\mathrm{gr} (f_r, e) \ge C_{\bf q} \log \ell 
	\} {\bf 1}_{\# \mathrm{Edges}(\mathfrak{m}) \le \ell^a}
	\right)
	+\P^{(\ell)} \left(
	\# \mathrm{Edges}(\mathfrak{m}) \ge \ell^a
	\right)
	\\
	&\le \ell^a\P^{(\ell)}\left( d^\dagger_\mathrm{gr} (f_r, E) \ge C_{\bf q} \log \ell  \right) 
	+ 
	\P^{(\ell)} \left(
	\# \mathrm{Edges}(\mathfrak{m}) \ge \ell^a
	\right)
	\mathop{\longrightarrow}\limits_{\ell\to \infty} 0,
\end{align*}
where the convergence comes from \eqref{eq distance arete uniforme} and from the fact that $\# \mathrm{Edges}(\mathfrak{m})/\ell^{a-1/2}$ converges in distribution under $\P^{(\ell)}$ as $\ell \to \infty$ by Proposition 10.4 of \cite{StFlour}.
In order to prove \eqref{eq distance arete uniforme}, one can see that by unzipping the uniform random edge, using Equation (3.4) of \cite{Kam23}, we have
\begin{align*}
\P^{(\ell)}\left(
d^\dagger_\mathrm{gr} (f_r, E) \ge C_{\bf q} \log \ell 
\right)
&= \frac{W_1^{(\ell)}}{W^{(\ell)}} 
\E^{(\ell)}_1 \left(
\frac{1}{\# \mathrm{Edges}(\mathfrak{m}_\square)-1} 
	{\bf 1}_{d^\dagger_\mathrm{gr} (f_r, \square) \ge C_{\bf q} \log \ell }
\right) \\
&\le \frac{W_1^{(\ell)}}{W^{(\ell)}} 
\P^{(\ell)}_1
\left(
d^\dagger_\mathrm{gr} (f_r, \square) \ge C_{\bf q} \log \ell
\right) \\
&= O(\ell^{a-1/2})
\P^{(1)}_\ell
\left(
d^\dagger_\mathrm{gr} (f_r, \square) \ge C_{\bf q} \log \ell
\right) ,
\end{align*}
where in the last line, we used that $W^{(\ell)}_1/W^{(\ell)} =\E^{(\ell)} \# \mathrm{Edges}(\mathfrak{m})=  O(\ell^{a-1/2})$ by Proposition 10.4 and Equation (10.8) of \cite{StFlour}. Note also that in the last line, we have exchanged the root and the target face. Thus, it is enough to show that
\begin{equation}\label{eq majoration proba loin de la cible}
	\P^{(1)}_\ell \left(
	d^\dagger_\mathrm{gr}(f_r, \square) \ge C_{\bf q} \log \ell \right)
	=O(\ell^{-4}).
\end{equation}
By the reversed uniform exploration, we know that the fpp geodesic from $f_r$ to $\square$ has 
$
\sum_{i= 0}^{\tau_{-\ell}-1} X_i
$ 
faces, where the $X_i$'s are independent from the perimeter process and are independent Bernoulli random variables of parameters ${\bf 1}_{\Delta P_\ell(i) \ge0}(2\Delta P_\ell(i)+1)/(2P_\ell(i+1))$ by \eqref{proba coalescence}, hence
$$
d^\dagger_\mathrm{gr}(f_r, \square)\le  \sum_{i= 0}^{\tau_{-\ell}-1} X_i.
$$
Let $\lambda >0$ satisfying the hypothesis of Lemma \ref{lemme martingale 2} and such that $\lambda /C'_{\bf q} <1$, where $C'_{\bf q}$ appears in Lemma \ref{lemme martingale 1}. Then, applying successively Markov's inequality, the fact that $h_\ell^\downarrow(-\ell)=1$ and Cauchy-Schwarz inequality,
\begin{align*}
\P^{(1)}_\ell \left(
\sum_{i=0}^{\tau_{-\ell}-1}
X_i
 \ge C_{\bf q} \log \ell
\right)
&\le 
\ell^{-\lambda C_{\bf q}/(2C'_{\bf q})} \E^{(1)}_\ell 
\exp
\left(
\frac{\lambda}{2C'_{\bf q}} \sum_{i=0}^{\tau_{-\ell}-1} X_i
\right)
\\
&=  \ell^{-\lambda C_{\bf q}/(2C'_{\bf q})}
\E^{(1)}_\ell \sqrt{M^{(\ell,\lambda/C'_{\bf q})}_{\tau_{-\ell}} \widetilde{M}^{(\ell, \lambda)}_{\tau_{-\ell}} }\\
&\le \ell^{-\lambda C_{\bf q}/(2C'_{\bf q})}
\sqrt{\E^{(1)}_\ell \left( M^{(\ell, \lambda/C'_{\bf q})}_{\tau_{-\ell}} \right)
\E^{(1)}_\ell \left( \widetilde{M}^{(\ell, \lambda)}_{\tau_{-\ell}} \right) 
}
\\
&\le  \ell^{-\lambda C_{\bf q}/(2C'_{\bf q})}
\sqrt{
\E^{(1)}_\ell \left( M^{(\ell, \lambda/C'_{\bf q})}_{0} \right)
\E^{(1)}_\ell \left( \widetilde{M}^{(\ell, \lambda)}_{0} \right)
}, 
\end{align*}
where the last inequality comes from Lemmas \ref{lemme martingale 1} and \ref{lemme martingale 2} and from the fact that $\tau_{-\ell}$ is a stopping time for the filtrations $(\mathcal{F}_n)_{n\ge 0}$ and $(\mathcal{G}_n)_{n\ge 0}$. Therefore, 
$$
\P^{(1)}_\ell \left(
d^\dagger_{\mathrm{gr}}(f_r, \square) \ge C_{\bf q} \log \ell
\right)
\le    \ell^{-\lambda C_{\bf q}/(2C'_{\bf q})} \frac{1}{\sqrt{h_\ell^\downarrow(1)}} = O\left( 
\ell^{-\lambda C_{\bf q}/(2C'_{\bf q})+1/2}
\right).
$$
This proves \eqref{eq majoration proba loin de la cible} by taking $C_{\bf q}$ large enough and thus completes the proof.
\end{proof}
\section{Scaling limit of the coalescing flow of geodesics}\label{section flot}
This section and the next one describe the scaling limit of the fpp geodesics to the root under $\P^{(\ell)}$. 
\subsection{The self-similar Markov process arising as the scaling limit of the perimeter process}\label{sous-section processus autosimilaire}
We recall the scaling limit of the perimeter process which is given in terms of positive self-similar Markov processes. These positive self similar Markov processes were first introduced and studied in \cite{BBCK}. As in p.\@ 40 of \cite{BBCK}, let $\Lambda$ be the image by $x \mapsto \log x$ of the measure {$\lambda$ given by}
$$
\lambda (dx) \coloneqq \frac{\Gamma(a)}{\pi} \left( 
\frac{1}{(x(1-x))^a} {\bf 1}_{1/2 < x < 1} + \cos(a \pi) \frac{1}{(x(x-1))^a} {\bf 1}_{x>1} \right) dx.
$$
Let $\xi$ be the L\'evy process with no Brownian part, with L\'evy measure $\Lambda$ and with drift
\begin{equation}\label{eq drift}
\frac{\Gamma(3-a)}{2\Gamma(4-2a) \sin (\pi (a-1))} + \frac{\Gamma(a)B_{1/2}(1-a, 3-a)}{\pi},
\end{equation}
where $B_{1/2}(x,y) = \int_0^{1/2} t^{x-1} (1-t)^{y-1} dt$ is the incomplete Beta function. As in \cite{BBCK}, for all $\alpha \in \R$, we define the Lamperti time substitution
$$\forall t \ge 0, \qquad \qquad
\tau_\alpha(t) = \inf \left\{ r \ge 0, \ \int_0^r e^{-\alpha \xi(s)} ds \ge t\right\}.$$
For all $x >0$, the self-similar Markov process $X^{(\alpha)}_a$ associated with $\xi$ of index $\alpha$ is defined by
$$
\forall t \ge 0, \qquad \qquad
X^{(\alpha)}_a(t)  = x e^{\xi(\tau_\alpha(t x^\alpha))}.
$$
Note that $X^{(0)}_a(t)  = x e^{\xi(t)}$. We can now describe the scaling limit of the perimeter process under $\P^{(\ell)}$ by Proposition 6.6 of \cite{BBCK}:
\begin{equation}\label{cv perimetre fini}
	\text{Under }\P^{(\ell)}, \qquad \qquad
	\left( \frac{P(\ell^{a-1}t)}{\ell} \right)_{t\ge 0}
	\mathop{\longrightarrow}\limits_{n \to \infty}
	\left(
	X_a^{(1-a)}(c_a p_{\bf q} t)
	\right)_{t\ge 0},
\end{equation} 
for the Skorokhod $J_1$ topology, where $c_a \coloneqq{\pi}/{\Gamma(a) }$. Recall the notation $T_n = \sum_{i=0}^{n-1} {\mathcal{E}_i}/({2P(i)})$ for all $n\ge 0$ from Subsection \ref{sous-section explo uniforme} where the $\mathcal{E}_i$'s are i.i.d.\@ exponential random variables of parameter $1$. Then, jointly with \eqref{cv perimetre fini}, for the topology of uniform convergence on compact sets,
\begin{equation}\label{eq cv Tn fini}
	\text{Under }\P^{(\ell)}, \qquad \qquad
	\left(\frac{T_{\lfloor \ell^{a-1} t \rfloor}}{\ell^{a-2}}
	\right)_{t\ge 0}
	\mathop{\longrightarrow}\limits_{\ell \to \infty}^{(\mathrm{d})}
	\left(\frac{1}{2 c_a p_{\bf q}} \int_0^{c_a p_{\bf q} t} \frac{1}{X^{(1-a)}_a(s)}ds\right)_{t\ge 0}.
\end{equation}
In particular, combining \eqref{cv perimetre fini} with \eqref{eq cv Tn fini} gives the scaling limit of the perimeter of the fpp balls. For all $r\ge 0$, let $\Theta(r)= \inf \{ n \ge 0, \ T_n \ge r\}$. Then we have jointly with \eqref{cv perimetre fini} and \eqref{eq cv Tn fini}
$$
\text{Under }\P^{(\ell)}, \qquad \qquad
\left(
\frac{1}{2\ell}\left\vert\partial\overline{\mathrm{Ball}}^{\mathrm{fpp}}_{\ell^{a-2}r}(\mathfrak{m}) \right\vert
\right)
=
\left(\frac{P(\Theta(\ell^{a-2}r))}{\ell}
\right)_{r\ge 0}
\mathop{\longrightarrow}\limits_{\ell \to \infty}^{(\mathrm{d})}
\left( X^{(2-a)}_a(2c_a p_{\bf q} r)\right)_{r\ge 0}.
$$
Actually, in order to introduce our coalescing flow of diffusions, we will also work at a different time-scale, so that the scaling limit of the perimeter process is $X^{(0)}_a$. Let ${\widetilde{P}=}(\widetilde{P}(t))_{t\ge 0}$ be the continuous time {càdlàg piecewise constant} Markov process constructed from $P$ under $\P^{(\ell)}$ {which takes the same values as $P$ and} where at each step $n$ we wait a time $\mathcal{E}_n/{(2P(n)^{a-1})}$, where $\mathcal{E}_n/(2P(n)) = T_{n+1}-T_n$. We have the scaling limit jointly with the previous ones
$$
\text{Under }\P^{(\ell)}, \qquad \qquad \left(\sum_{n=0}^{\lfloor \ell^{a-1} t\rfloor} \frac{\mathcal{E}_n}{2 P(n)^{a-1}}\right)_{t\ge 0}
\mathop{\longrightarrow}\limits_{\ell \to \infty}^{(\mathrm{d})}
\left(
\frac{1}{2c_a p_{\bf q}} \int_0^{c_a p_{\bf q} t} \frac{1}{X_a^{(1-a)}(s)^{a-1}} ds
\right)_{t\ge 0},
$$
and therefore we have the convergence for the Skorokhod $J_1$ topology jointly with the previous ones
\begin{equation}\label{eq cv Ptilde}
	\text{Under }\P^{(\ell)}, \qquad \qquad \left(\frac{\widetilde{P}(t)}{\ell} \right)_{t\ge 0}
	\mathop{\longrightarrow}\limits_{\ell \to \infty}^{(\mathrm{d})}
	\left(X^{(0)}_{a}(2c_a p_{\bf q}  t) \right)_{t\ge 0}
	=\left(e^{\xi(2c_a p_{\bf q} t)}\right)_{t\ge 0}.
\end{equation}
Note that in the case $a=2$, the process $\widetilde{P}$ is parametrized by the fpp distance to the root.
\subsection{A family of coalescing diffusions with jumps}\label{sous-section coalescent continu}
We denote by $\{x\}\coloneqq x- \lfloor x \rfloor $ the fractional part and by $x_+ = x {\bf 1}_{x>0}$ the positive part of a real number $x$. For every integer $d\ge 1$, we denote by $C_1(\R^d)$ the set of continuous non-negative bounded functions $f$ on $\R^d$ which are zero on a neighbourhood of zero.
For all $z \in \R^*_+, u \in [0,1], x \in \R$ we set
$$
g(x,z,u) =  \left( \{x-u\} - \left( \frac{z-1}{z}\right)_+\right)_+ z 
 - \{x-u\}
 +\frac{1}{2} \left( 1-\left(\frac{1}{z}\wedge z\right) \right).
$$
Note that the function $g$ is $1$-periodic in $x$. Moreover, for all $z \in \R_+^*$, $x \in \R$, $u\in [0,1]$ we have 
\begin{equation}\label{eq majoration g}
	g(x,z,u)\le 3\vert 1-z\vert
\end{equation} 
and
\begin{equation}\label{eq moyenne de g nulle}
\int_0^1 g(x,z,u) du =0.
\end{equation}
Let $N$ be a Poisson point process (PPP) of intensity $2c_a p_{\bf q}dt  \lambda(dz)du$. Let $\widetilde{N}$ be the associated compensated Poisson point process. For all $v \in \R$, let us consider the SDE with jumps
\begin{equation}\label{eq EDSsauts}
\forall t \ge 0, \qquad
X_t(v) = v+ \int_0^t \int_0^\infty \int_0^1 g(X_{s-}(v),z,u) \widetilde{N}(ds,dz,du).
\end{equation}
The process $(X_t(v))_{t\ge 0}$ can also be seen as a time homogeneous jump diffusion with kernel $K$ defined by for all $f \in C_1(\R)$, for all $x \in \R$,
$$
\int_\R f(y) K(x , dy )
=
2c_a p_{\bf q}\int_0^\infty \lambda(dz) \int_0^1 du f(g(x,z,u)).
$$
See e.g.\@ \cite{JS87} Chapter III Paragraph 2c. We will only rely on \cite{JS87} and on a general result of Li and Pu \cite{LP12} but one can also see for instance the books \cite{Jac79}, \cite{IW89} and \cite{App09} for more details on the theory of stochastic differential equations with jumps. 

\begin{lemma}\label{lemme solution forte}
	The existence and uniqueness (up to undistinguishability) of strong solutions to \eqref{eq EDSsauts} holds.
\end{lemma}
\begin{proof}
	In the case $a \in (3/2,2]\cup \{5/2\}$, the above result could be shown using the classical result under Lipschitz type assumptions which can be found e.g.\@ in Theorem 9.1 of Chapter IV of \cite{IW89}. 
	
	However, in order to cover the whole range $a \in (3/2, 5/2]$, we use Theorem 5.1 of \cite{LP12}. To apply this theorem, we need to check conditions (3.a), (3.b) and (5.a) of \cite{LP12}. The condition (3.a) of \cite{LP12} is automatically satisfied by our SDE \eqref{eq EDSsauts} since the corresponding coefficients are zero. 
	Moreover, for our SDE \eqref{eq EDSsauts},
	\begin{itemize}
		\item The condition (3.b) of \cite{LP12} is equivalent to the fact that $x\mapsto x + g(x,z,u)$ is non-decreasing for all $z\in \R_+^*,u \in [0,1]$ and that for all $m\ge 1$, there exists $C_m\ge0$ such that
		\begin{equation}\label{eq coeff lipschitz2}
		\forall x, y \in [-m,m], \qquad	\int_{0}^\infty \lambda(dz) \int_0^1du \left\vert g(x,z,u)-g(y,z,u)\right\vert^2 \le C_m \vert x-y \vert;
		\end{equation}
		\item The condition (5.a) of \cite{LP12} is equivalent to the fact that there exists a constant $C \ge 0$ such that for all $x \in \R$,
		$$
		\int_{(0,\infty) \setminus\{1\}}\lambda(dz) \int_{[0,1]} du g(x,z,u)^2 \le C(1+x^2).
		$$
	\end{itemize} 
	Let us start with condition (5.a). Using \eqref{eq majoration g}, we have for all $x \in \R$,
	\begin{equation}\label{eq integrale coeff au carre finie}
		\int_{(0,\infty) \setminus\{1\}}\lambda(dz) \int_{[0,1]} du g(x,z,u)^2 \le 9 \int_{(1/2,1)\cup (1,\infty)} (z-1)^2 \lambda (dz)<\infty 
	\end{equation}
	since the integral on the right is finite by definition of $\lambda$. Thus, condition (5.a) of \cite{LP12} is satisfied.
	
	To conclude, it now remains to check condition (3.b) of \cite{LP12}.
	
	It is easy to see that for all $z\in \R_+^*$, for all $u \in [0,1]$, the function $x\mapsto x+ g(x,z,u)$ is non-decreasing. Indeed, the points $n+u$ are fixed points of $x \mapsto x+g(x,z,u)-g(0,z,0)$ for all $n \in \Z$. Moreover, on any interval of the form $[n+u, n+1+u]$ for some $n \in \Z$, the function $x\mapsto x+ g(x,z,u)$ is constant on the interval $[n+u, n+u+ ((z-1)/z)_+]$ and is then linearly increasing on $[n+u+ ((z-1)/z)_+ , n+1+u)$ with slope $z$. 
	
	To check condition (3.b), it remains to prove \eqref{eq coeff lipschitz2}. Actually, it suffices to show that there exists $C>0$ such that for all $x, y \in \R$, 
	\begin{equation}\label{eq coeff lipschitz}
		\int_{0}^\infty \lambda(dz) \int_0^1du \left\vert g(x,z,u)-g(y,z,u)\right\vert^2 \le C \vert x-y \vert
	\end{equation}
	 By symmetry, we may assume that $x<y$ and, by \eqref{eq integrale coeff au carre finie}, we may assume that $y-x<1/2$. Using the inequality $(r+s)^2 \le 2(r^2+s^2)$, for all $x,y \in \R$,
	\begin{align}
		\int_{0}^\infty &\lambda(dz) \int_0^1du \left\vert g(x,z,u)-g(y,z,u)\right\vert^2
		\le  2\int_0^\infty \lambda(dz) \int_0^1 du \left\vert (z-1)\left(\{x-u\}-\{y-u\}\right)\right\vert^2\label{premiere ligne condition lipschitz}\\
		&+ 2	\int_{0}^\infty \lambda(dz) \int_0^1du\left\vert \left( \left(\frac{z-1}{z}\right)_+-\{x-u\}\right)_+-\left( \left(\frac{z-1}{z}\right)_+-\{y-u\}\right)_+ \right\vert^2.\label{deuxieme ligne condition lipschitz}
	\end{align}
	The integral on the right of the first line \eqref{premiere ligne condition lipschitz} is {bounded from above} by $8\int_0^\infty \lambda(dz) (z-1)^2 \vert x-y \vert$ using the fact that, since $x<y$ and $y-x<1$, we have
	\begin{align}
	\int_0^1  \vert \{x-u\} - \{y-u\} \vert du &= \int_0^1  \vert u- \{y-x+u\} \vert du \notag\\
	&= \int_0^1  \vert y-x \vert {\bf 1}_{y-x+u \le 1} du + \int_0^1  \vert 1- (y-x) \vert {\bf 1}_{y-x+u \ge 1} du\notag\\
	&\le y-x   + \int_0^1  {\bf 1}_{u \ge 1-(y-x)} du \le 2 (y-x)\label{majoration integrale partie fractionnaire lipschitz}.
\end{align}
For the second line \eqref{deuxieme ligne condition lipschitz} we split the integral between $z<3/2$ and $z\ge 3/2$. The part of the integral in \eqref{deuxieme ligne condition lipschitz} for $z>3/2$ is {bounded from above} by
$$
4 \int_{3/2}^\infty \lambda(dz) \int_0^1 du \vert \{x-u\}-\{y-u\} \vert \le C' \vert x-y \vert
$$
for some constant $C'>0$ which does not depend on $x,y$, using that $t\mapsto t_+$ is $1$-Lipschitz, then \eqref{majoration integrale partie fractionnaire lipschitz} and that $\int_{3/2}^\infty \lambda(dz) <\infty$.
 The part of the integral in \eqref{deuxieme ligne condition lipschitz} for $z<3/2$ can be rewritten
	\begin{align}
	&\int_1^{3/2} \lambda(dz) \int_0^1 du \left(\left( \frac{z-1}{z}-u\right)_+\right)^2 \label{premiere ligne carre lipschitz}
	\\ 
	&-  \int_1^{3/2}\lambda(dz) \int_0^1 du
	\left( \frac{z-1}{z}  - \{x-u\} \right)_+\left( \frac{z-1}{z}  - \{y-u\} \right)_+ \label{deuxieme ligne carre lipschitz}.
	\end{align}
	The first line \eqref{premiere ligne carre lipschitz} equals $(1/3) \int_1^{3/2} \lambda(dz)((z-1)/z)^3$. The second line \eqref{deuxieme ligne carre lipschitz} can be rewritten
	\begin{align*}
	- \int_1^{3/2}\lambda(dz) \int_0^{(z-1)/z} w&\left( \frac{z-1}{z} - \left\{y-x+ \frac{z-1}{z}-w\right\} \right)_+ dw\\
	&= -  \int_1^{3/2} \lambda(dz) \int_0^{(z-1)/z} w(w+x-y)_+dw\\
	&\le -  \int_1^{3/2} \lambda(dz) \int_{y-x}^{(z-1)/z} (w+x-y)^2dw\\
	&= - \frac{1}{3} \int_1^{3/2} \lambda(dz) \left((z-1)/z+x-y \right)^3,
	\end{align*}
	where in the first equality we used that $0<y-x<1/2$ and $(z-1)/z<1/2$. 
	Hence, by the mean value inequality, 
	$$	\int_1^{3/2}\lambda(dz) \int_0^1du\left\vert \left( \frac{z-1}{z}-\{x-u\}\right)_+-\left( \frac{z-1}{z}-\{y-u\}\right)_+ \right\vert^2\le \vert x-y\vert \int_1^{3/2} \lambda(dz) ((z-1)/z)^2.$$ This proves \eqref{eq coeff lipschitz}. 
Thus, we can apply Theorem 5.1 of \cite{LP12}.
\end{proof}
Finally, note that the integral in \eqref{eq EDSsauts} can be seen for all $t\ge0$ as the limit of the martingale
$$
	\int_0^t \int_{\R^*_+\setminus(1-\vp,1+\vp)}\int_0^1 g({X}_{s-}(v), z,u) \widetilde{N}(ds,dz,du)
	\mathop{\longrightarrow}\limits_{\vp \to 0}^{(\mathrm{L}^2)}
	\int_0^t \int_0^\infty\int_0^1 g({X}_{s-}(v), z,u) \widetilde{N}(ds,dz,du),
$$
and that, since for all $x \in \R$, the bijection $u\mapsto \{x-u\}$ preserves the uniform measure on $[0,1]$, for all $v\in \R$, the process $(X_t(v))_{t\ge 0}$ is a Lévy process with no drift, no Brownian part, whose Lévy measure is the image by $(z,u)\mapsto g(0,z,u)$ of $\lambda(dz) du$.
\subsection{The discrete coalescing flow of fpp geodesics}\label{sous-section flot coalescent discret}
Recall {from Subsection \ref{sous-section processus autosimilaire}} that $(\widetilde{P}(t))_{t\ge 0}$ is the continuous time Markov process constructed from $P$ under $\P^{(\ell)}$ where at each step $n$ we wait a time $\mathcal{E}_n/{(2P(n)^{a-1})}$. For every time $t \ge0$ of jump of $\widetilde{P}$ we sample a random variable $V^{(\ell)}_t$ which is uniform in $\{0,1/(2\widetilde{P}(t)),  \ldots, (2\widetilde{P}(t)-1)/(2\widetilde{P}(t))\}$. This random variable corresponds to the position where we glue a new face or where we insert edges in the reversed uniform exploration. 
For all $t\ge 0$, we set
$$
R^{(\ell)}_t = 
-\sum_{\substack{0\le s <t \\ \Delta \widetilde{P}(s) \neq 0}} \left(V^{(\ell)}_s +\frac{1}{2} \max\left(1-\frac{\widetilde{P}(s-)}{\widetilde{P}(s)}, \left(1-\frac{1}{2\widetilde{P}(s)} \right)\left(1-\frac{\widetilde{P}(s)}{\widetilde{P}(s-)}\right) \right)
\right).
$$ 
The quantity $R^{(\ell)}_t$ is introduced for technical reasons in order to compensate the jumps in the processes which are defined below. For every time $t\ge 0$ of jump of $\widetilde{P}$, we write $U^{(\ell)}_t \coloneqq \{V^{(\ell)}_t+ R^{(\ell)}_t\}$. Let $T>0$. 
We define the random measure $N^{(\ell),T}$ by
$$
N^{(\ell),T} = \sum_{0 \le t\le T; \ \Delta \widetilde{P}(t)>0}
\delta_{(T-t,\widetilde{P}(t)/\widetilde{P}(t-), U^{(\ell)}_t )}.
$$
The measure $N^{(\ell),T}$ gives the (translated) position at which we glue the face or insert edges and also the degree of the glued face or the number of inserted edges via the jump of $\widetilde{P}$.

For all $t \le 0$, for all $z \in \R_+^*$, let 
\begin{align*}
g_t^{(\ell)}(x,z,u) =  &\left( \{x-u\} - \left( \frac{z-1}{z}\right)_+\right)_+ z 
- \{x-u\} +\frac{1}{2}\max\left(1-\frac{1}{z}, \left(1-\frac{1}{2\widetilde{P}(-t)} \right)(1-z) \right)
\end{align*}
and for all $t \in [0,T]$, let $g^{(\ell),T}_t = g^{(\ell)}_{t-T}$. Note that $g$ is $1$-periodic in $x$. By construction, note that, by a simple computation, for all $t\in [0,T]$ such that $\Delta \widetilde{P}(T-t)\neq 0$, for all $z \in \R^*_+, x \in \R$ such that $\widetilde{P}(T-t)/z \in \Z$ and $2 \widetilde{P}(T-t) (x- R^{(\ell)}_{T-t}) \in \Z$, we have
\begin{equation}\label{eq moyenne de g nulle discret}
	\E^{(\ell)} \left( \left. g_t^{(\ell),T}(x,z,U^{(\ell)}_{T-t}) \right\vert \widetilde{P} \right) =0.
\end{equation}
Indeed, if we write $p=\widetilde{P}(T-t)$ and $k= p(z-1)/z$ then, when $k\ge 0$ (i.e.\@ $z\ge 1$), we have
\begin{align*}
	\E^{(\ell)} \left( \left. \left(\{x-U^{(\ell)}_{T-t}\} -  \frac{z-1}{z} \right)_+ z - \{x-U^{(\ell)}_{T-t}\}\right\vert \widetilde{P} \right)&=
	\frac{1}{2p} \sum_{i=0}^{2p-1} \left(\left(\frac{i}{2p} - \frac{k}{p} \right)_+ \frac{p}{p-k} - \frac{i}{2p} \right) \\
	&= \frac{1}{2} \left( \frac{p-k}{p} -1\right) = \frac{1}{2} \left(\frac{1}{z}-1\right),
\end{align*}
and a similar calculation proves \eqref{eq moyenne de g nulle discret} when $k \le 0$ (i.e.\@ $z\le 1$).
{Furthermore}, we again have for all $x\in \R, z\in \R^*_+, u \in [0,1]$, 
\begin{equation}\label{eq majoration g discret}
	g_t^{(\ell),T}(x,z,u) \le 3 \vert z-1 \vert.
\end{equation}
The coalescing flow $(X^{(\ell),T}_t(v))_{v \in \R, 0 \le t \le T}$ associated with $N^{(\ell),T}$ is defined by
\begin{equation}\label{eq EDSdiscret}
\forall v\in \R, \forall t \in [0,T], \qquad
X^{(\ell),T}_t(v) = v + \int_0^t \int_0^\infty \int_0^1 g_s^{(\ell),T}(X^{(\ell)}_{s-}(v), z, u) N^{(\ell),T}(ds,dz,du).
\end{equation}
Notice that the above process is not a diffusion with jumps since the Markov process $\widetilde{P}$ is not reversible. This coalescing flow gives the genealogy of the fpp geodesics to the root face under $\P^{(\ell)}$ and, when $a=2$, it is parametrized by the fpp distance.

Let us state the following useful lemma which is a direct consequence of \eqref{eq moyenne de g nulle discret}.

\begin{lemma}\label{lemme martingale}
	Let $\mathcal{F}^{(\ell),T}$ be the filtration defined by
	$$
	\forall t \in [0,T], \qquad
	\mathcal{F}^{(\ell),T}_t
	= \sigma\left(
	\widetilde{P}((T-s)-), U_{T-s}^{(\ell)}; \ s \in [0,t]
	\right).
	$$
	Then for all $\vp>0$, for all $v\in \R$, the process $(M^{(\ell),T,\vp}_t)_{t \in [0,T]}$ defined by
	$$
	\forall t \in [0,T],\qquad
	M^{(\ell),T,\vp}_t
	=\int_0^t\int_{1-\vp}^{1+\vp} \int_0^1 g_s^{(\ell),T} ( X^{(\ell),T}_{s-}(v),z,u) N^{(\ell),T}(ds,dz,du)
	$$
	is a martingale with respect to the filtration $\mathcal{F}^{(\ell),T}$.
\end{lemma}

Finally, let us state a lemma which controls the small jumps of $\widetilde{P}$.
\begin{lemma}\label{lemme petits sauts de P}
	We have the uniform convergence
	$$
	\sup_{\ell \ge 1}
	\E^{(\ell)}
	\left(
	\sum_{s \in[0,T]} {\bf 1}_{\widetilde{P}(s)/\widetilde{P}(s-) \in [1-\vp,1+\vp]\setminus \{1\}}
	\left(\frac{\Delta \widetilde{P}(s)}{\widetilde{P}(s-)}\right)^2 
	\right)
	\mathop{\longrightarrow}\limits_{\vp \to 0}0.
	$$
\end{lemma}
\begin{proof}
	Using the law of the peeling exploration described in Subsection \ref{sous-section loi de l'explo}, one can {bound from above}
\begin{align*}
	\E^{(\ell)}
	&\left(
	\sum_{s \in[0,T] } {\bf 1}_{\widetilde{P}(s)/\widetilde{P}(s-) \in [1-\vp,1+\vp]\setminus \{1\}}
	\left(\frac{\Delta \widetilde{P}(s)}{\widetilde{P}(s-)}\right)^2 
	\right)
	\\
	&\le  C \E^{(\ell)}
	\left(
	\sum_{s \in[0,T]  } {\bf 1}_{\frac{\widetilde{P}(s)}{\widetilde{P}(s-)} \neq 1}
	\sum_{-\vp \widetilde{P}(s-) \le k \le \vp \widetilde{P}(s-)}
	\frac{W^{(\widetilde{P}(s-)+k)} c_{\bf q}^{-\widetilde{P}(s-)-k}}
	{W^{(\widetilde{P}(s-))} c_{\bf q}^{-\widetilde{P}(s-)}}
	\nu(k)
	\left(\frac{k}{\widetilde{P}(s-)}\right)^2 
	\right) \\
	&\le C'\E^{(\ell)}
	\left(
	\sum_{s \in[0,T] } {\bf 1}_{ \frac{\widetilde{P}(s)}{\widetilde{P}(s-)} \neq 1 }
	\sum_{1\le  k \le \vp \widetilde{P}(s-)}
	k^{-a}
	\left(\frac{k}{\widetilde{P}(s-)}\right)^2 
	\right) \qquad \text{by \eqref{eq asymptotique W} and \eqref{queue nu}}\\
	&\le \vp^{3-a} C''
	\E^{(\ell)}
	\left(
	\sum_{s \in[0,T] }{\bf 1}_{ \frac{\widetilde{P}(s)}{\widetilde{P}(s-)} \neq 1 }
	\frac{1}{\widetilde{P}(s-)^{a-1}}
	\right) \\
	&\le \vp^{3-a} C''',
\end{align*}
where $C,C',C'',C'''$ are positive constants which do not depend on $\ell, \vp$. The last inequality comes from the fact that if we denote the exponential clocks defining $\widetilde{P}$ by $\mathcal{E}_n/(2P(n)^{a-1})$ for all integer $n\ge 0$, where the $\mathcal{E}_n$ for $n\ge 0$ are i.i.d.\@ exponential random variables of parameter $1$ independent from $P$, then
\begin{align*}
	\E^{(\ell)}
	\left(
	\sum_{s \in[0,T]}
	{\bf 1}_{ \frac{\widetilde{P}(s)}{\widetilde{P}(s-)} \neq 1 }
	\frac{1}{\widetilde{P}(s-)^{a-1}}
	\right)
	&\le 
	\E^{(\ell)}
	\left(
	\sum_{n=0}^\infty
	\frac{1}{P(n)^{a-1}}
	{\bf 1}_{\sum_{k=0}^{n-1} \mathcal{E}_k/(2P(k)^{a-1}) \le T}
	\right)\\
	&= 
	\E^{(\ell)}
	\left(
	\sum_{n=0}^\infty
	\frac{\mathcal{E}_n}{P(n)^{a-1}}
	{\bf 1}_{\sum_{k=0}^{n-1} \mathcal{E}_k/(2P(k)^{a-1}) \le T}
	\right)\\
	&\le 2T +1.
\end{align*}
This ends the proof of the lemma.
\end{proof}
\subsection{Recovering the fpp geodesics from the discrete coalescing flow}\label{sous-section lien explo renversée flot}
This subsection aims at making explicit the correspondence between the discrete coalescing flow and the fpp geodesics to the root via the reversed uniform exploration introduced in Subsection \ref{sous-section explo renversee}.

Let $T>0$. Let $t_1< \ldots <t_n$ be the times of jumps of $\widetilde{P}$ in $[0,T]$. We set $t_0=0$. Let $\mathfrak{u}^n_n$ be the map made of one distinguished face of degree $2 P(n)$ glued to one hole of the same degree. For all $k \in \lb 0,n \rb$, the edges on the inner boundary of $\mathfrak{u}^n_k$ are encoded modulo $\Z$ in the clockwise order by the real numbers
$$
R_{t_k}^{(\ell)} + \frac{ i}{2P(k)}
$$
for $i \in \lb 0, 2P(k)-1 \rb $, {where the $R_t^{(\ell)}$'s where defined in Subsection \ref{sous-section flot coalescent discret}}. See Figure \ref{fig: coalescent discret}.

For all $k \in \lb 0, n-1 \rb$, we relate the edges on the inner boundary of $\mathfrak{u}^n_k$ to the edges on the inner boundary of $\mathfrak{u}^n_{k+1}$ as follows:
\begin{itemize}
	\item If $\Delta P(k)\ge 0$, we glue a new face of degree $2 \Delta P(k) +2$ to the edges encoded by
	$$
	U^{(\ell)}_{t_{k+1}}, U^{(\ell)}_{t_{k+1}} + \frac{1}{2 P(k+1)}, \ldots 
	,
	U^{(\ell)}_{t_{k+1}} + \frac{ 2 \Delta P(k)}{2P(k+1)},
	$$
	and the non-glued edge of the face is encoded by
	\begin{align*}
	U^{(\ell)}_{t_{k+1}} &+ \frac{i}{2 P(k+1)}+ g_{t_{k+1}}^{(\ell),T} \left(U^{(\ell)}_{t_{k+1}}+ \frac{i}{2 P(k+1)}, \frac{P(k+1)}{P(k)}, U^{(\ell)}_{t_{k+1}}\right) \\
	=
	&U^{(\ell)}_{t_{k+1}} + \frac{i}{2 P(k+1)} + \left( \frac{i}{2P(k+1)} - \frac{\Delta P(k)}{P(k+1)} 
	\right)_+ \frac{P(k+1)}{P(k)} - \frac{i}{2P(k+1)} \\
	&+ \frac{1}{2}
	\max\left(1- \frac{P(k)}{P(k+1)}, \left( 1 - \frac{1}{2 P(k+1)} \right)\left(1- \frac{P(k+1)}{P(k)} \right) \right)\\
	=& U^{(\ell)}_{t_{k+1}} + \frac{1}{2}\max\left(1- \frac{P(k)}{P(k+1)}, \left( 1 - \frac{1}{2 P(k+1)} \right)\left(1- \frac{P(k+1)}{P(k)} \right) \right)  \\
	=& R^{(\ell)}_{t_k},
	\end{align*}
	where the result does not depend on $i \in \lb 0, 2\Delta P(k)\rb$. Note that by this computation, the trajectories of $X^{(\ell),T}$ which were in $[U^{(\ell)}_{t_{k+1}}, U^{(\ell)}_{t_{k+1}} + \frac{ 2 \Delta P(k)}{2P(k+1)}]$ in the time interval $[T-t_{k+1}, T-t_k)$ then coalesce in $R^{(\ell)}_{t_k}$ at time $t_{k+1}$. 
	By a similar computation, the other edges of $\mathfrak{u}^n_{k+1}$, which are encoded by $U_{t_{k+1}}^{(\ell),T}+ (2\Delta P(k) + j)/(2P(k))$ modulo $\Z$ for $j \in \lb 1, 2P(k)-1 \rb$ correspond in $\mathfrak{u}^n_k$ respectively to the edges $R_{t_k}^{(\ell)}+ j/(2P(k))$ modulo $\Z$ for $j \in \lb 1, 2P(k)-1 \rb$, and the corresponding trajectories do not coalesce at time $t_{k+1}$. See Figure \ref{fig: coalescent discret}.
	\item If $\Delta P(k) <0$, then we insert the edges
	$$
	R^{(\ell)}_{t_k} - \frac{1}{2P(k)}, R^{(\ell)}_{t_k} - \frac{2}{2P(k)}, \ldots, 
	R^{(\ell)}_{t_k} - \frac{-2\Delta P(k)}{2P(k)},
	$$
	we glue the first to the last edge and we fill in the hole which is hence created with a Boltzmann random map. For all $i \in \lb 0, 2P(k+1)-1 \rb$, the edge on the inner boundary of $\mathfrak{u}^n_{k+1}$ which is encoded by $U^{(\ell)}_{t_{k+1}}  + i/(2P(k+1))$ then corresponds to the edge encoded by 
	\begin{align*}
	U^{(\ell)}_{t_{k+1}} + \frac{i}{2 P(k+1)}+ g_{t_{k+1}}^{(\ell),T} \left(U^{(\ell)}_{t_{k+1}}+ \frac{i}{2 P(k+1)}, \frac{P(k+1)}{P(k)}, U^{(\ell)}_{t_{k+1}}\right)
	&= R^{(\ell)}_{t_k} + \frac{i}{2P(k+1)} \frac{P(k+1)}{P(k)} \\
	&= R^{(\ell)}_{t_k} + \frac{i}{2P(k)} .
	\end{align*}
	Here, the trajectories do not coalesce at time $t_{k+1}$. See Figure \ref{fig: coalescent discret}.
\end{itemize}

\begin{figure}[h]
	\centering
	\includegraphics[width=1\textwidth]{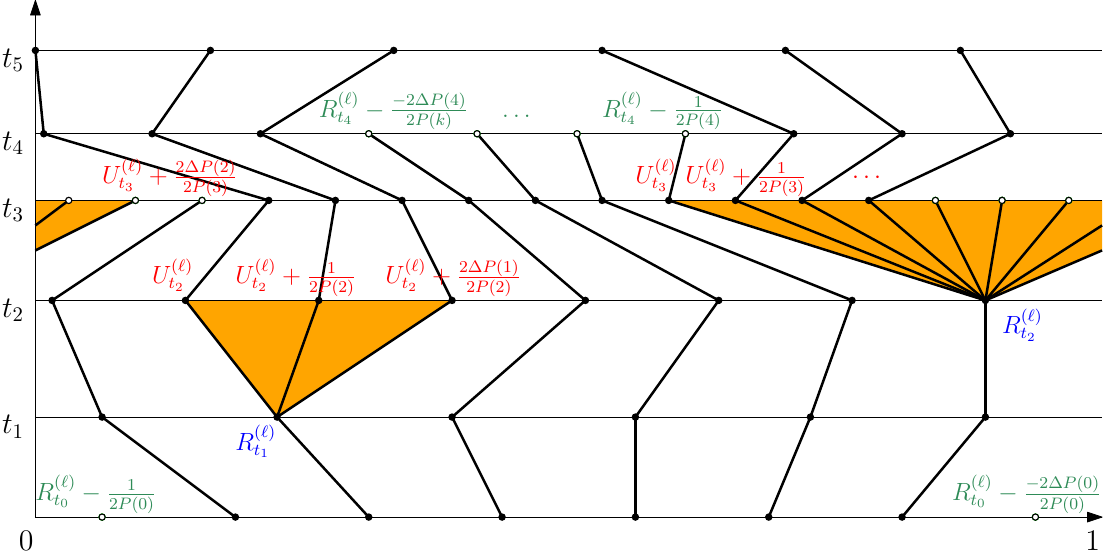}
	\caption{The embedding of the reversed uniform exploration $\mathfrak{u}^5_5 \subset \mathfrak{u}^5_4 \subset \mathfrak{u}^5_3 \subset \mathfrak{u}^5_2 \subset \mathfrak{u}^5_1 \subset \mathfrak{u}^5_0$ of the map $\mathfrak{m}$ depicted in Figure \ref{fig: explo unif renverse} into the discrete coalescing flow. Since the flow is $1$-periodic in space, we only draw what happens on $[0,1]$. The trajectories of the flow are drawn in black: for sake of visibility, they are interpolated by lines between jump times. The orange triangles mark the coalescence events which correspond to the gluing of a face. The empty dots correspond to the edges which are inserted when $\widetilde{P}$ makes a negative jump. The coordinates of the trajectories before they coalesce are in red and the coordinate of the point where they coalesce is in blue.}
	\label{fig: coalescent discret}
\end{figure}

\subsection{Scaling limit of the flow of geodesics}

In this subsection, we prove that the discrete coalescing flow of Subsection \ref{sous-section flot coalescent discret} converges to the family of coalescing diffusions with jumps introduced in Subsection \ref{sous-section coalescent continu} if one looks at a finite (or countable) number of trajectories. Let $T>0$. Let us introduce the restriction $N^T$ of $N$ to subsets of $[0,T] \times \R^*_+ \times[0,1]$. 
\begin{theorem}\label{th limite du coalescent}
For all $\ell \ge1$, let $v_1^{(\ell)}, v_2^{(\ell)}, \ldots, v_d^{(\ell)}\in \R$ and let $v_1, \ldots, v_d \in \R$ such that for all $i \in \lb 1, d \rb$, we have $
v^{(\ell)}_i \longrightarrow v_i$ as $\ell \to \infty$. Then 
$$
\left(X^{(\ell),T}_t(v^{(\ell)}_i)\right)_{1 \le i \le d, t \in [0,T]}
\mathop{\longrightarrow}\limits_{\ell \to \infty}^{(\mathrm{d})}
(X_{t}(v_i))_{1 \le i \le d, 0 \le t \le T},
$$
in the space $\mathbb{D}([0,T],\R)^d$, where each coordinate is equipped with the $J_1$ Skorokhod topology. Moreover, jointly with the above convergence, $N^{(\ell),T}$ converges in distribution towards $N^T$ for the topology of weak convergence of measures on $[0,T] \times ((0,1)\cup (1,\infty)) \times [0,1]$.
\end{theorem}
In order to prove the above result, we denote by $N^T_\vp$ (resp.\@ $N_\vp$) the restriction of $N$ to subsets of $[0,T] \times (\R^*_+\setminus (1-\vp , 1+\vp)) \times [0,1]$ (resp.\@ $\R_+ \times (\R^*_+\setminus (1-\vp , 1+\vp)) \times [0,1]$). Note that the Poisson point process $N^T_\vp$ has a finite intensity. Similarly, for all $\ell \ge 1$, we denote by $N^{(\ell),T}_\vp$ the restriction of $N^{(\ell), T}$ to the subsets of $[0,T] \times (\R^*_+\setminus (1-\vp , 1+\vp)) \times [0,1]$. 

Let $X^{(\ell),T,\vp}$ be the approximation of the discrete coalescing flow defined using $N^{(\ell), T}_\vp$ by 
$$
\forall v \in \R, \forall t \in [0,T], \qquad
X^{(\ell),T,\vp}_t(v) = v+ \int_0^t \int_0^\infty \int_0^1 g_s^{(\ell),T}(X^{(\ell),T,\vp}_{s-}(v),z,u) N^{(\ell), T}_\vp(ds,dz,du),
$$
and let $X^\vp$ be the process defined using $N_\vp$ by
$$
\forall v \in \R, \forall t \in \R_+, \qquad
X^{\vp}_t(v) = v+ \int_0^t \int_0^\infty \int_0^1 g(X^{\vp}_{s-}(v),z,u) N_\vp(ds,dz,du).
$$
The above integral is actually a finite sum since $\lambda(\R_+^* \setminus(1-\vp,1+\vp))<\infty$. Moreover, the integral with respect to the compensated PPP would be the same thanks to \eqref{eq moyenne de g nulle}. We define $X^{T,\vp}$ as the restriction of $X^\vp$ to the time interval $[0,T]$. Note that $X^{T,\vp}$ can also be defined using $N^T_\vp$.

By \eqref{eq cv Ptilde}, we get the convergence of $g^{(\ell),T}_t$ towards $g$:
\begin{equation}\label{eq cv g}
	\text{Under } \P^{(\ell)}, \qquad
	\sup_{t \in [0,T], z \in \R_+^*, u \in [0,1] , x \in \R} \left\vert g^{(\ell),T}_t(x,z,u)- g(x,z,u) \right\vert 
	\le \frac{\vert1-z \vert}{\inf_{t \in [0,T]} \widetilde{P}(t)}
	\mathop{\longrightarrow}\limits_{\ell \to \infty}^{(\P)} 0.
\end{equation}

We can then state the following scaling limit result:
\begin{lemma}\label{lemme cv du flot tronque}
Let $d \ge 1$. For all $\ell \ge1$, let $v_1^{(\ell)}, v_2^{(\ell)}, \ldots, v_d^{(\ell)}\in \R$ and let $v_1, \ldots, v_d \in \R$ such that for all $i \in \lb 1, d \rb$, we have 
$
v^{(\ell)}_i \longrightarrow v_i
$ as $\ell \to \infty$. 
Then 
$$
\left(X^{(\ell),T,\vp}_t(v^{(\ell)}_i)\right)_{1 \le i \le d, t \in [0,T]}
\mathop{\longrightarrow}\limits_{\ell \to \infty}^{(\mathrm{d})}
(X^{T,\vp}_t(v_i))_{1 \le i \le d, 0 \le t \le T},
$$
in the space $\mathbb{D}([0,T],\R)^d$.
\end{lemma}
The above lemma is a direct consequence of \eqref{eq cv g}, of the following lemma and of the fact that $N^T_\vp$ has a finite number of atoms.
\begin{lemma}\label{lemme cv ppp}
		For every continuous function on $[0,T] \times ((0,1)\cup (1,\infty)) \times [0,1]$ with compact support, 
		$$
		\int_0^T\int_{(0,1)\cup (1,\infty)} \int_0^1 f(t,z,u)N^{(\ell),T}(dt,dz,du) \mathop{\longrightarrow}\limits_{\ell \to \infty}^{(\mathrm{d})}
		\int_0^T\int_{(0,1)\cup (1,\infty)} \int_0^1
		f(t,z,u)
		N^{T}(dt,dz,du).
		$$
		In other words, $N^{(\ell),T}$ converges in distribution towards $N^T$ for the topology of weak convergence of measures on $[0,T] \times ((0,1)\cup (1,\infty)) \times [0,1]$.
\end{lemma}
\begin{proof}
The convergence \eqref{eq cv Ptilde} implies that 
$$
\sum_{t\ge 0; \ \Delta \widetilde{P}(t)>0} \delta_{(t,\widetilde{P}(t)/\widetilde{P}(t-))} 
\mathop{\longrightarrow}\limits_{\ell \to \infty}^{(\mathrm{d})}
\mathrm{PPP}\left(2 c_a p_{\bf q}dt  \lambda(dz)\right),
$$
for the topology of weak convergence of measures on $\R_+ \times (\R_+^*\setminus \left\{1\right\})$. This concludes the proof by reversing time and using that $\ell/\inf_{[0,T]} \widetilde{P}$ under $\P^{(\ell)}$ as $\ell \to \infty$ is tight by \eqref{eq cv Ptilde} so that the random variables $U^{(\ell)}_t$ converge in distribution towards independent uniform random variables in $[0,1]$. 
\end{proof}

Next, we can approximate $X^{(\ell),T}$ using an appropriate coupling with $X^{(\ell),T,\vp}$.
\begin{lemma}\label{lemme couplage}
For every sequence of real numbers $(v^{(\ell)})_{\ell \ge 1}$, there exist couplings of $(X^{(\ell),T}_t(v^{(\ell)}))_{t \in [0,T]}$ with $(X^{(\ell),T,\vp}_t(v^{(\ell)}))_{t \in [0,T]}$ for all $\ell \ge 1$ and for all $\vp >0$, denoted by $X^{(\ell),T}(v^{(\ell)}),\overline{X}^{(\ell),T,\vp}(v^{(\ell)})$ such that
$$
\sup_{\ell \ge 1} \E^{(\ell)} 
\left(\sup_{t\in [0,T]} 
\left|
X^{(\ell),T}_t(v^{(\ell)})
-
\overline{X}^{(\ell),T,\vp}_t(v^{(\ell)})
\right|^2
\right)
\mathop{\longrightarrow}\limits_{\vp \to 0} 0.
$$
\end{lemma}

\begin{proof}
Let $(v^{(\ell)})_{\ell \ge 1}$ be a sequence of real numbers. The coupling of the two processes is defined using the same perimeter process $P$, the same exponential random variables for $\widetilde{P}$ but the random variables $\overline{U}^{(\ell),T}$ defining $\overline{X}^{(\ell),T,\vp}_t(v^{(\ell)})$ are defined inductively by the relations
\begin{align}
\overline{U}_{T-t}^{(\ell),T,\vp}&= \left\{ \overline{X}^{(\ell),T,\vp}_{t-}(v^{(\ell)}) - X^{(\ell),T}_{t-}(v^{(\ell)}) + U^{(\ell)}_{T-t}\right\}, \\
\overline{N}^{(\ell),T,\vp}_{\vert [0,t]\times (\R^*_+\setminus \{1\})\times [0,1]} 
&= \sum_{s\in [T-t,T]; \ \widetilde{P}(s)/\widetilde{P}(s-) \in \R_+^*\setminus (1-\vp, 1+\vp)} \delta_{(T-s,\widetilde{P}(s)/\widetilde{P}(s-), \overline{U}^{(\ell),T,\vp}_{s})}, \\
\overline{X}^{(\ell),T,\vp}_t(v^{(\ell)}) &= v+ \int_0^t \int_0^\infty \int_0^1 g^{(\ell),T}_s(\overline{X}^{(\ell),T,\vp}_{s-}(v^{(\ell)}),z,u) \overline{N}^{(\ell), T}_\vp(ds,dz,du),
\end{align}
for all $t\in [0,T]$ such that $\Delta \widetilde{P}(t)\neq 0$. Then one can see that $\overline{X}^{(\ell),T,\vp}(v^{(\ell)})$ has the same law as ${X}^{(\ell),T,\vp}(v^{(\ell)})$ and that by construction, for all $t\in [0,T]$ such that $\Delta \widetilde{P}(t)\neq 0$, for all $z \in \R_+^*\setminus \left\{1\right\}$, 
\begin{equation}\label{eq egalite couplage}
g^{(\ell),T}_t(\overline{X}^{(\ell),T,\vp}_{t-}(v^{(\ell)}), z, \overline{U}_{T-t}^{(\ell),T,\vp}) = g^{(\ell),T}_t(X^{(\ell),T}_{t-}(v^{(\ell)}), z, U^{(\ell)}_{T-t}).
\end{equation}
Using the coupling, we see that
\begin{align*}
\E^{(\ell)} 
&\left(\sup_{t\in [0,T]} 
\left|
X^{(\ell),T}_t(v^{(\ell)})
-
\overline{X}^{(\ell),T,\vp}_t(v^{(\ell)})
\right|^2
\right)
\\
=
&\E^{(\ell)} 
\left(\sup_{t\in [0,T]} 
\left|
\int_0^t \int_{1-\vp}^{1+\vp} \int_0^1
g^{(\ell),T}_s(X^{(\ell),T}_{s-}(v^{(\ell)}),z,u)
N^{(\ell),T}(ds, dz, du)
\right|^2
\right)
\qquad
\text{by } \eqref{eq egalite couplage}.
\end{align*}
Moreover, thanks to Lemma \ref{lemme martingale}, the process
$$
(M_t^{(\ell),T,\vp})_{t \in [0,T]}
=
\left(
\int_{0}^t
\int_{1-\vp}^{1+\vp}
\int_0^1
g^{(\ell),T}_s(X^{(\ell),T}_{s-}(v^{(\ell)}),z,u) {N}^{(\ell),T} (ds, dz, du)
\right)_{t \in [0,T]}
$$
is a martingale. So by Doob's martingale inequality,
$$
\E^{(\ell)}
\left(
\sup_{t \in [0,T]}
(M^{(\ell),T,\vp}_t)^2
\right)
\le  4 \sup_{t \in [0,T]} \E^{(\ell)} \left( (M^{(\ell),T,\vp}_t)^2\right).
$$
Moreover, by \eqref{eq majoration g discret}, one can {bound from above} for all $t \in [0,T]$,
\begin{align*}
\E^{(\ell)} \left( (M^{(\ell),T,\vp}_t)^2 \right)
&=\E^{(\ell)} \left(\int_{0}^t \int_{1-\vp}^{1+\vp} \int_0^1
g^{(\ell),T}_s(X^{(\ell),T}_{s-}(v^{(\ell)}),z,u)^2 
{N}^{(\ell),T}(ds,dz,du)
\right) \enskip\text{(quadratic variation)}\\
&\le 9\E^{(\ell)} \left(\int_{0}^t \int_{1-\vp}^{1+\vp} \int_0^1
\vert z-1\vert^2 
{N}^{(\ell),T}(ds,dz,du)
\right) \\
&\le 9 \E^{(\ell)}
\left(
\sum_{s \in[0,T] } {\bf 1}_{ \widetilde{P}(s)/\widetilde{P}(s-) \in [1-\vp,1+\vp]\setminus \{1\} }
\left(\frac{\Delta \widetilde{P}(s)}{\widetilde{P}(s-)}\right)^2 
\right).
\end{align*}
Using Lemma \ref{lemme petits sauts de P} one concludes that
\begin{equation}\label{eq majoration unif de la martingale}
	\sup_{\ell \ge 1} \sup_{t \in [0,T]} \E^{(\ell)} \left( (M^{(\ell),T,\vp}_t)^2\right)
	\mathop{\longrightarrow}\limits_{\vp \to 0} 0.
\end{equation}
This ends the proof.
\end{proof}
 We will need this last lemma:
 \begin{lemma}\label{lemme detronque}
 	We have for all $v_1, \ldots, v_d \in \R$, 
 	$$
 	\left(X_t^{\vp}(v_i)\right)_{ 1 \le i \le d, t\ge 0} \mathop{\longrightarrow}\limits_{\vp \to 0}^{(\mathrm{d})}
 	\left(X_t (v_i)\right)_{1 \le i \le d, t \ge 0}
 	$$ in the space $\mathbb{D}(\R_+, \R)^d$.
 \end{lemma}
 \begin{proof}
 	It is a direct consequence of Theorem 4.15 of Chapter IX of \cite{JS87}. Indeed, for all $v_1, \ldots, v_d \in \R$, the process $(X_t^{\vp}(v_i))_{ i \in \lb 1,d\rb, t\ge 0}$ (resp.\@ $(X_t^{\vp}(v_i))_{ i \in \lb 1,d \rb, t\ge 0}$) is the time homogeneous jump diffusion starting from $(v_1,\ldots, v_d)$ with kernel $K^{d,\vp}$ (resp.\@ $K^d$) defined by the fact that for all $f \in C_1(\R^d)$, for all $x \in \R^d$,
 	$$
 	\int_{\R^d} f(y) K^{d,\vp}(x , dy )
 	=
 	2c_a p_{\bf q}\int_{(0,\infty) \setminus(1-\vp,1+\vp)} \lambda(dz) \int_{0}^1 du f(g(x_1,z,u), \ldots ,g(x_d,z,u)),
 	$$
 	and respectively
 	$$
 	\int_{\R^d} f(y) K^{d}(x , dy )
 	=
 	2c_a p_{\bf q}\int_{(0,\infty)\setminus\{1\}} \lambda(dz) \int_0^1du f(g(x_1,z,u), \ldots ,g(x_d,z,u)).
 	$$
 	We denote by $\vert \cdot \vert_\infty$ the infinite norm on $\R^d$. Then, by \eqref{eq integrale coeff au carre finie}, for all $x\in \R^d$,
 	$$
 	\int_{\R^d} \vert y \vert_\infty^2 K^{d,\vp}(x,dy) \le \int_{\R^d} \vert y \vert_\infty^2 K^d(x,dy)
 	\le 9  \int_{(1/2,1) \cup (1,\infty)} \lambda(dz) \vert z-1 \vert^2 <\infty,
 	$$
 	and
 	$$
 	\sup_{x \in \R^d} \int_{\R^d} \vert y \vert_\infty^2{\bf 1}_{\vert y \vert_\infty>b} K^d(x,dy)
 	\le 9 \int_{(1/2,1) \cup (1,\infty)} \lambda(dz) \vert z-1 \vert^2 {\bf 1}_{3\vert z-1 \vert >b}
 	\mathop{\longrightarrow}\limits_{b\to \infty} 0.
 	$$
 	
 	{Furthermore}, for all $x' \in \R^d$ for $\lambda(dz) du$ almost every $(z,u) \in \R_+^* \times [0,1]$, the function $x\mapsto (g(x_i,z,u))_{1 \le i \le d}$ is continuous at $x'$, so that by dominated convergence, for all $f \in C_1(\R^d)$,
 	$$
 	\int_0^\infty \lambda(dz) \int_0^1du f(g(x_1, z, u), \ldots, g(x_d,z,u))
 	\mathop{\longrightarrow}\limits_{x\to x'}
 	\int_0^\infty \lambda(dz) \int_0^1du f(g(x'_1, z, u), \ldots, g(x'_d,z,u)),
 	$$
 	where we use the domination
 	$$
 	f(g(x_1, z, u), \ldots, g(x_d,z,u)) \le {\bf 1}_{3\vert z -1 \vert >r} \sup (f),
 	$$
 	where $r>0$ is such that $\{y \in \R^d; \ \vert y \vert_\infty <r\}\subset \R^d \setminus \mathrm{Supp}(f)$. Therefore, the function $x\mapsto \int_{\R^d} K^d(x,dy)f(y)$ is continuous. By the same reasoning, it is also true for $x\mapsto \int_{\R^d} K^{d,\vp}(x,dy)f(y)$.
 	
 	Furthermore, by the Beppo Levi lemma, for all $x\in \R^d$, for all $f\in C_1(\R^d)$,
 	$$
 	\int_{\R^d} f(y)K^{d,\vp}(x,dy) \mathop{\longrightarrow}\limits_{\vp \to 0}
 	\int_{\R^d} f(y)K^d(x,dy),
 	$$
 	and this convergence holds locally uniformly on $x$ by Dini's theorem. Moreover, by \eqref{eq majoration g}, we have for all $i,j \in \lb 1, d \rb$, uniformly on $x$,
 	$$
 	\left\vert \int_{\R^d} y_iy_jK^{d,\vp}(x,dy) - \int_{\R^d} y_i y_j K^d (x,dy)\right\vert \le \int_{1-\vp}^{1+\vp} \lambda(dz)\int_0^1 du \max_{1 \le i \le d} g(x_i,z,u)^2\le 9 \int_{1-\vp}^{1+\vp} \vert 1-z\vert^2dz,
 	$$
 	which goes to zero as $\vp \to 0$. Finally, the last condition to check before applying Theorem 4.15 of Chapter IX is the ``uniqueness-measurability hypothesis'' 4.3 of Chapter IX. It is a consequence of the existence and uniqueness of strong solutions to \eqref{eq EDSsauts} which is provided by Lemma \ref{lemme solution forte}.
 \end{proof}

\begin{proof}[Proof of Theorem \ref{th limite du coalescent}]
	By combining the results of Lemmas \ref{lemme cv du flot tronque}, \ref{lemme couplage} and \ref{lemme detronque}, we obtain that for every converging sequence $v^{(\ell)} \to v \in \R$,
	$$
	\text{Under the law }\P^{(\ell)}, \qquad
	\left( X^{(\ell),T}_t(v^{(\ell)}) \right)_{t \in [0,T]}
	\mathop{\longrightarrow}\limits_{\ell \to \infty}^{(\mathrm{d})}
	\left(
	X_t(v)
	\right)_{t\in [0,T]}.
	$$
	Moreover, by Lemma \ref{lemme cv ppp}, we know that jointly with \eqref{eq cv Ptilde},
	$$
	\text{Under the law }\P^{(\ell)}, \qquad
	N^{(\ell),T}
	\mathop{\longrightarrow}\limits_{\ell \to \infty}^{(\mathrm{d})}
	N^T
	$$
	for the topology of weak convergence of measures on $[0,T]\times(\R^*_+\setminus\{1\}) \times [0,1]$. 
	
	Assume that, along some subsequence $(\ell_k)_{k\ge 0}$, we have 
	\begin{equation}\label{eq cv X N P}
	\text{Under }\P^{(\ell_k)}, \qquad
	\left(X^{(\ell_k),T}(v^{(\ell_k)}), N^{(\ell_k),T}, \frac{\widetilde{P}}{\ell_k}
	\right)
	\mathop{\longrightarrow}\limits_{k \to \infty}^{(\mathrm{d})}
	\left(\check{X}^T(v), N^T, \left(\exp(\xi_a(2 c_a p_{\bf q}t))\right)_{t\ge 0} 
	\right),
	\end{equation}
	where $\check{X}^T(v)$ has the same law as $X^T(v)$. Then, to end the proof of Theorem \ref{th limite du coalescent}, it suffices to show that 
	\begin{equation}\label{eq identification de la limite}
	\forall t \in [0,T], \qquad \check{X}^T_t(v) = v + \int_0^t \int_0^\infty \int_0^1 g(\check{X}^T_{s-}(v), z,u)\widetilde{N}^T(ds,dz,du),
	\end{equation}
	where $\widetilde{N}^T$ is the compensated Poisson process associated with $N^T$.
	By Skorokhod's representation theorem, we assume that \eqref{eq cv X N P} holds almost surely. For all $\ell\ge 1$, let $\psi_\ell$ be an increasing  homeomorphism from $[0,T]$ to $[0,T]$ such that
	\begin{equation}\label{eq cv X unif}
	\sup_{t \in [0,T]} \left\vert 
	X^{(\ell_k), T}_t(v^{(\ell_k)}) - \check{X}^T_{\psi_{\ell_k}(t)}(v)
	\right\vert
	+\sup_{t \in[0,T]} \left\vert \psi_{\ell_k}(t)-t \right\vert
	\mathop{\longrightarrow}\limits_{k\to \infty}^{(\mathrm{a.s.})} 0
	\end{equation}
	Using \eqref{eq cv g} and the fact that $(N^{(\ell),T}([0,T]\times ((1/2,1-\vp)\cup(1+\vp,\infty))\times[0,1]))_{\ell \ge 1}$ is bounded for all $\vp >0$, we have for all $t\in [0,T]$, for all $\vp>0$, a.s.\@ along the subsequence $(\ell_k)_{k\ge 0}$
	\begin{align*}
		\int_0^t &\int_{(1/2,1-\vp)\cup(1+\vp, \infty)} \int_0^1
		\left( g^{(\ell),T}_s(X^{(\ell),T}_{s-}(v^{(\ell)}), z, u) 
		-
		g(\check{X}^{T}_{\psi_\ell(s)-}(v), z, u)\right) N^{(\ell),T}(ds,dz,du) \\
		&= \int_0^t\int_{(1/2,1-\vp)\cup(1+\vp, \infty)} \int_0^1
		\left(g(X^{(\ell),T}_{s-}(v^{(\ell)}), z, u) -g(\check{X}^T_{\psi_\ell(s)-}(v), z, u) \right)N^{(\ell),T}(ds,dz,du)
		+o(1).
	\end{align*}
	and by taking \eqref{eq cv X unif} into account, using that for all $x_0 \in \R$, for $\lambda(dz) du$ almost all $(z,u) \in \R_+^* \times [0,1]$, the function $x\mapsto g(x,z,u)$ is continuous at $x_0$, we deduce by dominated convergence (with domination by a constant) that along $(\ell_k)_{k\ge 0}$,
	$$
	\int_0^t \int_{(1/2,1-\vp)\cup(1+\vp, \infty)} \int_0^1
	\left( g^{(\ell),T}_s(X^{(\ell),T}_{s-}(v^{(\ell)}), z, u) 
	-
	g(\check{X}^{T}_{\psi_\ell(s)-}(v), z, u)\right) N^{(\ell),T}(ds,dz,du)
	\mathop{\longrightarrow}\limits_{\ell \to \infty}^{(\mathrm{a.s.})}
	0.
	$$
	In addition to that, thanks to the convergence of $N^{(\ell),T}$ towards $N^T$ (and since $\lambda$ has no atom{s}), we have the convergence of finite sums for all $t \ge 0$.
	$$
	\int_0^t \int_{\R^*_+\setminus(1-\vp,1+\vp)}\int_0^1 g(\check{X}^T_{s-}(v), z,u) N^{(\ell),T}(ds,dz,du)
	\mathop{\longrightarrow}\limits_{\ell \to \infty}^{(\mathrm{a.s.})}
	\int_0^t \int_{\R^*_+\setminus(1-\vp,1+\vp)}\int_0^1 g(\check{X}^T_{s-}(v), z,u) N^T(ds,dz,du).
	$$
	Thus, 
	\begin{align}
	\int_0^t &\int_{\R^*_+\setminus(1-\vp,1+\vp)}\int_0^1 g_s^{(\ell),T}({X}^{(\ell),T}_{s-}(v^{(\ell)}), z,u) N^{(\ell),T}(ds,dz,du)\notag\\
	&\mathop{\longrightarrow}\limits_{\ell \to \infty}^{(\mathrm{a.s.})}
	\int_0^t \int_{\R^*_+\setminus(1-\vp,1+\vp)}\int_0^1 g(\check{X}^T_{s-}(v), z,u) N^T(ds,dz,du).\label{eq cv grands sauts}
	\end{align}
	Moreover, by \eqref{eq majoration unif de la martingale}, 
	\begin{equation}\label{eq majoration uniforme des petits sauts}
	\sup_{\ell\ge 1} \E^{(\ell)} \left( 
	\left\vert
	\int_0^t \int_{1-\vp}^{1+\vp} \int_0^1
	g^{(\ell),T}_s(X^{(\ell),T}_{s-}(v^{(\ell)}), z,u) N^{(\ell),T}(ds,dz,du)
	\right\vert^2\right)
	\mathop{\longrightarrow}\limits_{\vp \to 0} 0.
	\end{equation}
	{Furthermore}, using \eqref{eq moyenne de g nulle}, one can see that 
	we know that for all $t\ge 0$,
	$$
	\int_0^t \int_{\R^*_+\setminus(1-\vp,1+\vp)}\int_0^1 g(\check{X}^T_{s-}(v), z,u) N^T(ds,dz,du)=
	\int_0^t \int_{\R^*_+\setminus(1-\vp,1+\vp)}\int_0^1 g(\check{X}^T_{s-}(v), z,u) \widetilde{N}^T(ds,dz,du).
	$$
	Finally, we have the convergence of martingale
	$$
	\int_0^t \int_{\R^*_+\setminus(1-\vp,1+\vp)}\int_0^1 g(\check{X}^T_{s-}(v), z,u) \widetilde{N}^T(ds,dz,du)
	\mathop{\longrightarrow}\limits_{\vp \to 0}^{(\mathrm{L}^2)}
	\int_0^t \int_{\R^*_+\setminus\{1\}}\int_0^1 g(\check{X}^T_{s-}(v), z,u) \widetilde{N}^T(ds,dz,du).
	$$
	Combining this with \eqref{eq cv grands sauts} and \eqref{eq majoration uniforme des petits sauts} gives \eqref{eq identification de la limite}. This concludes the proof.
\end{proof}

\section{Scaling limit of the tree of geodesics}\label{section arbre geodesiques}
In this section, we define a random infinite countable metric space $\T_a$ associated with the stochastic flow arising from the coalescing diffusions with jumps and we prove that it is the scaling limit of the tree of fpp geodesics of random maps of type $a \in (3/2,5/2)$.

\subsection{The tree of the positive jumps via the stochastic flow}\label{sous-section definition de la limite}
Let $N$ be a PPP on $(-\infty,0]\times (\R^*_+\setminus \{1\}) \times [0,1]$ of intensity $dt \lambda(dz)  du$. For all $v\in \R$ and $s\le 0$, let $(X_{s,t}(v))_{v \in \R,  t  \in [s,0]}$ be the solution starting from $v \in \R$ at time $s\le 0$ of the SDE
\begin{equation}\label{eq EDSsauts flot}
	\forall t \in [s,0],  \qquad
	X_{s,t}(v) = v+ \int_s^t \int_0^\infty \int_0^1 g(X_{s,t'-}(v),z,u) \widetilde{N}(d{t'},dz,du),
\end{equation}
where $\widetilde{N}$ is the compensated PPP associated with $N$. Note that $(X_{s,t}(v))_{v \in \R, s \le t \le 0}$ is a stochastic flow, and more precisely a Lévy flow in the terminology of \cite{App09} Section 6.4.1 thanks to Lemma \ref{lemme solution forte}. It is also easy to see that $X$ is a stochastic coalescing flow on the torus $\R/\Z$ in the sense that for all $s_1,s_2< 0$, for all $v_1,v_2 \in \R$, we have $\P(\exists t \in [s_1 \vee s_2, 0], \ \{X_{s_1,t}(v_1)\}=  \ \{X_{s_2,t}(v_2)\})>0$ and if we set $T= \inf \{ t\in [s_1 \vee s_2, 0], \ \{X_{s_1,t}(v_1)\}=  \ \{X_{s_2,t}(v_2)\} \}$ then on the event $\{T<\infty\}$, for all $t \in [T,0]$, we have $ \{X_{s_1,t}(v_1)\}=  \ \{X_{s_2,t}(v_2)\} $.

Let $\mathbb{U}= \bigcup_{n\ge 0} \N^n$ be the Ulam tree, where $\N^0$ has only one element which is the empty word $\emptyset$. For all $v,w\in \U$, we write $v\prec w$ if $v$ is a prefix of $w$. Let $N_w$ for $w \in \mathbb{U}$ be independent copies of $N$. Let $\partial$ be a cemetery point. For convenience, we will also denote $N_w$ in the form $(Z_t^w,U_t^w)_{t\le 0}$, where $Z_t^w = \partial$ when there is no atom of $N_w$ at time $t$, so that
$$
N_w =\sum_{t\le 0} {\bf 1}_{Z_t^w \neq \partial} \delta_{(t,Z^w_t, U^w_t)}.
$$ 
For all $w \in \mathbb{U}$, $s\le 0$ and $v\in \R$, let $\left(X^w_{s,t}(v) \right)_{t \in [s,0]}$ be the solution of the SDE \eqref{eq EDSsauts flot} where $N$ is replaced by $N_w$. By analogy with the discrete coalescing flow, for all $t \le 0$ such that $Z^w_t \neq \partial$ we set
$$
R^w_t \coloneqq U^w_t + g(U^w_t, Z^w_t, U^w_t)
$$
and when $Z^w_t = \partial$ we set $ R^w_t  = \partial$. Then $R^w_t$ represents the position at which the trajectories which were in $[U^w_t, U^w_t + (Z^w_t-1)/Z^w_t]$ at time $t-$ coalesce when $Z^w_t>1$ and $R^w_t$ represents the position on the right of the empty hole of size $1-Z^w_t$ created when $Z^w_t<1$.

Let $(\mathcal{X}_w(t), b_w)_{w \in \U, t\ge 0}$ be the cell system which is constructed as in Section 2.2 of \cite{BBCK} as follows using the $N_w$'s: 
\begin{itemize}
	\item For all $w\in \U$, let $\xi_w$ be the L\'evy process with no Brownian part, whose jumps are given by the Poisson point process $\Delta \xi_w$ on $\R_+\times \R^*$ which is defined as the image measure of $N_w$ by $(t,z,u)\mapsto (-t, \log z)$, i.e. $(\Delta \xi_w(t))_{t\ge 0} = (\log Z^w_{-t})_{t\ge 0}$, and whose drift is given by \eqref{eq drift}.
	\item We set $\mathcal{X}_\emptyset = \exp(\xi_\emptyset)$, with birth time $b_\emptyset =0$. 
	\item Assume by induction that we have defined $\mathcal{X}_w$ and its birth time $b_w$ for some $w \in \U$. Let $(x_{wi})_{i\in \N}$ be the collection of negative jumps of $\mathcal{X}_w$, ordered in decreasing order of absolute value. Let $(\beta_{wi})_{i\ge 1}$ be the times at which they occur. For all $i \in \N$ we set $\mathcal{X}_{wi} =\vert x_{wi} \vert \exp(\xi_{wi})$, where $wi$ is the finite sequence obtained by adding $i$ at the end of the sequence $w$. We define the birth time of the cell $wi$ by setting $b_{wi}= b_w+\beta_{wi}$. 
\end{itemize}

For all $w\in \U$, let $(\Delta_{wi})_{i\in \N}$ be the collection of the positive jumps of $\mathcal{X}_{w}$ in decreasing order, with by convention $\Delta_\emptyset = 1$ and let $(\tau_{wi})_{i\in \N}$ be the times at which they occur. We define the times $(t_w)_{w\in \U}$ by setting $t_\emptyset =0$ and for all $w \in \U$, for all $i \in\N$, 
$$
t_{wi} = b_w+\tau_{wi}.
$$

When $a\neq 2$, it is useful to introduce the Lamperti transformed birth times and jump times $(\widetilde{b}_w)_{w \in \U}$, $(\widetilde{t}_w)_{w \in \U}$, $(\widetilde{\beta}_w)_{w \in \U \setminus \{\emptyset\}}$ and $(\widetilde{\tau}_w)_{w \in \U\setminus\{\emptyset\}}$ by setting for all $w\in \U$, for all $i \in \N$,
$$
\widetilde{\beta}_{wi} = \int_0^{\beta_{wi}} e^{(2-a) \xi_w(t)} dt 
\qquad
\text{and}
\qquad
\widetilde{\tau}_{wi} = \int_0^{\tau_{wi}} e^{(2-a) \xi_w(t)} dt,
$$
then setting $\widetilde{t}_\emptyset = \widetilde{b}_\emptyset=0$ and for all $w \in \U$, for all $i \in \N$,
$$
\widetilde{b}_{wi} = \widetilde{b}_w + \widetilde{\beta}_{wi} \qquad
\text{and}
 \qquad
\widetilde{t}_{wi} = \widetilde{b}_w+ \widetilde{\tau}_{wi}.
$$

Let $\mathcal{T}_a$ be the countable metric space whose underlying set is $
\mathcal{T}_a\coloneqq
\U$. 
Each $w\in \T_a$ corresponds to either a positive jump $\Delta_w$ when $w\neq \emptyset$ or to the root. To each $w = i_1\cdots i_k \in \mathcal{T}_a\setminus \{\emptyset \}$, we associate a trajectory $(Y^{z,w})_{z\prec w}$ in the Lévy flows $X^z$ for $z\prec w$. Let $\emptyset  = w_0\prec w_1 \prec \ldots \prec w_k = w$ be the prefixes of $w$, so that for all $j\in \lb 0,k-1\rb$, we have $w_{j+1} = w_j i_{j+1}$. The idea is to follow the trajectory starting from the jump $\Delta_w$ (which corresponds to a coalescence event) in the Lévy flow $X^{w_{k-1}}$, and when we reach time $0$ we continue the trajectory in the flow $X^{w_{k-2}}$ starting from the negative jump arising at time $-\beta_{w_{k-1}}$ in this flow, etc. 

More precisely, we set
$$
\left(Y^{w_{k-1},w}_t\right)_{-\tau_w\le t \le 0}\coloneqq 
\left(X^{w_{k-1}}_{-\tau_w,t}(R^{w_{k-1}}_{-\tau_w}))\right)_{ -\tau_w\le t\le0}.
$$
Note that the starting point of the above process corresponds to the point where previous trajectories of the flow $X^{w_{k-1}}$ coalesce due to the positive jump $\Delta_w$. See the left-hand side of Figure \ref{fig: trajectoire}.

Next, for all $j \in \lb 1, k-1\rb$, we define inductively
$$
(Y^{w_{j-1},w}_t)_{-\beta_{w_j}\le t \le 0} 
\coloneqq \left(X^{w_{j-1}}_{-\beta_{w_j},t}\left(
R^{w_{j-1}}_{-\beta_{w_j}}-(1- Z^{w_{j-1}}_{-\beta_{w_j}}) + (1-Z^{w_{j-1}}_{-\beta_{w_j}})Y^{w_j,w}_0
\right)\right)_{ -\beta_{w_j}\le t\le0}.
$$
Here, the starting point corresponds to a point chosen in the ``gap'' created by the negative jump $x_{w_j}$, and it is chosen according to the ending point $Y^{w_j,w}_0$ of the previous trajectory. See the right-hand side of Figure \ref{fig: trajectoire}.

\begin{figure}[h]
	\centering
	\includegraphics[width=1\textwidth]{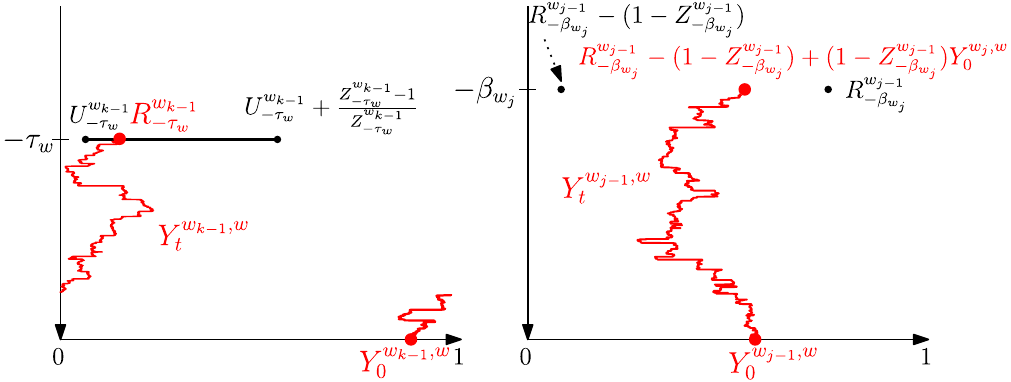}
	\caption{The trajectory associated with a positive jump $\Delta_w$. The trajectory is in red and is represented in $[0,1]$ since the flow is $1$-periodic in space. On the left, the trajectory starts at the point $R^{w_{k-1}}_{-\tau_w}$ where there is a coalescence in the flow at time $-\tau_w$ due to the positive jump. The segment of the positions at time $(-\tau_w)-$ which coalesce in $R^{w_{k-1}}_{-\tau_w}$ is drawn in black. The first part of the trajectory then ends at $Y_0^{w_{k-1},w}$. On the right the next parts of the trajectories for $j \in \lb 0, k-1 \rb$ start at a position corresponding to the end of the previous part of the trajectory in the ``gap'' of size $1-Z^{w_{j-1}}_{-\beta_{w_j}}$ created at time $-\beta_{w_j}$.} 
	\label{fig: trajectoire}
\end{figure}

Let us denote by $\mathcal{J}_a$ the (random) set of sequences of positive jumps $(\Delta_{wi_n})_{n\ge 0}$ such that the sequence of times of jump $\tau_{wi_n}$ is increasing and converges as $n\to \infty$. For all $u = (wi_n)_{n\ge 0} \in \mathcal{J}_a$, we define $\tau_u$ (resp. $t_u$, $\widetilde{\tau}_u$, $\widetilde{t}_u$) as the limit of $\tau_{wi_n}$ (resp. $t_{wi_n}$, $\widetilde{\tau}_{wi_n}$, $\widetilde{t}_{wi_n}$) as $n \to \infty$. 

According to the value of $a$, we describe differently the way two trajectories coalesce modulo $\Z$.
\begin{lemma}\label{lemme coalescence trajectoires}
	Let $v\neq w \in  \mathcal{T}_a \setminus \{\emptyset\}$. Let $z$ be a common prefix of $v$ and $w$. Assume that $\{Y_{t_0}^{z,v}\}\neq \{Y_{t_0}^{z,w}\}$ for some fixed $t_0<0$. Let
	$$
	\sigma \coloneqq \inf\{t\in [t_0,0], \ \{Y_t^{z,v}\} = \{Y_t^{z,w}\}\},
	$$
	with $\inf \emptyset = \infty$ by convention. Then, on the event $\sigma <\infty$,
	\begin{itemize}
		\item If $a \in (3/2,2)$, then a.s.\@ $\Delta \xi_z (-\sigma) >0$.
		\item If $a \in [2,5/2)$, then there exists a sequence of positive jumps $u=(\Delta_{zi_n})_{n\ge 0}\in \mathcal{J}_a$ such that $\tau_u = -\sigma$ and such that for all $n \ge 0$, we have $\{Y^{z,w}_{-\tau_{zi_n}-}-U^z_{-\tau_{zi_n}}\} \le (Z^z_{-\tau_{zi_n}}-1)/Z^z_{-\tau_{zi_n}}$.
	\end{itemize}
\end{lemma}
\begin{proof}
	Let us start with the case $a \in (3/2, 2)$. Note that $\int_0 x^{1-a} dx <\infty$ and $\int^\infty x^{1-2a} <\infty$, so that 
	$$\int_{(1/2,\infty)\setminus\{1\}} \vert x-1 \vert \lambda(dx) <\infty. $$ 
	Therefore, almost surely we have 
	\begin{equation}\label{eq somme des sauts finie}
		\sum_{t \in [t_0,0]} {\bf 1}_{Z^z_t \neq 1}\left\vert {Z^{z}_t -1} \right\vert <\infty.
	\end{equation}
	By {the} Borel-Cantelli lemma, we deduce that almost surely there is a finite number of times $t\in [t_0,0]$ such that $Z^z_t>1$ and $\{Y^{z,w}_t-U^z_t\} \le (Z^z_t-1)/Z^z_t$ or $\{Y^{z,v}_t-U^z_t\} \le (Z^z_t-1)/Z^z_t$. We denote those times by $s_1< \ldots <s_N$
	Furthermore, by \eqref{eq somme des sauts finie}, we have a.s.
	$$
	\sum_{t \in [t_0, 0]} {\bf 1}_{1/2< Z^z_t <1} \log Z^z_t >- \infty,
	$$
	so that between two consecutive coalescence times $s_{i}<s_{i+1}<0$ such that $\{Y^{z,w}_{s_i}\} \neq \{Y^{z,v}_{s_i}\}$, the distance between $Y^{z,w}_t$ and $Y^{z,v}_t$ in the torus $\R/\Z$ cannot converge to zero since it is {bounded from below} by $\prod_{s_i<t<s_{i+1}; \ 1/2<Z_t^z<1}  Z^z_t>0$ times the distance between $\{Y^{z,w}_{s_i}\}$ and $\{Y^{z,v}_{s_i}\}$ in $\R/\Z$. Thus, the coalescence can only take place during a time $s_i$ for some $i \in \lb 1, N \rb$. This means that $Z^z_{\sigma} >1$, i.e.\@ $\Delta \xi_z(-\sigma) >0$.

	Now let us treat the case $a \in [2,5/2)$. Note that $\int_0 x^{1-a} dx =\infty$ and hence
	$$\int_{(1,\infty)} \frac{x-1}{x} \lambda(dx) =\infty. $$ Therefore, by Campbell's theorem, a.s.\@
	$$\sum_{t \in [t_0,0]}  {\bf 1}_{Z^{z}_t>1}\frac{Z^{z}_t -1}{Z^{z}_t} =\infty.$$
	By the second Borel-Cantelli lemma, using the independence of the $U_t^{z}$'s, we deduce that a.s.\@ the set of times $t\in [\sigma,0]$ such that $Z^z_t>1$ and $\{Y^{z,w}_{t-}-U^z_t\} \le (Z^z_t-1)/Z^z_t$ is dense in $[\sigma,0]$. In particular, there exists a sequence $(\Delta_{zi_n})_{n\ge 0}\in \mathcal{J}_a$ of positive jumps such that $\tau_{zi_n} \uparrow -\sigma$ as $n\to \infty$ and such that for all $n \ge 0$, we have $\{Y^{z,w}_{-\tau_{zi_n}-}-U^z_{-\tau_{zi_n}}\} \le (Z^z_{-\tau_{zi_n}}-1)/Z^z_{-\tau_{zi_n}}$. 
\end{proof}

For all $v, w\in \mathcal{T}_a$, their nearest common ancestor $c(v,w)$ in $\mathcal{T}_a\cup \mathcal{J}_a$ is defined as follows:
\begin{itemize}
	\item If $v=\emptyset$ or $w=\emptyset$, then $c(v,w)=\emptyset$.
	\item If $v=w$, then $c(v,w)=v=w$.
	\item If $v\neq w \in \mathcal{T}_a\setminus \{\emptyset\}$, then let $z$ be their largest common prefix. Let $\emptyset=z_0\prec z_1\prec \ldots \prec z_k = z$ be the sequence of prefixes of $z$. Let
	$$
	\sigma_{v,w}(z_k)\coloneqq \inf\left\{ t \le 0, \ Y^{z_k,v}_t \text{ and } Y^{z_k,w}_t \text{ are defined and }
	\{Y^{z_k,v}_t\}
	=
	\{Y^{z_k,w}_t\}
	  \right\}.
	$$
	If $\sigma_{v,w}(z_k)<\infty$, then the trajectories $Y^{z_k,w}$ and $Y^{z_k,v}$ coalesce at time $\sigma_{v,w}(z_k)$. We distinguish two subcases:
	\begin{itemize}
		\item If $a\in (3/2,2)$, then by Lemma \ref{lemme coalescence trajectoires}, we have $\Delta \xi_{z_k} (- \sigma_{v,w}(z_k))>0$, so that there exists a unique $i\in \N$ such that $\Delta \mathcal{X}_{z_k}({-\sigma_{v,w}(z_k)}) = \Delta_{z_ki}$. Then we set $c(v,w)= z_k i$. 
		\item Otherwise, when $a \in [2,5/2)$, by Lemma \ref{lemme coalescence trajectoires} there exists a sequence $(\Delta_{z_{k}i_n})_{n\ge 0}\in \mathcal{J}_a$ of positive jumps such that $\tau_{z_ki_n} \uparrow -\sigma_{v,w}(z_k)$ as $n\to \infty$ and such that for all $n \ge 0$, we have $\{Y^{z_k,w}_{-\tau_{z_ki_n}-}-U^{z_k}_{-\tau_{z_ki_n}}\} \le (Z^{z_k}_{-\tau_{z_ki_n}}-1)/Z^{z_k}_{-\tau_{z_ki_n}}$. Then we set $c(v,w) = (\Delta_{z_{k}i_n})_{n\ge 0}$.
	\end{itemize}
	When $\sigma_{v,w}(z_k) = \infty$, we set 
	$$
	\sigma_{v,w}(z_{k-1})\coloneqq \inf\left\{ t \le 0, \
	\{Y^{z_{k-1},v}_t\}
	=
	\{Y^{z_{k-1},w}_t\}
	  \right\}
	$$
	and we do as above. If $\sigma_{v,w}(z_{k-1})<\infty$, then the trajectories $Y^{z_{k-1},w}$ and $Y^{z_{k-1},v}$ coalesce at time $\sigma_{v,w}(z_{k-1})$. Then when $a\in (3/2,2)$, we have $\Delta \xi_{z_{k-1}}({-\sigma_{v,w}(z_{k-1})}) > 0$, so that there exists a unique $i\in \N$ such that $\Delta \mathcal{X}_{z_{k-1}}({-\sigma_{v,w}(z_{k-1})}) = \Delta_{z_{k-1}i}$. Then we set $c(v,w)= z_{k-1} i$. When $a \in [2,5/2)$ there exists a sequence $(\Delta_{z_{k-1} i_n})_{n\ge 0}$ such that $\tau_{z_{k-1}i_n} \uparrow -\sigma_{v,w}(z_{k-1})$ as $n \to \infty$ and such that for all $n \ge 0$, we have $\{Y^{z_{k-1},w}_{-\tau_{z_{k-1}i_n}-}-U^{z_{k-1}}_{-\tau_{z_{k-1}i_n}}\} \le (Z^{z_{k-1}}_{-\tau_{z_{k-1}i_n}}-1)/Z^{z_{k-1}}_{-\tau_{z_{k-1}i_n}}$. Then we set $c(v,w) = (\Delta_{z_{k-1}i_n})_{n\ge 0}$, etc.  Finally, if for all $j \in \lb 0,k\rb$ the trajectories do not coalesce, i.e.\@ $\sigma_{v,w}(z_j)=\infty$ (this happens with positive probability), then we set $c(v,w)=\emptyset$.
\end{itemize}
The space $\mathcal{T}_a$ is equipped with the distance $d_{\mathcal{T}_a}$ defined as follows. 

For all $w \in \mathcal{J}_a \cup \mathcal{T}_a$, we set
\begin{equation}\label{eq hauteur sur T}
d_{\mathcal{T}_a}(\emptyset, w) \coloneqq \widetilde{t}_w
\end{equation}
Note that in the case $a=2$, the above expression can also be written $d_{\mathcal{T}}(\emptyset, w)= t_w$. Besides, for future use, we stress that, thanks to the absence of atom{s} in the intensity of the PPPs $N_w$, the function $w \mapsto d_{\mathcal{T}}(\emptyset, w)$ is a.s.\@ injective.
For all $w, v \in \mathcal{T}_a $, we set
\begin{equation}\label{eq distance sur T}
d_{\mathcal{T}_a}(w,v) \coloneqq d_{\mathcal{T}_a}(\emptyset, w) +d_{\mathcal{T}_a}(\emptyset, v)-2d_{\mathcal{T}_a}(\emptyset, c(v,w)).
\end{equation}
It is easy to check that $d_{\mathcal{T}_a}$ is a distance on $\mathcal{T}_a$. Indeed, positivity and symmetry are obvious. For the triangle inequality, by \eqref{eq distance sur T} it suffices to check that for all $u,v,w \in \mathcal{T}_a$, we have
$$
d_{\T_a} (\emptyset, v) + d_{\T_a}(\emptyset ,c(u,w)) \ge d_{\T_a}(\emptyset,c(u,v)) + d_{\T_a}(\emptyset, c(v,w)).
$$
This holds clearly since:
\begin{itemize}
\item If the trajectories of $u$ and $v$ coalesce before those of $u$ and $w$, then $c(v,w) = c(u,w)$, and one can see that $d_{\T_a}(\emptyset, v) \ge d_{\T_a}(\emptyset , c(u,v))$;
\item Otherwise, $d_{\T_a}(\emptyset, c(u,w)) \ge d_{\T_a}(\emptyset, c(u,v))$ and one can see that $d_{\T_a}(\emptyset, v) \ge d_{\T_a}(\emptyset , c(v,w))$.
\end{itemize}

\subsection{The tree of geodesics via the discrete coalescing flow}

We recall the branching peeling exploration of a map $\mathfrak{m}$ via discrete cell-systems from \cite{BCM,BBCK}. One can then perform the time-reversed uniform exploration on each branch via the discrete coalescing flow thanks to the encoding presented in Subsection \ref{sous-section lien explo renversée flot}.

The branching peeling exploration of a map $\mathfrak{m}$, contrary to the filled in peeling exploration, does not fill in the holes which are created when we identify two edges (i.e.\@ when the perimeter process has a negative jump) but instead starts another process which explores this hole. Thanks to the spatial Markov property, under $\P^{(\ell)}$, this new exploration is independent conditionally on the perimeter of the hole and has the same law as the exploration of a Boltzmann map of perimeter corresponding to the perimeter of the hole. {In each of these explorations, we use the uniform peeling exploration recalled in Subsection \ref{sous-section explo uniforme}.}

Under $\P^{(\ell)}$, the perimeter processes of this branching exploration are described by a cell-system $((P_w(n))_{n \ge 0})_{w \in \U}$ which is defined as follows: $P_\emptyset = P$ is the perimeter process of the exploration of $\mathfrak{m}$ following the locally largest component. Recall that it is a Markov chain whose probability transitions are given by \eqref{eq probas de transition perimetre fini}. We define the birth-time of the cell $\emptyset$ by $B_\emptyset =0$. Then, assume inductively that $(P_w(n))_{n\ge 0}$ and its birth-time $B_w$ are defined for some $w \in \U$. Let $\ell_{wi}$ for $i\ge 1$ be the sizes in absolute value of the negative jumps, ranked in non-increasing order and let $(n_{wi})_{i \ge 1}$ be the times at which they occur. For all $i\ge 1$ we define $(P_{wi}(n))_{n\ge 0}$ as the exploration of the independent Boltzmann map of law $\P^{(\ell_{wi})}$ which fills the hole created by the associated negative jump. Then conditionally on the $\ell_{wi}$'s (and on all the branches $P_{w'}$ for $w'\in \U$ such that $w$ is not a strict prefix of $w'$), the processes $(P_{wi}(n))_{n\ge 0}$ for $i \ge 1$ are independent Markov chains whose transitions are given by \eqref{eq probas de transition perimetre fini}. This property is called the branching Markov property. For all $i\ge 1$, birth-time of the cell $wi$ is defined by $B_{wi} =  B_w +n_i$.

On each branch corresponding to a cell $P_w$ for $w \in \U$, we define the associated discrete coalescing flow which gives the fpp geodesics in $\mathfrak{m}$. Let us denote by $(\mathcal{E}^w_n)_{n\ge 0}$ the i.i.d.\@ exponential random variables arising in the uniform peeling exploration of the map filling in the corresponding hole (note that $\mathcal{E}^\emptyset_n = \mathcal{E}_n$ for all $n \ge 0$). Let $(\widetilde{P}_w(t))_{t\ge 0}$ be the continuous time process obtained by waiting a time $\mathcal{E}^w_n/{(2P_w(n)^{a-1})}$ at $P_w(n)$ before jumping at $P_w(n+1)$ for all $n\ge 0$. We define $U^{(\ell),w}_{t}, g^{(\ell), w}_t$  as in Subsection \ref{sous-section flot coalescent discret} using $\widetilde{P}_w$ in place of $\widetilde{P}$. We define the random measure
$$
N_w^{(\ell)}= \sum_{t \ge 0  } {\bf 1}_{\Delta \widetilde{P}_w(t) \neq 0}\delta_{(-t,\widetilde{P}_w(t)/\widetilde{P}_w(t-), U^{(\ell),w}_t)}.
$$
 We denote by $(X^{(\ell),w}_{s,t}(v))_{v \in \R, s\le t \le 0}$ the discrete coalescing flow defined by setting
 $$
 \forall v \in \R, \ \forall s \le t \le 0, \qquad \qquad
 X^{(\ell),w}_{s,t} (v)= v
 +\int_s^t g_{s'}^{(\ell),w} ( X^{(\ell),w}_{s,s'}(v), z, u) N^{(\ell)}_w(ds',dz,du).
 $$
 Notice that at each time of creation of a hole, conditionally on the already explored region and on the fpp lengths of the inner dual edges, the dual edges on the boundary of the hole are i.i.d.\@ exponential random variables of parameter one. 
 Therefore, the fpp geodesic in the map filling the hole from a face $f$ to the boundary of the hole is a part of the fpp geodesic from $f$ to $f_r$ in $\mathfrak{m}$.
 
 For all $w \in \U$, for all $i \in \N$, let $\partial \mathfrak{h}_w$ be the boundary of the hole corresponding to the cell $P_w$ and let $f_{wi}$ be the face created by the $i$-th largest positive jump of $P_w$ which occurs at time $m_i$. By convention, we set $f_\emptyset= f_r$. We can thus record the fpp distances
 \begin{equation}\label{eq distance fpp cellule}
 d^\dagger_{\mathrm{fpp}} (\partial \mathfrak{h}_w, \partial \mathfrak{h}_{wi})
 = \sum_{n=0}^{n_i -1} \frac{\mathcal{E}^w_n}{2P_w(n)}
 \qquad\text{and}
 \qquad
 d^\dagger_{\mathrm{fpp}} (\partial \mathfrak{h}_w,  f_{wi})
 = \sum_{n=0}^{m_i -1} \frac{\mathcal{E}^w_n}{2P_w(n)}
 \end{equation}
 and then we have
 \begin{equation}\label{eq distance fpp racine recurrence}
 d^\dagger_{\mathrm{fpp}} (f_r, \partial \mathfrak{h}_{wi})
 =d^\dagger_{\mathrm{fpp}} (f_r, \partial \mathfrak{h}_{w}) + d^\dagger_{\mathrm{fpp}} (\partial \mathfrak{h}_w, \partial \mathfrak{h}_{wi})
 \quad \text{and}
 \quad
 d^\dagger_{\mathrm{fpp}} (f_r, f_{wi})
 =d^\dagger_{\mathrm{fpp}} (f_r, \partial \mathfrak{h}_{w})
 +d^\dagger_{\mathrm{fpp}} (\partial \mathfrak{h}_w, f_{wi}).
 \end{equation}
 We also denote the associated times of negative and positive jumps in $\widetilde{P}_w$ by
 $$
 \beta^{(\ell)}_{wi}
 \coloneqq \sum_{n=0}^{n_i -1} \frac{\mathcal{E}^w_n}{2P_w(n)^{a-1}}
 \qquad\text{and}
 \qquad
 \tau^{(\ell)}_{wi}
 \coloneqq \sum_{n=0}^{m_i -1} \frac{\mathcal{E}^w_n}{2P_w(n)^{a-1}}
 $$
 and we set
 $$
 b^{(\ell)}_{wi}
 \coloneqq b^{(\ell)}_w + \beta^{(\ell)}_{wi}
 \quad \text{and}
 \quad
 t^{(\ell)}_{wi}
 \coloneqq b^{(\ell)}_{w}
 + \tau^{(\ell)}_{wi}.
 $$
 In order to understand the distances between two faces in the tree of geodesics to the root, we describe the trajectories associated with the fpp geodesics in the discrete coalescing flows. Let $w = i_1 \cdots i_k \in \U \setminus \{\emptyset\}$. Let $\emptyset = w_0 \prec w_1 \prec \ldots \prec w_k = w$ be the prefixes of $w$. The face $f_w$ corresponds to the $i_k$-th largest positive jump of $\widetilde{P}_{w_{k-1}}$. We follow the fpp geodesic from $f_w$ to the root $f_r$ in the flows $X^{(\ell),w_k},X^{(\ell),w_{k-1}}, \ldots , X^{(\ell),w_0}$. More precisely, we set
 $$
 (Y^{(\ell), w_{k-1}, w}_t)_{-\tau^{(\ell)}_w \le t \le 0} \coloneqq 
 \left(
 X^{(\ell),w_{k-1}}_{-\tau^{(\ell)}_w, t}\left(R^{(\ell),w}_{\tau^{(\ell)}_w}\right)
 \right)_{-\tau^{(\ell)}_w \le t \le 0} 
 $$
 Note that the starting point of the above process corresponds to the point where previous trajectories of the flow $X^{(\ell),w_{k-1}}$ coalesce due to the positive jump $\Delta \widetilde{P}_{w_{k-1}}(\tau^{(\ell)}_w)$.

Next, for all $j \in \lb 1, k-1 \rb$, we define inductively $Y^{(\ell), w_{j-1}, w}_t$ of every $t \in[-\beta^{(\ell)}_{w_j} , 0]$ by
$$
Y^{(\ell), w_{j-1}, w}_t 
\coloneqq 
X^{(\ell),w_{j-1}}_{-\beta^{(\ell)}_{w_j}, t}
\left(
R^{(\ell),w}_{\beta^{(\ell)}_{w_j}}  -  \frac{- \Delta \widetilde{P}_{w_{j-1}}(\beta^{(\ell)}_{w_j})}{\widetilde{P}_{w_{j-1}}(\beta^{(\ell)}_{w_j}-)}+ \frac{- \Delta \widetilde{P}_{w_{j-1}}(\beta^{(\ell)}_{w_j})}{\widetilde{P}_{w_{j-1}}(\beta^{(\ell)}_{w_j}-)} Y^{(\ell), w_j,w}_0
\right)
$$

Let us then describe the nearest common ancestor of two faces $f$ and $f'$ of $\mathfrak{m}$ in the tree of fpp geodesics to the root using the above-defined trajectories. When $f \neq f_r$ (resp.\@ $f' \neq f_r$), the face $f$ (resp.\@ $f'$) corresponds to the $i$-th (resp. $i'$-th) largest positive jump of $P_{w}$ (resp.\@ $P_{w'}$), i.e.\@ $f=f_{wi}$ (resp.\@ $f'=f_{w'i'}$). Note that the fpp geodesics from $f$ and $f'$ to the root cannot coalesce until they lie in the same branch. See Figure \ref{fig: coalescence geodesique}.

\begin{figure}[h]
	\centering
	\includegraphics[width=0.43\textwidth]{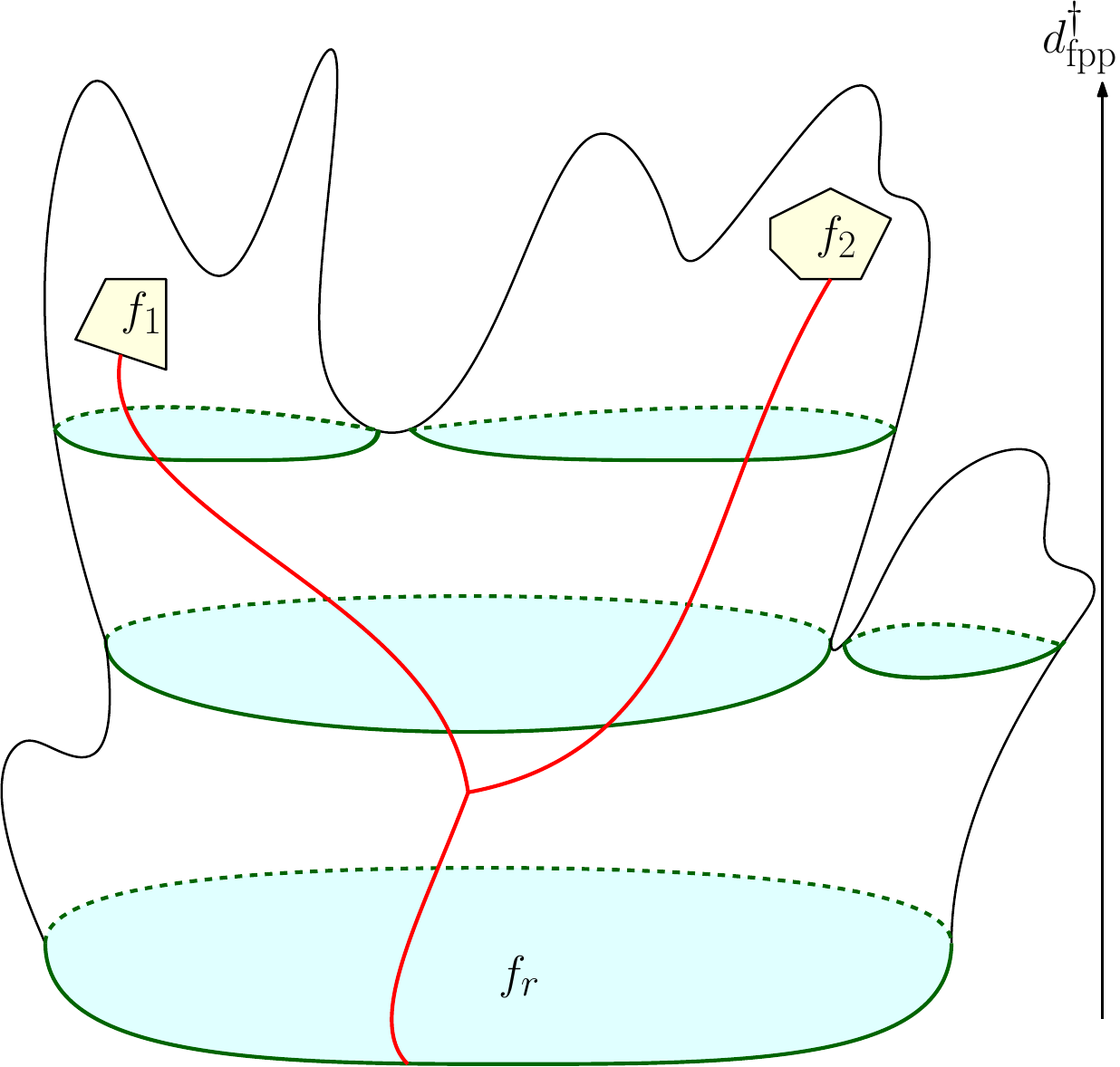}
	\caption{Coalescence of the fpp geodesics from $f_1$ and $f_2$ to the root face $f_r$.} 
	\label{fig: coalescence geodesique}
\end{figure}

The nearest common ancestor $c^{(\ell)}(f, f')$ in the tree of fpp geodesics to the root is characterized as follows:
\begin{itemize}
	\item If $f = f_r$ or $f'=f_r$, then $c^{(\ell)}(f,f')= f_r$;
	\item If $f=f'$, then $c^{(\ell)}(f,f')=f=f'$;
	\item If $f\neq f'$ and $f,f' \neq f_r$, then let $z$ be the largest prefix of $wi,w'i'$. Let $\emptyset = z_0 \prec z_1 \prec \ldots \prec z_k = z$ be the sequence of the prefixes of $z$. Let
	$$
	\sigma^{(\ell)}_{wi,w'i'}(z_k) \coloneqq \inf \{t \le 0 , \ Y^{(\ell), z_k, wi}_t ,Y^{(\ell), z_{k}, w'i'}_t \text{ are defined and } \{Y^{(\ell), z_k, wi}_t \} = \{Y^{(\ell), z_{k}, w'i'}_t\}\}.
	$$
	If $\sigma^{(\ell)}_{wi,w'i'}(z_k) <\infty$, then the trajectories $Y^{(\ell), z_{k}, wi}$ and $Y^{(\ell), z_{k}, w'i'}$ coalesce at time $\sigma^{(\ell)}_{wi,w'i'}(z_k)$. Then $\Delta \widetilde{P}_{z_k} (-\sigma^{(\ell)}_{wi,w'i'}(z_k))>0$, so that there exists a unique $j\in \N$ such that $\Delta \widetilde{P}_{z_k} (-\sigma^{(\ell)}_{wi,w'i'}(z_k))$ is the $j$-th largest positive jump of $P_{z_k}$. In other words, the fpp geodesics coalesce at the face $f_{z_k j}$. Then we have $c^{(\ell)}(f,f')= f_{z_k j}$. Otherwise, we set
	$$
	\sigma^{(\ell)}_{wi,w'i'}(z_{k-1}) \coloneqq \inf \{t \le 0 , \ \{Y^{(\ell), z_{k-1}, wi}_t \} = \{Y^{(\ell), z_{k-1}, w'i'}_t\}\}.
	$$
	If $\sigma^{(\ell)}_{wi,w'i'}(z_{k-1}) <\infty$, then the trajectories $Y^{(\ell), z_{k-1}, wi}$ and $Y^{(\ell), z_{k}-1, w'i'}$ coalesce at time $\sigma^{(\ell)}_{wi,w'i'}(z_{k-1}) $, so that there exists a unique $j \ge 1$ such that $\Delta \widetilde{P}_{z_{k-1}} (-\sigma^{(\ell)}_{wi,w'i'}(z_{k-1}))$ is the $j$-th largest positive jump of $P_{z_{k-1}}$. In other words, the fpp geodesics coalesce at the face $f_{z_{k-1} j}$. Then we have $c^{(\ell)}(f,f') = f_{z_{k-1}j}$, etc. Finally, if for all $m \in \lb 0,k \rb$ the trajectories do not coalesce, i.e.\@ $\sigma^{(\ell)}_{wi,w'i'}(z_m) = \infty$, then the fpp geodesics coalesce at $c^{(\ell)}(f,f')= f_r$.
\end{itemize}

Furthermore, the distance $d_{\T(\mathfrak{m})}$ induced by the fpp lengths on the tree $\T(\mathfrak{m})$ of fpp geodesics to the root is characterized by the relation
\begin{equation}\label{eq distance dans l arbre des geosesiques fpp}
d_{\T(\mathfrak{m})}(f,f') = d^\dagger_\mathrm{fpp}(f_r,f)+ d^\dagger_\mathrm{fpp}(f_r,f') - 2 d^\dagger_\mathrm{fpp}(f_r, c^{(\ell)}(f,f')).
\end{equation}

\subsection{Scaling limit of the tree of geodesics to the root}
In this subsection, assuming $a \in (3/2, 5/2)$, we prove Theorem \ref{Th limite d echelle de l arbre}. 

\begin{proof}[Proof of Theorem \ref{Th limite d echelle de l arbre}]By \eqref{eq cv Ptilde} and by the branching Markov property, one can check the joint convergence for all $w \in \U$,
\begin{equation}\label{eq cv ptilde w}
\text{Under }\P^{(\ell)}, \qquad \qquad \left(\frac{\widetilde{P}_w(t)}{\ell} \right)_{t\ge 0}
\mathop{\longrightarrow}\limits_{\ell \to \infty}^{(\mathrm{d})}
\left(e^{\xi_w(2c_a p_{\bf q} t)}\right)_{t\ge 0},
\end{equation}
{where we recall that the Lévy process $\xi_w$ is defined in Subsection \ref{sous-section definition de la limite}.}
Notice that the above convergence entails the scaling limit of the degrees of the faces, i.e.\@ jointly with the above convergence, $\deg( f_w)/ \ell $ converges in law towards $\Delta_w$ jointly for all $w \in \U$. Furthermore, the fact that one can enumerate the $\Delta_w$'s in decreasing order comes from Proposition 3 of \cite{CCM20}, which, combined with \eqref{eq cv ptilde w}, shows in particular that the law of the family of the $\Delta_w$'s is absolutely continuous with respect to the jumps of a stable Lévy process with no negative jump stopped when it reaches $-1$.

By \eqref{eq cv Ptilde}, jointly with the above convergence, we have the joint convergence of the jump times for all $w \in \U$, for all $i \in \N$,
$$
\text{Under }\P^{(\ell)}, \qquad \qquad
\left(\beta^{(\ell)}_{wi}, \tau^{(\ell)}_{wi}\right) \mathop{\longrightarrow}\limits_{\ell \to \infty}^{(\mathrm{d})} \frac{1}{2c_a p_{\bf q}} \left(\beta_{wi}, \tau_{wi}\right),
$$
hence the convergence of the times
$$
\text{Under }\P^{(\ell)}, \qquad \qquad
\left(b^{(\ell)}_{w}, t^{(\ell)}_{wi}\right) \mathop{\longrightarrow}\limits_{\ell \to \infty}^{(\mathrm{d})} \frac{1}{2c_a p_{\bf q}} \left(b_{w}, t_{wi}\right).
$$
Moreover, thanks to \eqref{eq cv Tn fini} and \eqref{eq distance fpp cellule}, using a Lamperti transform, jointly with the previous convergences, for all $w \in \U$, for all $i \in \N$,
$$
\text{Under }\P^{(\ell)}, \qquad \qquad
\left(\frac{1}{\ell^{a-2}} d^\dagger_\mathrm{fpp}(\partial \mathfrak{h}_w, \partial \mathfrak{h}_{wi}),\frac{1}{\ell^{a-2}}d^\dagger_\mathrm{fpp}(\partial \mathfrak{h}_w, f_{wi})\right) \mathop{\longrightarrow}\limits_{\ell \to \infty}^{(\mathrm{d})} \frac{1}{2c_a p_{\bf q}} \left(\widetilde{\beta}_{wi}, \widetilde{\tau}_{wi}\right),
$$
Therefore, by \eqref{eq distance fpp racine recurrence}, we have for all $w \in \U,i \in \N$,
\begin{equation}\label{eq cv distances fpp racine}
	\text{Under }\P^{(\ell)}, \qquad \qquad
\left(\frac{1}{\ell^{a-2}} d^\dagger_\mathrm{fpp}(f_r, \partial \mathfrak{h}_{w}),\frac{1}{\ell^{a-2}}d^\dagger_\mathrm{fpp}(f_r, f_{wi})\right) \mathop{\longrightarrow}\limits_{\ell \to \infty}^{(\mathrm{d})} \frac{1}{2c_a p_{\bf q}} \left(\widetilde{b}_{w}, \widetilde{t}_{wi}\right),
\end{equation}

Moreover, for all $d\ge 1$, for every sequences $s_i^{(\ell)} \to s_i <0$ and $v_i^{(\ell)} \to v_i \in \R$ for $i \in \lb 1, d \rb$, for all $w_1, \ldots, w_d \in \U$, jointly with the previous convergences, a consequence of Theorem \ref{th limite du coalescent} is the joint convergence for all $i \in \lb 1, d \rb$,
\begin{equation}\label{eq cv flot w}
	\text{Under }\P^{(\ell)}, \quad 
	\begin{pmatrix}
	\left(X^{(\ell),w_i}_{s_i^{(\ell)},t}(v_i^{(\ell)})\right)_{t \in [s_i^{(\ell)},0]}\\
	\left(
	\frac{\widetilde{P}_{w_i}(t)}{\widetilde{P}_{w_i}(t-)}, U^{(\ell),w_i}_t
	\right)_{t \ge 0}
	\end{pmatrix}
	\mathop{\longrightarrow}\limits_{\ell \to \infty}^{(\mathrm{d})}
	\begin{pmatrix}
	\left(X^{w_i}_{2c_a p_{\bf q} s_i,2c_a p_{\bf q} t}(v_i)\right)_{t \in [s_i,0]} 
	\\
	\left(
	Z^{w_i}_{-t}, U^{w_i}_{-t}
	\right)_{t \ge 0}
	\end{pmatrix},
\end{equation}
where the first coordinate converges for the Skorokhod $J_1$ topology while the second coordinate converges for the topology of weak convergence of measures on $\R_+ \times \R_+^*\setminus \{1\} \times [0,1]$. The above convergence for the Skorokhod topology is written on an interval $[s_i^{(\ell)}, 0]$ whose length varies as $\ell$ varies but is simply defined as the convergence for the Skorokhod topology on $\mathbb{D}(\R_-, \R)$ where we extend the process on $(-\infty,s_i^{(\ell)})$ with its value at $s_i^{(\ell)}$. In particular, we get the convergence of the trajectories jointly with the previous ones: for all $w \in \U$,
\begin{equation}\label{eq cv trajectoires}
	\text{Under }\P^{(\ell)}, \qquad \qquad
	(Y^{(\ell),z,w})_{z\prec w} 
	\mathop{\longrightarrow}\limits_{\ell \to \infty}^{(\mathrm{d})}
	(Y^{z,w})_{z \prec w}.
\end{equation}
By Skorokhod's representation theorem, we may assume that the above convergences hold almost surely and we add a superscript $(\ell)$ to $\widetilde{P}_w$ and to the faces $f_w$ to distinguish them in the same probability space. Let $w\neq w' \in \U$. 

Let us first focus on the dense case $a\in (3/2,2)$, where by definition $c(w,w')$ is an element of $\U$. By Equations \eqref{eq distance dans l arbre des geosesiques fpp} and  \eqref{eq distance sur T}, it suffices to prove that almost surely for $\ell$ large enough,
	\begin{equation}\label{eq cv temps de coalescence}
	c^{(\ell)}(f^{(\ell)}_w,f^{(\ell)}_{w'}) = f^{(\ell)}_{c(w,w')}.
	\end{equation}
In order to check this, we use the convergence \eqref{eq cv flot w}. 

If $c(w,w')=\emptyset$, then the trajectories $(Y^{z,w})_{z \prec w} $ and $(Y^{z,w'})_{z\prec w'}$ do not coalesce. Then by \eqref{eq cv trajectoires} we know that for all $\ell$ large enough the trajectories $(Y^{(\ell),z,w})_{z \prec w} $ and $(Y^{(\ell),z,w'})_{z\prec w'}$ do not coalesce, hence \eqref{eq cv temps de coalescence}.

Otherwise, if $c(w,w')\neq \emptyset$, then one can write $c(w,w')= zj$ for some $z \in \U,j \in \N$. Then, by \eqref{eq cv flot w}, we have a.s. the convergence of the jump
\begin{equation}\label{eq cv saut coalescence}
\left(
\frac{\widetilde{P}^{(\ell)}_{z}(\tau_{zj}^{(\ell)})}{\widetilde{P}^{(\ell)}_{z}(\tau^{(\ell)}_{zj}-)}, U^{(\ell),z}_{\tau^{(\ell)}_{zj}}
\right)
\mathop{\longrightarrow}\limits_{\ell \to \infty}^{(\mathrm{a.s.})}
\left(Z^{z}_{-\tau_{zj}}, U^{z}_{-\tau_{zj}} \right).
\end{equation}
But by taking \eqref{eq cv trajectoires} into account, one obtains \eqref{eq cv temps de coalescence}. Indeed, by the convergence of the trajectories, we know that a.s.\@ for $\ell$ large enough, the two trajectories will not coalesce until the jump associated with the face $f^{(\ell)}_{c(w,w')}$. But at the same time, by the convergence of the trajectories and by the convergence of the jump, a.s.\@ for $\ell$ large enough they have to coalesce due to this jump. See Figure \ref{fig: coalescence geodesique saut}. Let us turn the idea on Figure \ref{fig: coalescence geodesique saut} into a rigourous proof. Let $T<-\tau_{zj}$. By \eqref{eq cv trajectoires}, for all $\ell \ge 1$, let $\psi_w^{(\ell)}$ and $\psi_{w'}^{(\ell)}$ be increasing homeomorphisms from $[T,0]$ to $[T,0]$ such that 
\begin{equation}\label{eq cv unif Y w w prime}
\sup_{t \in [T,0]} \left\vert Y^{(\ell), z, w}_{\psi_w^{(\ell)}(t)}  - Y^{z,w}_t\right\vert + \sup_{t \in [T,0]} \left\vert  \psi_w^{(\ell)}(t)-t\right\vert
+ \sup_{t \in [T,0]} \left\vert Y^{(\ell), z, w'}_{\psi_{w'}^{(\ell)}(t)}  - Y^{z,w'}_t\right\vert + \sup_{t \in [T,0]} \left\vert  \psi_{w'}^{(\ell)}(t)-t\right\vert
\mathop{\longrightarrow}\limits_{\ell \to \infty}^{(\mathrm{a.s.})} 0.
\end{equation}
Almost surely for $\ell$ large enough, we have
\begin{equation}\label{eq identite temps de coalescence}
\psi_w^{(\ell)}(-\tau_{zj}) = \psi_{w'}^{(\ell)}(-\tau_{zj}) = -\tau^{(\ell)}_{zj}.
\end{equation}
Indeed, by \eqref{eq cv unif Y w w prime},
$$
\left(Y^{(\ell), z, w}_{\psi_w^{(\ell)}(-\tau_{zj})},
Y^{(\ell), z, w}_{\psi_w^{(\ell)}(-\tau_{zj})-},
Y^{(\ell), z, w'}_{\psi_{w'}^{(\ell)}(-\tau_{zj})},
Y^{(\ell), z, w'}_{\psi_{w'}^{(\ell)}(-\tau_{zj})-}
\right)
\mathop{\longrightarrow}\limits_{\ell \to \infty}^{(\mathrm{a.s.})}
\left(
Y^{z,w}_{-\tau_{zj}},
Y^{z,w}_{-\tau_{zj}-},
Y^{z,w'}_{-\tau_{zj}},
Y^{z,w'}_{-\tau_{zj}-}
\right),
$$
and using that the jumps of $Y^{z,w}$ are a.s.\@ distinct given that the image of $\lambda(dz) du$ by $(z,u) \mapsto g(0,z,u)$ has no atom{s} (and the same holds for $Y^{z,w'}$), we deduce \eqref{eq identite temps de coalescence} from \eqref{eq cv saut coalescence}. Therefore,
$$
\left(Y^{(\ell), z, w}_{-\tau_{zj}^{(\ell)}},
Y^{(\ell), z, w}_{-\tau^{(\ell)}_{zj}-},
Y^{(\ell), z, w'}_{-\tau^{(\ell)}_{zj}},
Y^{(\ell), z, w'}_{-\tau^{(\ell)}_{zj}-}
\right)
\mathop{\longrightarrow}\limits_{\ell \to \infty}^{(\mathrm{a.s.})}
\left(
Y^{z,w}_{-\tau_{zj}},
Y^{z,w}_{-\tau_{zj}-},
Y^{z,w'}_{-\tau_{zj}},
Y^{z,w'}_{-\tau_{zj}-}
\right),
$$
But since $Y^{z,w}$ and $Y^{z,w'}$ coalesce, and since the law of $U^z_{-\tau_{zj}}$ has no atom{s},
$$
\{Y^{z,w}_{-\tau_{zj}-} -U^z_{-\tau_{zj}}\}< (Z^z_{-\tau_{zj}}-1)/Z^z_{-\tau_{zj}} \qquad \text{and}\qquad
\{Y^{z,w'}_{-\tau_{zj}-} -U^z_{-\tau_{zj}}\}< (Z^z_{-\tau_{zj}}-1)/Z^z_{-\tau_{zj}}.
$$
Thus, by \eqref{eq cv saut coalescence}, a.s.\@ for $\ell$ large enough,
$$
\{Y^{(\ell),z,w}_{-\tau^{(\ell)}_{zj}-} - U^{(\ell),z}_{\tau^{(\ell)}_{zj}}\}\vee \{Y^{(\ell),z,w'}_{-\tau^{(\ell)}_{zj}-} - U^{(\ell),z}_{\tau^{(\ell)}_{zj}}\}
\le \left(\frac{\widetilde{P}^{(\ell)}_{z}(\tau_{zj}^{(\ell)})}{\widetilde{P}^{(\ell)}_{z}(\tau^{(\ell)}_{zj}-)}-1\right)/\left(\frac{\widetilde{P}^{(\ell)}_{z}(\tau_{zj}^{(\ell)})}{\widetilde{P}^{(\ell)}_{z}(\tau^{(\ell)}_{zj}-)}\right).
$$
So, a.s.\@ for $\ell$ large enough, the trajectories $Y^{(\ell),z,w},Y^{(\ell), z, w'}$ coalesce at time $-\tau^{(\ell)}_{zj}$, hence \eqref{eq cv temps de coalescence}.

\begin{figure}[h]
	\centering
	\includegraphics[width=0.4\textwidth]{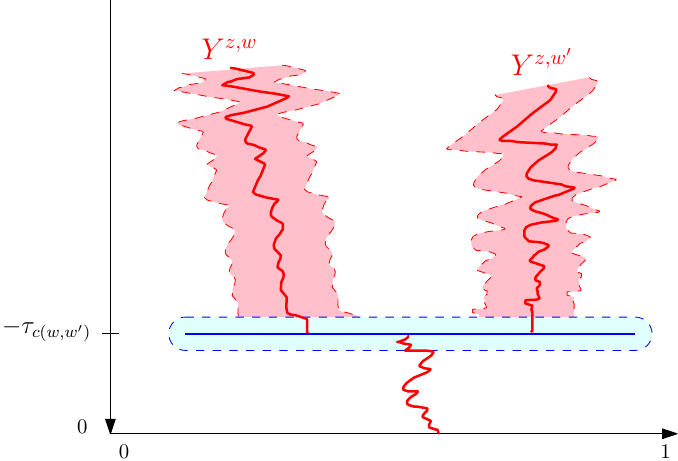}
	\caption{Sketch of the end of the proof of Theorem \ref{Th limite d echelle de l arbre} in the case $a \in (3/2,2)$: the two trajectories $Y^{z,w}$ and $Y^{z,w'}$ for some $z \prec w,w'$ in red coalesce due to the jump $\Delta_{c(w,w')}$ in blue. A neighbourhood of the trajectories is drawn in pink and a neighbourhood of the jump is in light blue.}
	\label{fig: coalescence geodesique saut}
\end{figure}

Next, we deal with the case $a \in [2,5/2)$. If $c(w,w')=\emptyset$, then we conclude as before. When $c(w,w') \neq \emptyset$, then by definition $c(w,w') = (\Delta_{zi_n})_{n\ge 0}\in \mathcal{J}_a$ for some common prefix $z\prec w,w'$, where $\tau_{zi_n} $ is an increasing sequence which converges towards the time $-\sigma_{w,w'}(z)$ which is the time at which the trajectories coalesce. Let $\vp>0$. Let $n$ large enough such that $\widetilde{t}_{zi_n} \ge \widetilde{t}_{c(w,w')}-\vp$. Then, by the same reasoning as in the case $a\in (3/2,2)$, by \eqref{eq cv trajectoires} we know that a.s.\@ for $\ell$ large enough, $f^{(\ell)}_{zi_n}$ is a common ancestor of $f^{(\ell)}_w,f^{(\ell)}_{w'}$. As a result,
$$
d_{\T(\mathfrak{m})} (f^{(\ell)}_w,f^{(\ell)}_{w'}) \le d^\dagger_{\mathrm{fpp}}(f^{(\ell)}_r,f^{(\ell)}_w)+ d^\dagger_{\mathrm{fpp}}(f^{(\ell)}_r,f^{(\ell)}_{w'})- 2d^\dagger_{\mathrm{fpp}}(f^{(\ell)}_r,f^{(\ell)}_{zi_n}).
$$
Consequently, a.s.
$$
\limsup_{\ell \to \infty} \frac{1}{\ell^{a-2}}d_{\T(\mathfrak{m})} (f^{(\ell)}_w, f^{(\ell)}_{w'})
\le
\frac{1}{2c_a p_{\bf q}}\left(\widetilde{t}_w + \widetilde{t}_{w'}- 2 \widetilde{t}_{zi_n}
 \right)
 \le \frac{1}{2 c_a p_{\bf q}} \left( d_{\T_a}(w,w') +2 \vp \right).
$$
One can then let $\vp \to 0$. Finally the almost sure inequality
$$
\liminf_{\ell \to \infty} \frac{1}{\ell^{a-2}} d_{\T(\mathfrak{m})} (f^{(\ell)}_w, f^{(\ell)}_{w'})\ge \frac{1}{2c_a p_{\bf q}} d_{\T_a}(w,w')
$$
comes directly form \eqref{eq cv trajectoires} since if at some time the trajectories $(Y^{z,w})_{z \prec w}$ and $(Y^{z,w'})_{z\prec w'}$ have not yet coalesced, then for $\ell$ large enough the same holds for the trajectories $(Y^{(\ell),z,w})_{z \prec w}$ and $(Y^{(\ell),z,w'})_{z\prec w'}$. This ends the proof.
\end{proof}

\section{Conjectures}\label{section conjectures}

Our choice of topology for Theorem \ref{Th limite d echelle de l arbre} comes from the fact that $\T_a$ is non-compact as soon as $a \le 2$. Still, in the dilute case $a \in (2, 5/2)$, we expect that the scaling limit of Theorem \ref{Th limite d echelle de l arbre} holds in the sense of Gromov-Hausdorff, after taking the completion of $\T_a$. We also expect a Gromov-Hausdorff-Prokhorov scaling limit in the dilute case when we equip the map with the uniform measure on the vertices/faces/edges.

Next, we state a few conjectures on the scaling limit of the map itself. We first define a random countable metric space which is conjecturally the scaling limit of large planar maps with high degrees. Let us define another distance by inserting shortcuts in the metric space $\T_a$ {which was defined in Subsection \ref{sous-section definition de la limite}}. To introduce our shortcuts, it will be more convenient to focus on the macroscopic positive jumps. For all $\vp \in (0,1)$, we define $(\T_a^\vp,d_{\T_a^\vp})$ as the subspace of $\T_a$ consisting of the elements $w\in \T_a$ such that $\Delta_w>\vp$. We now modify the distance $d_{\T_a^\vp}$ using the negative jumps of the $\xi_w$'s for $w\in \U$. 

Let $w\in \U\setminus \{\emptyset\}$. Let us denote by $\emptyset = w_0 \prec w_1 \prec \ldots \prec w_k  = w$ its prefixes. From the negative jump $x_{w}$ of $\mathcal{X}_{w_{k-1}}$, we define two trajectories: one starting on the left of the ``gap'' created by the jump, one starting from the right. Set $\mathcal{Y}^{w,w}_L(0)=0$ and $\mathcal{Y}^{w,w}_R(0)=1$. For all $i \in \{L,R\}$, for all $j \in \lb 1, k\rb$, we define inductively
\begin{align*}
	(\mathcal{Y}_i^{w_{j-1},w}(t))_{-\beta_{w_j}\le t \le 0} \coloneqq \left(X^{w_{j-1}}_{-\beta_{w_j},t}\left(
	R^{w_{j-1}}_{-\beta_{w_j}}
	-(1- Z^{w_{j-1}}_{-\beta_{w_j}}) + (1-Z^{w_{j-1}}_{-\beta_{w_j}})\mathcal{Y}_i^{w_j,w}(0)
	\right)\right)_{ -\beta_{w_j}\le t\le0}.
\end{align*}
For all $\vp \in (0,1)$, $i \in \{L,R\}$, let $J_i^{w,\vp}$ be the largest $j\in \lb 1, k \rb$ such that there exists $t\in [-\beta_{w_j},0]$ satisfying $\Delta \mathcal{X}_{w_{j-1}} (-t) >\vp$ and $
\left\{\mathcal{Y}_i^{w_{j-1},w}(t-) - U^{w_{j-1}}_t\right\} \le (Z^{w_{j-1}}_t -1)/{Z^{w_{j-1}}_t}.$
When it exists, we also define $V^{w,\vp}_i \in \T^\vp_a\setminus\{\emptyset\}$ as the index of the positive jump such that $-\tau_{V_i^{w,\vp}}$ is the smallest $t\in [-\beta_{w_j},0]$ to satisfy the above conditions. If $J_i^{w,\vp}$ does not exist, then we set $J^{w,\vp}_i=0$ and $V^{w,\vp}_i=\emptyset$. In other words, $V^{w,\vp}_L$ (resp. $V^{w,\vp}_R$) corresponds to the first coalescence corresponding to a positive jump of size larger than $\vp$ happening in the trajectory $(\mathcal{Y}_L^{w',w})_{w'\prec w}$ (resp. $(\mathcal{Y}_R^{w',w})_{w'\prec w}$). 

For all $w \in \U \setminus \{\emptyset\}$, we define the length $\mathcal{L}^\vp_w$ of the shortcut between $V^{w,\vp}_L$ and $V^{w,\vp}_R$ associated with the negative jump $x_w$ by
$
\mathcal{L}^\vp_w
=\widetilde{b}_w-\widetilde{t}_{V_L^{w,\vp}} +\widetilde{b}_w-\widetilde{t}_{V_R^{w,\vp}}
$. 
In the case $a=2$, the expression simplifies to $\mathcal{L}^\vp_w = b_w-t_{V_L^{w,\vp}} +b_w-t_{V_R^{w,\vp}}$. Moreover, for all $u,v \in \T_a^\vp$, we define the length of the smallest shortcut between $u$ and $v$ by setting
$$
\mathcal{L}^\vp(u,v) \coloneqq \inf \left\{\mathcal{L}^\vp_w ; \ w \in \U\setminus \{\emptyset\} \  \text{s.t.} \ (u,v)= (V^{w,\vp}_L, V^{w,\vp}_R) \ \text{or} \ (u,v) = (V^{w,\vp}_R,V^{w,\vp}_L) \right\}
\in [0,\infty].
$$
We set $\mathcal{D}^\vp \coloneqq \T_a^\vp$ and we equip $\mathcal{D}^\vp$ with the distance $d_{\mathcal{D}^\vp}$ defined for all $v,w \in \mathcal{D}^\vp$ by
\begin{align*}
	&d_{\mathcal{D}^\vp}(v,w)=\\
	&\inf\left\{ d_{\T_a}(v,v_1) + \sum_{k=1}^n \left( \mathcal{L}^\vp(v_k,w_k)+ d_{\T_a}(w_k,v_{k+1}) \right); \ 
	n\in \N, \
	v_1,\ldots, v_{n+1}, w_1,\ldots, w_{n} \in \mathcal{D}^\vp, \
	v_{n+1}=w
	\right\}.
\end{align*}
Note that by definition of the shortcuts we have $d_{\mathcal{D}^\vp}(v,w) \ge \vert d_{\T_a} (\emptyset,v)- d_{\T_a} (\emptyset, w) \vert$, so that positivity is preserved since $d_{\T_a} (\emptyset,v)\neq d_{\T_a} (\emptyset, w)$ as soon as $v \neq w$.
\begin{lemma}
	For all $\vp \in (0,1)$, $\vp' \in (0,\vp)$ and $v,w \in \mathcal{D}^\vp$, $
	d_{\mathcal{D}^{\vp'}}(v,w)\le d_{\mathcal{D}^{\vp}}(v,w).
	$
\end{lemma}
\begin{proof}
	It suffices to check that for all $w\in \U\setminus \{\emptyset\}$, we have $d_{\mathcal{D}^{\vp'}}(V^{w,\vp}_L, V^{w,\vp}_R) \le \mathcal{L}^\vp_w$. But 
	\begin{align*}
		d_{\mathcal{D}^{\vp'}}(V^{w,\vp}_L, V^{w,\vp}_R)
		&\le
		d_{\mathcal{D}^{\vp'}}(V^{w,\vp}_L, V^{w,\vp'}_L)
		+ \mathcal{L}^{\vp'}_w
		+ d_{\mathcal{D}^{\vp'}}(V^{w,\vp}_R, V^{w,\vp'}_R) 
		\\
		&=
		\big(\widetilde{t}_{V_L^{w,\vp'}}-\widetilde{t}_{V_L^{w,\vp}}\big)
		+
		\big(\widetilde{b}_w -\widetilde{t}_{V_L^{w,\vp'}}
		+ \widetilde{b}_w - \widetilde{t}_{V^{w,\vp'}_R}\big)
		+
		\big(\widetilde{t}_{V_R^{w,\vp'}}-\widetilde{t}_{V_R^{w,\vp}}\big)
		\\
		&=\mathcal{L}^\vp_w,
	\end{align*}
	where in the second line we used the fact that $V_L^{w,\vp}$ (resp.\@ $V_R^{w,\vp}$) is an ancestor of $V_L^{w,\vp'}$ (resp.\@ $V_R^{w,\vp'}$) in $\T_a$, in the sense that $c(V_L^{w,\vp}, V_L^{w,\vp'})= V_L^{w,\vp}$ (resp.\@ $c(V_R^{w,\vp}, V_R^{w,\vp'})= V_R^{w,\vp}$).
\end{proof}
Notice also that the sets $\mathcal{D}^\vp$ increase as $\vp$ decreases. One can thus define $\mathcal{D}_a= \mathcal{T}_a$ and the distance $d_{\mathcal{D}_a}$ on $\mathcal{D}_a$ by writing for all $v,w \in \mathcal{D}_a$,  
$$
d_{\mathcal{D}_a}(v,w) = \lim_{\vp \to 0} d_{\mathcal{D}^\vp}(v,w),
$$
where the term on the right is well defined for $\vp$ small enough. The function $d_{\mathcal{D}_a}$ satisfies the triangle inequality and symmetry since it is a limit of distances. Moreover, it also satisfies positivity since $d_{\mathcal{D}^\vp}(v,w) \ge \vert d_{\T_a} (\emptyset,v)- d_{\T_a} (\emptyset, w) \vert$.

\begin{conjecture}
	Suppose $a \in (3/2,5/2)$. Then, for the product topology,
	$$
	\text{Under } \P^{(\ell)}, \qquad \qquad
	\left(\ell^{2-a}d^\dagger_\mathrm{fpp}(f_i,f_j)\right)_{i,j \ge 1}
	\mathop{\longrightarrow}\limits_{\ell \to \infty}^{(\mathrm{d})}
	\left(\frac{1}{2c_a p_{\bf q}}d_{\mathcal{D}_a} (w_i,w_j)
	\right)_{i,j \ge 1}.
	$$
\end{conjecture}
\begin{conjecture}
	Suppose $a \in (2,5/2)$. Then, for the product topology,
	$$
	\text{Under } \P^{(\ell)}, \qquad \qquad
	\left(\ell^{2-a}d^\dagger_\mathrm{gr}(f_i,f_j)\right)_{i,j \ge 1}
	\mathop{\longrightarrow}\limits_{\ell \to \infty}^{(\mathrm{d})}
	\left(\frac{1+g_{\bf q}}{2c_a p_{\bf q}}d_{\mathcal{D}_a} (w_i,w_j)
	\right)_{i,j \ge 1},
	$$
	where $g_{\bf q} = (1/2) \sum_{k\ge 1} \nu(-k) (2k-1)$. Moreover, $\mathcal{D}_a$ has a compact completion and this convergence of metric spaces also holds in the sense of Gromov-Hausdorff for this completion of $\mathcal{D}_a$.
\end{conjecture}
\begin{conjecture}
	Suppose $a =2$. Then, for the product topology,
	$$
	\text{Under } \P^{(\ell)}, \qquad \qquad
	\left(\frac{1}{\log \ell} d^\dagger_\mathrm{gr}(f_i,f_j)\right)_{i,j \ge 1}
	\mathop{\longrightarrow}\limits_{\ell \to \infty}^{(\mathrm{d})}
	\left(\frac{1}{2c_a}d_{\mathcal{D}_a} (w_i, w_j)
	\right)_{i,j \ge 1}.
	$$
\end{conjecture}

\paragraph{Acknowledgements.} I am grateful to Cyril Marzouk and Nicolas Curien their insightful suggestions and comments. I also acknowledge the support of the ERC consolidator grant 101087572 ``SuPer-GRandMa''. I thank an anonymous referee for useful comments and corrections. 

\bibliographystyle{alpha}
\bibliography{distances_CLE4_nouveau}

\end{document}